\documentclass[10pt, english,oneside,reqno]{amsart}

\usepackage[utf8]{inputenc}
\usepackage[english]{babel}
\usepackage{hyperref}
\usepackage{color}
\usepackage{enumerate}
\usepackage{fancyhdr}

\usepackage{amsmath}
\usepackage{amscd}
\usepackage{amsfonts}
\usepackage{amsthm}
\usepackage{amstext}
\usepackage{amssymb}
\usepackage{amsbsy}
\usepackage{mathrsfs}
\usepackage{tikz-cd}
\usepackage{algorithm}
\usepackage{algpseudocode}
\usepackage{listings}

\usepackage{float}
\usepackage{caption}
\usepackage{subcaption}
\usepackage{graphicx}

\usepackage[alphabetic,abbrev]{amsrefs}

\usepackage{url}

\theoremstyle{plain}
\newtheorem{theorem}{Theorem}[section]

\newtheorem{definition}[theorem]{Definition}
\newtheorem{proposition}[theorem]{Proposition}
\newtheorem{corollary}[theorem]{Corollary}
\newtheorem{lemma}[theorem]{Lemma}

\theoremstyle{remark}

\numberwithin{equation}{section}
\numberwithin{figure}{section}

\newcommand{\eps}{\varepsilon}

\newcommand{\Z}{\mathbb{Z}}

\newcommand{\R}{\mathbb{R}}
\newcommand{\C}{\mathbb{C}}

\newcommand{\SL}{\mathrm{SL}}

\newcommand{\Ad}{\mathrm{Ad}}

\newcommand{\Schwartz}[1]{\mathscr{S}(#1)}

\newcommand{\DifferentialOps}[1]{\mathscr{D}(#1)}

\newcommand{\Haarof}[1]{m_{#1}}

\newcommand{\supp}[1]{\mathrm{supp}(#1)}
\newcommand{\SphericalPlancharelMeasure}{\nu_{\mathrm{sph}}}
\newcommand{\FurstenbergMeasure}{\nu_{\mathrm{F}}}
\newcommand{\FurstenbergDensity}{\psi_{\mathrm{F}}}

\begin{document}
	\title[Local Limit Theorem for Random Walks on Symmetric Spaces]{Local Limit Theorem for Random Walks on Symmetric Spaces}
	\author[C. Kogler]{Constantin Kogler}\thanks{The author gratefully acknowledges support from the European Research Council (ERC) grant No. 803711 as well as from the CCIMI at Cambridge}
	
	\address{Mathematical Institute \\ Radcliffe Observatory Quarter \\ Woodstock Road \\ Oxford OX2 6GG, United Kingdom} 
	\email{kogler@maths.ox.ac.uk}
	
	\begin{abstract}
		We reduce the local limit theorem for a non-compact semisimple Lie group acting on its symmetric space to establishing that a natural operator associated to the measure is quasicompact. Under strong Diophantine assumptions on the underlying measure, we deduce the necessary spectral results for the operator in question. We thereby give the first examples of finitely supported measures satisfying such a local limit theorem. Moreover, quantitative error rates for the local limit theorem are proved under additional assumptions. 
	\end{abstract}

\maketitle

\tableofcontents

\section{Introduction}

Let $G$ be a group and $\mu$ a probability measure on $G$. A fundamental problem in the theory of random walks is to describe the distribution  of the product of independent $\mu$-distributed random elements, in other words to study the measures $\mu^{*n}$. Local limit theorems, which establish the existence of a sequence $a_n \in \R$ such that $a_n\mu^{*n}$ converges to a limit measure, were studied by many authors. The case where $G$ is commutative or compact is classical (cf. for instance \cite{Stone1965}, \cite{ItoKawada1940}). Breuillard \cite{Breuillard2005GAFA} and Diaconis-Hough \cite{DiaconisHough2021} considered the Heisenberg group and a local limit theorem for the $\mathrm{Isom}(\R^d)$ action on $\R^d$ was proved by Varjú \cite{Varju2012}. For the latter case, under further assumptions on $\mu$, results with strong error terms were shown by Lindenstrauss-Varjú \cite{LindenstraussVarju2016}. The reader interested in discrete groups may consult Lalley's local limit theorem for the free group \cite{Lalley1993}, which was extended by Gouëzel \cite{Gouezel2014} to hyperbolic groups. 

The above results establish local limit theorems for the various mentioned settings under weak assumptions on $\mu$. In contrast, the understanding for non-compact semisimple Lie groups is less developed. The only case where a local limit theorem is known is by assuming that $\mu$ is spread out, i.e. a convolution power $\mu^{*n}$ for some $n \geq 1$ is not singular to the Haar measure. For spread out measures Bougerol \cite{Bougerol1981} proved in 1981 a local limit theorem that will be recalled in \eqref{Bougerol}. 

For a finitely supported measure whose support generates a dense subgroup, the convolutions $\mu^{*n}$ become increasingly well-distributed, more and more resembling a continuous measure. Therefore Bougerol's theorem is expected to hold. In this paper we give the first examples of finitely supported measures on semisimple Lie groups that satisfy Bougerol's theorem for the Lie group acting on the associated symmetric space. Indeed, we reduce the question at hand to understanding spectral properties of a natural operator $S_0 = S_0(\mu)$ associated to $\mu$. 

The operator $S_0$ may be viewed as the Fourier transform of the measure $\mu$ at $0$ and was studied by Bourgain \cite{Bourgain2012} in his construction of a finitely supported measure on $\SL_2(\R)$ with absolutely continuous Furstenberg measure. Further results on $S_0$ are due to \cite{BoutonnetIoanaSalehiGolsefidy2017}, generalizing \cite{Bourgain2012}, as well as \cite{BenoistQuint2018}. These results imply the necessary spectral properties for $S_0$ in order to establish local limit theorems and will be discussed after stating Theorem~\ref{StrongLCLT}. In certain cases, the necessary results for $S_0$ will also be proved in this paper following closely Bourgain's \cite{Bourgain2012} original ideas.

In addition, we deduce quantitative error rates for the local limit theorem (Theorem~\ref{LCLT} and Theorem~\ref{StrongLCLT}).

We proceed with stating Bougerol's theorem. Recall that a measure $\mu$ on $G$ is said to be non-degenerate whenever the semigroup generated by its support is dense in $G$. Let $G$ be a non-compact connected semisimple Lie group with finite center. For a probability measure $\mu$ on $G$, denote $\sigma = ||\lambda_G(\mu)||$, where $\lambda_G$ is the left regular representation and $\lambda_G(\mu) = \int \lambda_G(g) \, d\mu(g)$. Furthermore denote by $p$ the number of positive indivisible roots of $G$ and by $d$ the rank of $G$ (these notions are further discussed in Section~\ref{Notations}) and write $\ell = 2p + d$. For a non-degenerate and spread out probability measure $\mu$ with finite second moment (defined in \eqref{FiniteSecondMoment}), Bougerol \cite{Bougerol1981} showed that there is a continuous function $\psi_0$ on $G$ (depending on $\mu$) such that 
\begin{equation}\label{Bougerol}
	\lim_{n \to \infty} \frac{n^{\ell/2}}{\sigma^n} \int f(g) \, d\mu^{*n}(g) = \int f(g) \psi_0(g) \, d\Haarof{G}(g)
\end{equation}  for all $f \in C_c^{\infty}(G)$. The function $\psi_0$ satisfies $\mu * \psi_0 =  \psi_0 * \mu = \sigma \psi_0$.  

To introduce further notation, let $K$ be a maximal compact subgroup of $G$ and denote by $X = G/K$ the associated symmetric space. We recall the definition of the Furstenberg boundary. Let $G = KAN$ be an Iwasawa decomposition of $G$ as introduced in Section~\ref{Notations}. Let $M$ be the centralizer of $A$ in $K$ and write $P = MAN$. The Furstenberg boundary of $G$ is defined as $\Omega = G/P = K/M$.  The measure $\Haarof{\Omega}$ is the pushforward of the Haar probability measure $\Haarof{K}$ onto $\Omega$.

Denote by $\rho_0$ the Koopman unitary representation of the $G$ action on the measure space $(\Omega,\Haarof{\Omega})$, which is also called the $0$-principal series representation (see Section~\ref{Notations}). For a probability measure $\mu$ on $G$,  consider the operator $S_0 = \rho_0(\mu) = \int \rho_0(g) \, d\mu(g)$. In order to state the first theorem, recall that a bounded operator is called quasicompact if the essential spectral radius $\rho_{\mathrm{ess}}(A)$ (defined in \eqref{DefinitionEssentalSpectralRadius}) is strictly less than the spectral radius. 

Let $\mathfrak{a} = \mathrm{Lie}(A)$ and choose a closed Weyl chamber $\mathfrak{a}^{+}$. Then for every $g \in G$ denote by $\kappa(g) \in \mathfrak{a}^{+}$ the unique element such that $g \in K\exp(\kappa(g))K$. We say that $\mu$ has finite $k$-th moment for some $k \geq 1$ if 
\begin{equation}\label{FiniteSecondMoment}
	\int |\kappa(g)|^k \, d\mu(g) < \infty.
\end{equation}

\begin{theorem}(Local limit theorem)\label{QuasicompactnessImpliesLLT} Let $G$ be a non-compact connected semisimple Lie group with finite center. Choose a maximal compact subgroup $K$ and denote $X = G/K$. Let $\mu$ be a non-degenerate probability measure on $G$ with finite second moment and assume that $S_0 = \rho_0(\mu)$ is quasicompact. Write $\sigma = ||\lambda_G(\mu)|| = ||S_0||$ and $\ell = 2p + d$ for $p$ the number of indivisible positive roots of $G$ and $d$ the rank of $G$. 
	
	Then there is a continuous real-valued function $\psi_{0}$ on  $G$ satisfying $\mu* \psi_0 = \psi_0 * \mu = \sigma \psi_{0}$ such that for $x_0  \in X$ and $f \in C_c^{\infty}(X)$, 
	\begin{equation}\label{LLTSymmetricSpace}
		\lim_{n \to \infty} \frac{n^{\ell/2}}{\sigma^n} \int f(g.x_0) \, d\mu^{*n}(g) = \int f(g.x_0)\psi_{0}(g) \, d\Haarof{G}(g).
	\end{equation}
	Moreover, the operator $S_0$ has a unique $\sigma$-eigenfunction $\eta_0 \in L^2(\Omega)$ of unit norm and there exists a unique $\sigma$-eigenfunction  $\eta_0'$ of $S_0^{*}$ satisfying $\langle \eta_0, \eta_0' \rangle = 1$. Then $\eta_0$ and $\eta_0'$ are positive almost surely and $\psi_0$ is given as $\psi_0(g) = c_{\mu}\cdot \langle \eta_0,\rho_0(g)\eta_0' \rangle$ for $c_{\mu} > 0$ a constant depending on $\mu$. 
\end{theorem}

The only difference between \eqref{Bougerol} and \eqref{LLTSymmetricSpace} is that the latter is only proved on $X$. Indeed, the limit function of Bougerol's theorem arises as in Theorem~\ref{QuasicompactnessImpliesLLT} and since a non-degenerate, spread out measure $\mu$ satisfies that $S_0$ is quasicompact (cf. Proposition 2.2.1 of \cite{Bougerol1981}), Theorem~\ref{QuasicompactnessImpliesLLT}  is a generalization of Bougerol's theorem on $X$. We furthermore mention that it is conjectured that \eqref{Bougerol} and therefore also \eqref{LLTSymmetricSpace} holds for every non-degenerate probability measure (with finite second moment) on $G$.

Having stated Theorem~\ref{QuasicompactnessImpliesLLT}, the question arises to give quantitative error rates for \eqref{LLTSymmetricSpace}. Towards this aim and in order to motivate Theorem~\ref{LCLT}, we discuss $G = \R$.  Let $\mu$ be a non-degenerate measure on $\R$ with mean zero and variance $\sigma^2 < \infty$. The local limit theorem on $\R$ (cf. \cite{BreimanProbabilityBook} Section 7.4) states that $\sqrt{n}\mu^{*n} \to \frac{\Haarof{\R}}{\sqrt{2\pi \sigma^2}}$. 

Denote $$\eta_{n}(x) = \frac{1}{\sqrt{2\pi \sigma^2}}\exp\left(-\frac{x^2}{2n\sigma^2}\right).$$
Using that $|\widehat{\mu}(r)| < 1$ for $r \neq 0$ and $\widehat{\mu}(r) = \int e^{irx} \, d\mu(x)$ the Fourier transform of $\mu$, one can show for $f \in C^{\infty}(\R)$ a smooth function whose Fourier transform is compactly supported that there is a constant $c_f = c_{f}(\mu)$ depending on $\mu$ and the support of $\widehat{f}$ such that
\begin{equation}\label{RealGeneralQuantLCLT}
	\sqrt{n}\mu^{*n}(f) = \int f(x) \eta_{n}(x) \, d\Haarof{\R}(x) + \left(O_{\mu}(n^{-1})  + O_{\mu, f}(e^{-c_fn})\right)||f||_1,
\end{equation} where the first implied constant depends on $\mu$ and the second on $\mu$ and  the support of $\widehat{f}$. The result \eqref{RealGeneralQuantLCLT} may be referred to as the local central limit theorem as it implies the local limit theorem as well as the central limit theorem. Using that $|\tfrac{1}{\sqrt{2\pi\sigma^2}} - \eta_{n}(x)| \ll_{\sigma} n^{-1}x^2$, it follows that
\begin{equation}\label{RealLCLT}
	\sqrt{n}\mu^{*n}(f) = \frac{1}{\sqrt{2 \pi \sigma^2}}\int f(x) \, d\Haarof{\R}(x) + O_{\mu}(n^{-1}||f||_{*}) + O_{\mu,f}(e^{-c_f n} ||f||_1)
\end{equation}
for $$||f||_{*} = \int |f(x)|(1 + x^2) \, d\Haarof{\R}(x).$$

We deduce the same behaviour as \eqref{RealLCLT} even with matching error terms for the $G$ action on its symmetric space under the assumption that $S_0$ is quasicompact. Choosing a maximal compact subgroup $K$ corresponds to fixing the origin $o = eK \in X$ of $X$. Denote by $d_X(\cdot,\cdot)$ the distance function induced by a Riemannian metric on $X$ (for which $X$ is a symmetric space, see~\eqref{SymmetricSpaceMetric}). In the theorem below we refer to the Fourier transform of a function $f \in C^{\infty}(X)$ as discussed in Section~\ref{Notations}. For the asymptotic notation used see also Section~\ref{Notations}.

\begin{theorem}\label{LCLT}(Local limit theorem with weak quantitative error rates) With the notation and assumptions from Theorem~\ref{QuasicompactnessImpliesLLT}, assume further that $\mu$ has finite fourth moment. Then for $f \in C^{\infty}(X)$ with compactly supported Fourier transform, there is a constant $c_f = c_f(\mu) > 0$ depending on $\mu$ and the support of $\widehat{f}$ such that for $n \geq 1$ and all $x_0 \in X$,
	\begin{align}\label{LCLTQuantitativeFourierCompact}
		\frac{n^{\ell/2}}{\sigma^n} \int f(g.x_0) \, d\mu^{*n}(g) &= \int f(g.x_0)\psi_{0}(g) \, d\Haarof{G}(g) \\ &+ O_{\mu}(n^{-1} ||f||_* + n^{-1}d_X(x_0,o)^2||f||_1)  + O_{\mu, f}(e^{-c_fn} ||f||_1), \nonumber
	\end{align} where the first implied constant depends on $\mu$, the second on $\mu$ and the support of $\widehat{f}$ and 
	\begin{equation}\label{StarNormDef}
		||f||_* = \int |f(x)| (1 + d_X(x,o)^2) \, d\Haarof{X}(x).
	\end{equation}
\end{theorem}

For $G = \R$, it is only possible to give strong error rates for \eqref{RealLCLT} if one gains control over the behaviour of the function $|\widehat{\mu}(r)|$ as $r \to \infty$, which as is shown in \cite{Breuillard2005ProbRel} is equivalent to assuming certain Diophantine properties on the support of $\mu$. 

In similar vein, we give strong error rates for \eqref{LCLTQuantitativeFourierCompact} under a suitable Fourier decay assumption. The Schwartz space $\mathscr{S}(X)$ of below theorem is defined in Section~\ref{Notations}. For $r \in \mathfrak{a}^{*}$ denote by $\rho_r$ the $r$-principal series representation defined in \eqref{PrincipalSeries} and write $$S_r = \rho_r(\mu).$$ 

\begin{theorem}\label{StrongLCLT}(Local limit theorem with strong quantitative error rates) With the notation and assumptions from Theorem~\ref{QuasicompactnessImpliesLLT}, assume further that $\mu$ has finite fourth moment and that 
	\begin{equation}\label{StrongLCLTAssumption}
		\sup_{|r| \geq 1} ||S_r||  < ||S_0||.
	\end{equation} 
	Then for $f \in\Schwartz{X}$, $x_0 \in X$ and $n \geq 1$, 
	\begin{align}\label{StrongLCLTFormula}
		\frac{n^{\ell/2}}{\sigma^n}\int f(g.x_0) \, d\mu^{*n}(g) &= \int f(g.x_0) \psi_{0}(g) \, d\Haarof{G}(g) \\ &+ O_{\mu}(n^{-1}||f||_* + n^{-1}d_X(x_0,o)^2||f||_1 + e^{-cn}||f||_{H^s}), \nonumber
	\end{align}
	where $c = c(\mu)$ is a constant depending on $\mu$, $s = \frac{1}{2}(\dim X + 1)$, $||\cdot||_{H^s}$ is the Sobolev norm \eqref{SobolevSpaceSymmetricSpace} of degree $s$ and the implied constant depends only on $\mu$. Moreover, the assumption \eqref{StrongLCLTAssumption} holds whenever $\mu$ is spread out or bi-$K$-invariant (i.e. $\mu = \Haarof{K} * \mu * \Haarof{K}$). 
\end{theorem}

We proceed with discussing spectral properties of the operator $S_0$ and also related results on absolute continuity of the Furstenberg measure. In order to introduce convenient notation, we recall the definition of almost Diophantine measures introduced in \cite{BenoistdeSaxce2016}.

\begin{definition}
	Let $G$ be a connected Lie group, $\mu$ a probability measure on $G$ and let $c_1, c_2 > 0$. The measure $\mu$ is called \textbf{$(c_1, c_2)$-almost Diophantine} or simply \textbf{$(c_1,c_2)$-Diophantine} if $$\sup_{H < G} \mu^{*n}(B_{e^{-c_1 n}}(H)) \leq e^{-c_2 n}$$ for sufficiently large $n$, where $B_{e^{-c_1 n}}(H) = \{ g \in G \,:\, d(g,H) < e^{-c_1 n} \}$ and the supremum is taken over all closed connected subgroups $H$ of $G$.
\end{definition}

Almost Diophantine measures are useful in understanding random walks on compact groups. Generalizing the Bourgain-Gamburd method developed for $\mathrm{SU}(2)$ by \cite{BourgainGamburd2008Invent} and for $\mathrm{SU}(d)$ in \cite{BourgainGamburd2010}, it was shown in  \cite{BenoistdeSaxce2016} for $K$ a compact connected simple Lie group, that a symmetric measure $\mu$ is $(c_1,c_2)$-Diophantine for some $c_1, c_2 > 0$ if and only if $\lambda_K(\mu)$ has strong spectral gap (Definition~\ref{StrongSpectralGapDef}), in this setting being equivalent to $||\lambda_K(\mu)|_{L^2_0(K)}||_{\mathrm{op}} < 1$ for $L^2_0(K) = \{ f \in L^2(K) \,:\, \Haarof{K}(f) = 0 \}$. Indeed, the essential spectral radius of $\lambda_K(\mu)$ can be bounded in terms of $K,c_1$ and $c_2$. Strong spectral gap of $\lambda_K(\mu)$ can be used to deduce by using the Fourier inversion formula on $K$ that for $f \in C^{\infty}(K)$
\begin{equation}\label{CompactGroupsQuantitativeLLT}
	|\mu^{*n}(f) - \Haarof{K}(f)| \ll e^{-cn}||f||_{H^s}
\end{equation}
with $c > 0$ a constant depending on $\mu$ and $||\cdot||_{H^s}$ a Sobolev norm \eqref{CompactGroupSobolevSpace} on $K$ of high enough degree. Without assuming that $\mu$ is almost Diophantine, only weaker results than \eqref{CompactGroupsQuantitativeLLT} are known. Nonetheless, it is conjectured that every non-degenerate measure is almost Diophantine.  For finitely supported measures it is established in \cite{BenoistdeSaxce2016} that non-degenerate symmetric measures with matrices supported on algebraic entries are almost Diophantine. 

For finitely supported measures, most known spectral results for $S_0$ also rely on the Bourgain-Gamburd method. However one requires stronger Diophantine conditions. Indeed, as in contrast to compact groups it is necessary to control the exponential norm growth of the $\mu$-random walk on $G$, we have to demand that the measure is $(c_1, c_2)$-Diophantine while being close to the identity in terms of $c_1$ and $c_2$. We therefore introduce the following definition. 

\begin{definition}
	Let $G$ be a connected Lie group, $\mu$ a probability measure on $G$ and let $c_1, c_2, \varepsilon > 0$. The measure $\mu$ is called \textbf{$(c_1, c_2, \varepsilon)$-Diophantine} if 
	\begin{enumerate}[(i)]
		\item $\mu$ is $(c_1 \log \frac{1}{\varepsilon}, c_2 \log\frac{1}{\varepsilon})$-Diophantine, i.e. for $n$ large enough, $$\sup_{H < G} \mu^{*n}(B_{\varepsilon^{c_1 n}}(H)) \leq \varepsilon^{c_2 n}.$$
		\item $\supp \mu \subset B_{\varepsilon}(e)$.
	\end{enumerate}
\end{definition}

We state a result of \cite{BoutonnetIoanaSalehiGolsefidy2017} showing that there is an abundant collection of examples of $(c_1, c_2, \varepsilon)$- Diophantine measures for arbitrarily small $\varepsilon$.

\begin{theorem}(Theorem 3.1 of \cite{BoutonnetIoanaSalehiGolsefidy2017})\label{BIG17Theorem3.1}
	Let $G$ be a connected simple Lie group with finite center and adjoint representation $\mathrm{Ad} : G \to \mathrm{GL}(\mathfrak{g})$. Let $\Gamma < G$ be a countable dense subgroup and assume that there is a basis of $\mathfrak{g}$ such that $\mathrm{Ad}(\gamma)$ is algebraic with respect to that basis for every $\gamma \in \Gamma$.
	
	Then there exist $c_1, c_2 > 0$ such that for every $\varepsilon_0 > 0$ there is $0 < \varepsilon < \varepsilon_0$ and a finitely supported symmetric $(c_1, c_2, \varepsilon)$-Diophantine probability measure $\mu$ satisfying $\supp{\mu} \subset \Gamma \cap B_{\varepsilon}$.
\end{theorem}

Using the above defined notion of Diophantine measures, one can establish the following result on quasicompactness of $S_0$. Together with Theorem~\ref{BIG17Theorem3.1}, numerous examples of finitely supported measures satisfying \eqref{LLTSymmetricSpace} are provided.

\begin{theorem}\label{FurstenbergMeasureLLTMainTheorem}
	Let $G$ be a non-compact connected simple Lie group with finite center. Let $c_1, c_2 > 0$. Then there is $\varepsilon_0 = \varepsilon_0(G,c_1,c_2) > 0$ depending on $G$ and $c_1,c_2 > 0$, such that every symmetric and $(c_1, c_2, \varepsilon)$-Diophantine probability measure $\mu$ with $\varepsilon \leq \varepsilon_0$ satisfies that $S_0 = \rho_0(\mu)$ is quasicompact.  In particular, Theorem~\ref{QuasicompactnessImpliesLLT} and Theorem~\ref{LCLT} holds for $\mu$.
\end{theorem}

Theorem~\ref{FurstenbergMeasureLLTMainTheorem} is a straightforward consequence of the techniques and results developed in \cite{BoutonnetIoanaSalehiGolsefidy2017} and will be deduced in section~\ref{BIGConsequence}. Under the additional assumption that the maximal compact subgroup is semisimple, we offer an alternative proof following more closely the method by Bourgain \cite{Bourgain2012}, leading to marginally stronger results (Theorem~\ref{BourgainMainResult}). Indeed, using an idea from \cite{LindenstraussVarju2016}, we simplify Bourgain's original approach by exploiting that the irreducible representations of $K$ have high dimension.

We proceed with discussing the Furstenberg measure. Let $\mu$ be a measure on $G$ whose support generates a Zariski dense subgroup. Then the Furstenberg measure of $\mu$ is the unique $\mu$-stationary Borel probability measure $\FurstenbergMeasure$ on the boundary $\Omega$ (cf. for example \cite{GoldsheidMargulis1989}). It was initially conjectured by Kaimanovich-Le Prince \cite{KaimanovichLePrince2011} that the Furstenberg measure of a finitely supported measure is singular to the Haar measure $\Haarof{\Omega}$. However Bourgain \cite{Bourgain2012} and B\'{a}r\'{a}ny-Pollicott-Simon \cite{BaranyPollicottSimon2012} disproved the latter conjecture, with Bourgain \cite{Bourgain2012} giving an explicit construction while \cite{BaranyPollicottSimon2012} exploiting probabilistic methods. 

\cite{BenoistQuint2018} also provide examples of finitely supported measures with absolutely continuous Furstenberg measure, yet their construction does not lead to results as versatile as Theorem~\ref{BIG17Theorem3.1}. It is apparent from their proof, that $S_0$ is also quasicompact for these examples. 

A further result of \cite{Bourgain2012} is the construction of finitely supported measures on $\SL_2(\R)$ satisfying $\frac{d\FurstenbergMeasure}{d\Haarof{\Omega}} \in C^k(\Omega)$ for any $k \in \Z_{\geq 1}$. Following Bourgain's technique, we also deduce smoothness results for the Furstenberg measure for arbitrary simple Lie groups.

\begin{theorem}\label{FurstenbergMeasureSmoothness}
	Let $G$ be a non-compact connected simple Lie group with finite center. Let $c_1, c_2 > 0$ and $m \in \Z_{\geq 1}$. Then there is $\eps_m = \eps_m(G,c_1,c_2) > 0$ depending on $G,c_1,c_2$ and $m$ such that every symmetric and $(c_1,c_2, \varepsilon)$-Diophantine probability measure $\mu$ with $\varepsilon \leq \varepsilon_m$ has absolutely continuous Furstenberg measure with density in $C^m(\Omega)$.
\end{theorem}

While writing this paper, the author became aware of \cite{Lequen2022} who establishes a similar yet less general result to Theorem~\ref{FurstenbergMeasureSmoothness}. Since our proof is short and differs from \cite{Lequen2022} for instance in introducing Agmon's inequality (Lemma~\ref{AgmonsInequality}) for compact Lie groups, it is included in this paper.  

We comment on the organization of this paper. After reviewing the necessary notation and giving an outline of proofs in Section~\ref{Notation}, we discuss some preliminary results in Section~\ref{SectionPreliminary}. Then the local limit theorems Theorem~\ref{QuasicompactnessImpliesLLT}, Theorem~\ref{LCLT} and Theorem~\ref{StrongLCLT} are proved in Section~\ref{SectionProofLocalLimitTheorem}. Finally, quasicompactness of $S_0$ and the Furstenberg measure are discussed in Section~\ref{SectionQuasicompactnessS0}, establishing Theorem~\ref{FurstenbergMeasureLLTMainTheorem} and Theorem~\ref{FurstenbergMeasureSmoothness}.

\subsection*{Acknowledgment} I am grateful to my advisors Emmanuel Breuillard and Péter Varjú for introducing me to this topic and for their patient and thoughtful guidance. Furthermore, I am indebted to Timothée Bénard, Emmanuel Breuillard, Amitay Kamber and Péter Varjú and the anonymous referee for comments on a preliminary draft. The author gratefully acknowledges support from the European Research Council (ERC) grant No. 803711 as well as from the CCIMI at Cambridge. This work is part of a PhD thesis conducted at the University of Cambridge and the University of Oxford. 

\section{Notation and Outline} \label{Notation}

\subsection{Notation} \label{Notations}

In this section we collect the notations used in this paper. 

Throughout this paper, we denote by $G$ a non-compact connected semisimple Lie group with finite center, by $K$ a maximal compact subgroup of $G$ and write $X = G/K$ for the associated symmetric space. 

We use the asymptotic notation $X \ll Y$ or $X = O(Y)$ to denote that $|X| \leq CY$ for a constant $C > 0$ and for a sequences $X_n$ and $Y_n$ we write $X_n = o(Y_n)$ to symbolize $|\frac{X_n}{Y_n}| \to 0$ as $n \to \infty$. If the constant $C$ or the speed of convergence depends on additional parameters we add subscripts, unless the quantity depends on the fixed group $G$ in which case we omit addition subscripts for convenience.  

Let $\mathscr{B}$ be a Banach space and let $A: \mathscr{B} \to \mathscr{B}$ be a bounded operator. Recall that $A$ is called a Fredholm operator if there exists a bounded operator $T$ such that $TA - \mathrm{Id}$ and $AT - \mathrm{Id}$ are compact operators. Denote by $\mathrm{spec}(A)$ the spectrum of $A$. The essential spectrum $\mathrm{spec}_{\mathrm{ess}}(A)$ is defined as the set of complex numbers $\lambda$ such that $A - \lambda\cdot \mathrm{Id}$ is not Fredholm. The spectral radius is defined as $\rho(A) = \max_{\lambda \in \mathrm{spec}(A)} |\lambda|$ and the essential spectral radius as 
\begin{equation}\label{DefinitionEssentalSpectralRadius}
	\rho_{\mathrm{ess}}(A) = \max_{\lambda \in \mathrm{spec}_{\mathrm{ess}}(A)} |\lambda|,
\end{equation}  if $\rho_{\mathrm{ess}}(A) \neq \emptyset$ and otherwise $\rho_{\mathrm{ess}}(A) = 0$.

For a locally compact Hausdorff group $H$, write $\Haarof{H}$ for a fixed choice of Haar measure. Whenever $H$ is compact, $\Haarof{H}$ is the Haar probability measure. The left-regular representation is denoted $\lambda_H$ while we write $\rho_H$ for the right regular representation.

If $\mu$ is a finite measure on $H$ and $\pi : H \to \mathscr{U}(\mathscr{H})$ is a unitary representation, where $\mathscr{H}$ is a Hilbert space and $\mathscr{U}(\mathscr{H})$ the space of unitary operators $\mathscr{H} \to \mathscr{H}$, then 
\begin{equation}\label{pimudef}
	\pi(\mu) = \int \pi_g \, d\mu(g)
\end{equation}
is the operator uniquely characterized by $\langle \pi(\mu) v,w \rangle = \int \langle \pi_g v,w \rangle \, d\mu(g)$ for $v,w \in \mathscr{H}$.

For a group $H$ with metric $d_H$, for $R > 0$ and $x \in H$ we will denote by $B_{R}(x) = \{ y \in H \,:\, d_H(y,x) < R \}$ and abbreviate $B_R = B_R(e)$ for $e \in H$ the identity element. On $G$ we fix a left invariant metric such that $B_R(g) = gB_R(e)$. For a closed subset $H' \subset H$ we define $B_{R}(H') = \{ h \in H \,:\, d(h,H') < R \}$, where $d(h,H') = \sup_{h' \in H'} d(h,h')$. 

We first fix notation for structure theory on $K$. Write $T$ for a maximal torus in $K$ with Lie algebra $\mathfrak{t}$ and real dual Lie algebra $\mathfrak{t}^{*}$. Let $W_K$ be the Weyl group and we fix a $W_K$-invariant inner product on $\mathfrak{t}$, inducing an $W_K$-invariant inner product on $\mathfrak{t}^{*}$.  The set of real roots is denoted as $R$ and we choose a fundamental Weyl chamber $C$ which we consider as a subset of $\mathfrak{t}^{*}$. The fundamental Weyl chamber determines a basis $S$ of the real roots and the set of positive roots $R^{+}$. We denote by $I^{*} \subset \mathfrak{t}^{*}$ the set of integral forms. Then (cf. \cite{BrocknertomDieckRepresentationBook} Section 6) the set $\overline{C} \cap I^{*}$ parametrizes the irreducible representations of $K$. 

For $\gamma \in \overline{C} \cap I^{*}$ denote by $\pi_{\gamma}$ the associated irreducible unitary representation of $K$ and by $M_{\gamma}$ the span of matrix coefficients of $\pi_{\gamma}$. By the Peter-Weyl Theorem it holds that 
\begin{equation}\label{PeterWeyl}
	L^2(K) = \bigoplus_{\gamma \in \overline{C} \cap I^{*}} M_{\gamma},
\end{equation} where we used the convention applied throughout this paper that by a direct sum we denote the closure of the algebraic direct sum of the involved vector spaces. For any $\gamma \in \overline{C} \cap I^{*}$ and an orthonormal basis $v_1, \ldots , v_{d_{\gamma}}$ of $\pi_{\gamma}$, we set $\chi_{ij}^{\gamma}(k) = \langle \pi_{\gamma}(k) v_i, v_j \rangle$. Then the set of functions $d_{\gamma}^{1/2}\chi_{ij}^{\gamma}$ forms an orthonormal basis of $L^2(K)$. For $\varphi \in L^2(K)$, we set $\widehat{\varphi}^{\gamma}_{ij} = a_{ij}^{\gamma} = \langle \varphi, d_{\gamma}^{1/2}\chi_{ij}^{\gamma}\rangle$. For $\varphi \in C^{\infty}(K)$ and all $k \in K$, 
\begin{equation}\label{CompactGroupFourierInversion}
	\varphi(k) = \sum_{\gamma \in \overline{C}\cap I^{*}} \sum_{i,j = 1}^{d_{\gamma}} d_{\gamma}^{1/2}a_{ij}^{\gamma} \chi_{ij}^{\gamma}(k).
\end{equation}

We want to group together functions on $K$ that oscillate at roughly the same rate. Therefore, one defines 
\begin{equation}\label{LittlewoodPayleySpaces}
	V_0 = \bigoplus_{\substack{\gamma \in \overline{C} \cap I^{*}  \\ 0 \leq ||\gamma|| < 1}} M_{\gamma} \quad\quad\quad\quad\text{and}\quad\quad\quad\quad  V_{\ell} = \bigoplus_{\substack{\gamma \in \overline{C} \cap I^{*}  \\ 2^{\ell - 1} \leq ||\gamma|| < 2^{\ell}}} M_{\gamma}
\end{equation}
for $\ell \geq 1$. The decomposition
\begin{equation}\label{LittlewoodPayleyDecomposition}
	L^2(K) = \bigoplus_{\ell \geq 0} V_{\ell}
\end{equation} is referred to as the Littlewood-Paley decomposition of $L^2(K)$. For $\ell \geq 0$ we denote by $P_{\ell}$ the orthogonal projection from $L^2(K)$ to $V_{\ell}$. Therefore any $\varphi \in L^2(K)$ can be decomposed as $\varphi = \sum_{\ell \geq 0} P_{\ell} \varphi$. For Littlewood-Paley decompositions on groups in more general contexts we refer the reader to \cite{Mallahi-KaraiMohammdiSalehiGolsefidy2022}.

We finally define Sobolev spaces and Sobolev norms on $K$. Denote by $\mathfrak{k}$ the Lie algebra of $K$ and fix an orthonormal basis $X_1, \ldots , X_n$ of $\mathfrak{k}$. Then the Casimir operator given by $\triangle = - \sum_{i = 1}^n X_i \circ X_i$ is a central element of the universal enveloping algebra $U(\mathfrak{k})$. For $\gamma \in \overline{C} \cap I^{*}$ denote by $\lambda_{\gamma}$ the eigenvalue of $\triangle$ acting on $\pi_{\gamma}$. For $s \in \Z_{\geq 0}$, we define 
\begin{align}
	H^s(K) &= \{ \varphi \in L^2(K) \,:\, \lambda_K(\triangle)^{s/2} \varphi \in L^2(K) \} \label{CompactGroupSobolevSpace} \\ &= \left\{  \varphi = \sum_{\gamma \in \overline{C} \cap I^{*}} \varphi_{\gamma} \in \bigoplus_{\gamma \in \overline{C} \cap I^{*}} M_{\gamma} \,:\, ||\varphi||_{H^s}^2 = \sum_{\gamma \in \overline{C} \cap I^{*}}  \lambda_{\gamma}^s ||\varphi_{\gamma}||_2^2 < \infty   \right\}. \nonumber 
\end{align}

We also need structure theory for $G$. We take care not to confuse the notation introduced for the structure theory of $K$. The Lie algebra of $G$ is denoted as $\mathfrak{g}$ and we choose a Cartan decomposition $\mathfrak{g} = \mathfrak{k} \oplus \mathfrak{a}\oplus \mathfrak{n}$ for $\mathfrak{k}$ the Lie algebra of $K$. Denote by $\mathfrak{a}^{*}$ the real dual of $\mathfrak{a}$. Let $\Sigma$ be the sets of roots, choose a closed Weyl chamber $\mathfrak{a}^{+}$ and let $\Sigma^{+} = \{ r_1, \ldots , r_k \} \subset \mathfrak{a}^{*}$ be the system of positive roots. For a root $r \in \Sigma$ write $m(r)$ for the multiplicity of $r$ and denote by $\delta = \frac{1}{2} \sum_{r \in \Sigma^{+}} m(r)r$ the half sum of the positive roots counted with multiplicities. We fix a norm $|\cdot|$ on $\mathfrak{g}$ arising from an $\Ad$-invariant inner product. The latter norm restricts to $\mathfrak{a}$ and induces the operator norm on $\mathfrak{a}^{*}$.

Denote $A = \exp(\mathfrak{a})$, $N = \exp(\mathfrak{n})$ and $P^{+} = AN$. Then (cf. \cite{KnappLieBook} chapter VI) the multiplication map $K \times A \times N \to G$ is a diffeomorphism, giving rise to the Iwasawa decomposition $G = KAN$. Write further $K: G \to K$, $A: G \to A$ and $N : G \to N$ for the maps induced from the Iwasawa decomposition and the map $H: G \to \mathfrak{a}$ is defined for $g \in G$ as 
\begin{equation}\label{DefH}
	H(g) = \log A(g).
\end{equation}

Set $A^{+} = \exp(\mathfrak{a}^{+})$. Then the Cartan decomposition $G = KA^{+}K$ holds and denote by $\kappa : G \to \mathfrak{a}^{+}$ the map uniquely characterized by $g \in K\exp(\kappa(g))K$. We furthermore define 
\begin{equation}\label{NormG}
	||g|| = |\kappa(g)|.
\end{equation} 

On the symmetric space $X = G/K$, one defines the metric $d_X$ as 
\begin{equation}\label{SymmetricSpaceMetric}
	d_X(g.o,o) = |\kappa(g)|
\end{equation}
for the origin $o = K \in X$ and all $g \in G$. Then for $g \in KA$ it holds that $|H(g)| = d_X(g.o,o) = |\kappa(g)|$. Recall  Exercise B2 (iv) from Chapter VI of \cite{HelgasonDifferentialBook} stating that $d(a.o,o) \leq d(an.o,o)$ for all $a \in A$ and $n \in N$, which follows by applying suitably that the manifolds $A.o$ and $N.o$ are perpendicular at their unique intersection point $o\in X$. It therefore holds for all $g \in G$ that
\begin{equation}\label{Hkappainequality}
	|H(g)| \leq |\kappa(g)| = ||g||.
\end{equation}

For each $g \in G$ consider the diffeomorphism $$\alpha_g : K \to K, \quad\quad  k \mapsto \alpha_g(k) = K(gk).$$ The map $G \to \mathrm{Diff}(K), g \mapsto \alpha_g$ defines an action of $G$ on $K$. Denote by $\alpha_g'$ the Radon-Nikodym derivative of $(\alpha_g)_{*}\Haarof{K}$ with respect to $\Haarof{K}$. Then by I Lemma 5.19 of \cite{HelgasonGroupsBook},
\begin{equation}\label{GactiononKRN}
	\alpha'_g(k) = \frac{d(\alpha_g)_{*}\Haarof{K}}{d\Haarof{K}}(k) = e^{-2\delta H(g^{-1}k)}.
\end{equation}
For $r \in \mathfrak{a}^{*}$, we consider the unitary representation $\rho_r^{+}: G \to L^2(K)$ defined for $g \in G$ and $\varphi \in L^2(K)$ as
\begin{equation}\label{rhorplus}
	(\rho_r^{+}(g)\varphi)(k) = e^{-(\delta + ir)H(g^{-1}k)}\varphi(K(g^{-1}k))
\end{equation}
with $k \in K$.

The representation \eqref{rhorplus} is not irreducible in general. In order to make it irreducible, denote by $M$ the centralizer of $A$ in $K$ and write $P = MAN$ for the associated minimal parabolic subgroup. The Furstenberg boundary $\Omega = G/P$ can be identified with $K/M$ and we therefore view functions on $\Omega$ as $M$-invariant functions on $K$. The probability measure $\Haarof{\Omega}$ is the pushforward of $\Haarof{K}$ under the projection map.  For $r \in \mathfrak{a}^{*}$ we consider the $r$-principal series $\rho_r: G \to \mathscr{U}(L^2(\Omega))$ defined for $g \in G$ and $\varphi \in L^2(\Omega)$, 
\begin{equation}\label{PrincipalSeries}
	(\rho_r(g)\varphi)(\omega) = e^{-(\delta + ir)H(g^{-1}\omega)}\varphi(g^{-1}\omega)
\end{equation}
for $\omega \in \Omega$ where we denote by $g^{-1}\omega$ the element $K(g^{-1}k)M$ for any representative $\omega = kM$ with $k \in K$ and note that $H(g^{-1}\omega)$ does not depend on the representative of $\omega$ (cf. \cite{WarnerHarmonic1Book} section 5.5). The principal series is irreducible.

The Weyl group $W_G$ of $G$ is defined as the group quotient $N_K(\mathfrak{a})/Z_K(\mathfrak{a})$, where $N_K(\mathfrak{a}) = \{ k \in K \,:\, \mathrm{Ad}(k) \mathfrak{a} \subset \mathfrak{a} \}$ and $Z_K(\mathfrak{a}) = M = \{ k \in K \,:\,  ka = ak \text{ for all } a\in A \}$.

We call a root $r \in \Sigma$ indivisible if $\frac{1}{2}r$ is not a root and we order the positive roots in such a way that $r_1, \ldots , r_p$ are the indivisible roots. For any complex linear form $r$ on $\mathfrak{a}$ denote $$I(r) = \left(\prod_{\ell = 1}^p B\left( \frac{m(r_\ell)}{2}, \frac{\langle r, r_\ell \rangle}{\langle r_\ell , r_\ell \rangle} \right) \right) \cdot \left(\prod_{\ell = p + 1}^k B\left( \frac{m(r_\ell)}{2}, \frac{m(r_\ell/2)}{4}  +  \frac{\langle r, r_\ell \rangle}{\langle r_\ell , r_\ell \rangle} \right) \right),$$ where $B(x,y) = \int_0^1 t^{x-1}(1- t)^{y-1} \, dt$ is the Beta function. We further set for $r \in \mathfrak{a}^{*}$, $$c(r) = \frac{I(ir)}{I(\delta)}.$$

The spherical function of parameter $r \in \mathfrak{a}^{*}$ is defined as $\phi_r(g) = \langle \rho_{r}(g)1,1 \rangle.$ Denote by $\DifferentialOps{G}$ the set of differential operators on $G$ (see \cite{HelgasonGroupsBook} chapter 2). The Harish-Chandra Schwartz space introduced in \cite{Harish-Chandra1958} (see further page 230 of \cite{WallachReductive1Book}) is defined as 
\begin{equation}\label{SchwartzSpaceG}
	\Schwartz{G} = \{ f \in C^{\infty}(G) \,:\,  (1 + |H(g)|)^{\ell}|Df|(g) \ll_{f,D,\ell} \phi_0(g) \text{ for all } D \in \DifferentialOps{G}, \ell \geq 0  \}.
\end{equation}
The Schwartz space on $X$, denoted $\Schwartz{X}$, is defined as the set of right $K$-invariant functions in $\Schwartz{G}$. 

Recall that a function $f$ on $G$ is called bi-$K$-invariant or radial if $f(k_1gk_2) = f(g)$ for all $g \in G$ and $k_1,k_2 \in K$. For a radial function $f \in \Schwartz{G}$ we denote by $\rho_r(f)$ as in \eqref{pimudef} the operator $\int f(g) \rho_r(g)  \, d\Haarof{G}(g)$. We then define the spherical Fourier transform as 
$$\widehat{f}(r) = \langle 1, \rho_{r}(f) 1 \rangle   = \langle \rho_{-r}(f) 1,1 \rangle = \int_G f(g)\phi_{-r}(g) \, d\Haarof{G}(g).$$ Note that for all $\omega \in \Omega$ it holds that $\widehat{f}(r) = (\rho_{-r}(f)1)(\omega)$. For all $g \in G$, the spherical Fourier inversion formula holds
\begin{equation}\label{SphericalFourierInversion}
	f(g) = \int_{\mathfrak{a}^{*}} \widehat{f}(r)\phi_r(g) \, d\SphericalPlancharelMeasure(r),
\end{equation} where $d\SphericalPlancharelMeasure(r) = |c(r)|^{-2}d\Haarof{\mathfrak{a}^{*}}(r)$ is the spherical Plancharel measure. 

We furthermore define for $f \in \Schwartz{X}$, $r \in \mathfrak{a}^{*}$ and $\omega \in \Omega$, $$\widehat{f}(r,\omega) = (\rho_{-r}(f)1)(\omega) = \int_G f(g) (\rho_{-r}(g)1)(\omega) \, d\Haarof{G}(g).$$ Then it follows by a brief calculation from \eqref{SphericalFourierInversion}, for $f \in \Schwartz{X}$ and $g\in G$,
\begin{equation}\label{FourierInversionSymmetricSpace}
	f(g) = \int_{\mathfrak{a}^{*}}\int_{\Omega} \widehat{f}(r,\omega) (\rho_{r}(g)1)(\omega) \, d\Haarof{\Omega}(\omega)d\SphericalPlancharelMeasure(r).
\end{equation} We say that $f \in \Schwartz{X}$ has compactly supported Fourier transform if there is a constant $R > 0$ such that $\widehat{f}(r,\omega) = 0$ for $|r| \geq R$ and $\omega \in \Omega$.

We will further need Sobolev spaces and Sobolev norms on $X$, defined for $s \geq 0$ as 
\begin{align}\label{SobolevSpaceSymmetricSpace}
	H^s(X) = \left\{ f\in L^2(X) \,:\,||f||_{H^s}^2 =  \int_{\mathfrak{a}^{*}} ||\widehat{f}(r,\cdot)||_{L^2(\Omega)}^2 (1 + |r|^2)^{s} \, d\SphericalPlancharelMeasure(r) < \infty \right\}.
\end{align} It holds that $C_c^{\infty}(X) \subset \Schwartz{X} \subset H^s(X)$ for all $s\geq 0$ (c.f. \cite{HelgasonGroupsBook} chapter IV). 

For a probability measure $\mu$ on $G$, we write for $r \in \mathfrak{a}^{*}$ 
\begin{equation}\label{S_rDefinition}
	S_r^{+} = \rho_r^{+}(\mu) \quad\text{ and }\quad S_r = \rho_r(\mu),
\end{equation} using the definition \eqref{pimudef} for the unitary representations $\rho_r^{+}$ and $\rho_r$.

We further use the notation $\sigma = ||S_0||$. Since $MAN$ is an amenable group, it holds by Section D of \cite{Guivarch1980} that $\sigma = ||\lambda_G(\mu)||$. If $\lambda(r) \in \C$ satisfying $|\lambda(r)| = \rho(S_r)$ is in the discrete spectrum of $S_r$, has geometric multiplicity one and is the unique element of $\mathrm{spec}(S_r)$ on the circle of radius $\rho(S_r)$, then we denote by $\eta_r \in L^2(\Omega)$ the $\lambda(r)$-eigenfunction of $S_r$ with unit norm. Furthermore, if the same properties hold for $S_r^{*}$ and $\overline{\lambda(r)}$, choose $\eta_r'$ the $S_r^{*}$-eigenfunction with eigenvalue $\overline{\lambda(r)}$ satisfying $\langle \eta_r',\eta_r \rangle = 1$, provided there exists such an $\eta_r'$. Then we denote 
\begin{equation}\label{psimudef}
	\psi_{\mu,r}(g) = \langle \eta_{r},   \rho_{r}(g)\eta_{r}' \rangle
\end{equation}
for $g \in G$.

The operator $T_0: L^2(\Omega) \to L^2(\Omega)$ is defined as $$T_0 \varphi = \int \varphi \circ \alpha_g \, d\mu(g)$$ for $\varphi \in L^2(\Omega)$, where we equally denote by $\alpha_g: \Omega \to \Omega$ the map on $\Omega$ induced by $\alpha_g : K \to K$, and $$ T_0^+: L^2(K) \to L^2(K) \quad \text{ defined as }\quad  T_0^+ \varphi = \int \varphi \circ \alpha_g \, d\mu(g)$$ for $\varphi \in L^2(K)$.

\subsection{Outline of Proofs}\label{SectionOutline}

For the proof of Theorem~\ref{QuasicompactnessImpliesLLT}, Theorem~\ref{LCLT} and Theorem~\ref{StrongLCLT} one uses the Fourier inversion formula on $X$ to reduce the question at hand to spectral problems about the operators $S_r$. Indeed, by \eqref{FourierInversionSymmetricSpace} it holds for $x_0 = h_0K \in X$ with $h_0 \in G$ and $f \in \Schwartz{X}$,  
\begin{equation}\label{LLTFourierSide}
	\frac{n^{\ell/2}}{\sigma^n}\int f(g.x_0) \, d\mu^{*n}(g) = \frac{n^{\ell/2}}{\sigma^n} \int_{\mathfrak{a}^{*}} \int_{\Omega}   \widehat{f}(r,\omega) (S_r^n \rho_r(h_{0})1)(\omega) \, d\Haarof{\Omega}(k)d\SphericalPlancharelMeasure(r). 
\end{equation}
One then decomposes \eqref{LLTFourierSide} into high and low frequencies. Namely for $\delta_0 \in (0,1)$ small enough depending on $\mu$ and for $f \in \Schwartz{X}$,
\begin{align}
	\eqref{LLTFourierSide} &= \frac{n^{\ell/2}}{\sigma^n} \int_{|r| > \delta_0} \int_{\Omega}   \widehat{f}(r,\omega) (S_r^n \rho_r(h_{0})1)(\omega) \, d\Haarof{\Omega}(k)d\SphericalPlancharelMeasure(r) \label{HighFrequency}
	\\ &+ \frac{n^{\ell/2}}{\sigma^n}\int_{|r| \leq \delta_0}\int_{\Omega}   \widehat{f}(r,\omega) (S_r^n \rho_r(h_{0})1)(\omega) \, d\Haarof{\Omega}(\omega)d\SphericalPlancharelMeasure(r). \label{LowFrequency}  
\end{align} 

The following spectral properties of $S_r$ are used to deal with the arising terms:
\begin{enumerate}
	\item There are operators $E_0$ and $D_0$ such that
	\begin{equation}\label{StrongSpectralGap}
		S_0 = \sigma E_0 + D_0, 
	\end{equation}
	where $E_0$ is a projection to a one-dimensional subspace, $E_0 \circ D_0 = D_0 \circ E_0 = 0$ and $D_0$ satisfies $\rho(D_0) < \sigma = ||S_0||$. In Section~\ref{SectionStrongSpectralGap} we refer to the property \eqref{StrongSpectralGap} as strong spectral gap.
	\item For $|r| \leq \delta_0$, the operator $S_r$ has a decomposition as \eqref{StrongSpectralGap}, i.e.
	\begin{equation}\label{SrStrongSpectralGap}
		S_r = \lambda(r)E_r + D_r,
	\end{equation}
	for $E_r$ and $D_r$ as in \eqref{StrongSpectralGap}.
	\item For any $r \neq 0$, $\rho(S_r) < \sigma = ||S_0||$.
\end{enumerate}

One deduces (1) from quasicompactness of $S_0$ and by using that $S_0$ is a positive operator in the sense of Banach lattices (c.f. Section \ref{SectionStrongSpectralGap}). (2) will follow as quasicompactness is an open property under certain assumptions (Corollary~\ref{QuasicompactnessOpenProperty}) and (3) by a convexity argument similar to an argument of Conze-Guivarc'h \cite{ConzeGuivarch2013}. The necessary spectral properties are proved in Section~\ref{SectionSpectralPropertiesS(r)}.

Properties (1) and (2) will be necessary to deal with low frequencies \eqref{LowFrequency}, whereas (3) is used for high frequencies \eqref{HighFrequency}. However, (3) only allows to prove a decay for \eqref{HighFrequency} either by assuming that $f$ has compactly supported Fourier transform or by imposing the stronger assumption $(\sup_{|r| \geq 1} ||S_r||) < ||S_0||$ of Theorem~\ref{StrongLCLT}. One then deduces Theorem~\ref{QuasicompactnessImpliesLLT} and Theorem~\ref{LCLT} by approximating a given function $f \in \Schwartz{X}$ with functions whose Fourier transform is compactly supported. 

A novel contribution is the observation that the functions $\psi_{\mu,r}$ as defined in \eqref{psimudef}, where $|r| \leq \delta_0$ such that \eqref{SrStrongSpectralGap} holds, satisfy
\begin{equation}\label{FourierBacktransform}
	\int f \cdot \psi_{\mu,r} \, d\Haarof{G} = \int_{\Omega} \widehat{f}(r,\omega) (E_r1)(\omega) \, d\Haarof{\Omega}(\omega)
\end{equation}
for $f \in \Schwartz{X}$ (see Lemma~\ref{MuStationaryMeasure}). We further mention that \eqref{FourierBacktransform} may be viewed as an analogue of the formula 
\begin{equation}\label{RealFourierBackTransform}
	\int f(x)e^{-\sigma^2\frac{x^2}{2}} \, d\Haarof{\R}(x) = \frac{1}{\sqrt{2 \pi \sigma^2}} \int \widehat{f}(r)e^{-\frac{r^2}{2\sigma^2}} \, d\Haarof{\R}(r)
\end{equation}
on $\R$, where $f \in \Schwartz{\R}$ and $\sigma > 0$, which is used in the proof of the local limit theorem on $\R$. 

The outline of the proof of the local limit theorem is concluded. We next discuss quasicompactness of $S_0$. As in \cite{Bourgain2012} and \cite{BoutonnetIoanaSalehiGolsefidy2017}, the main tool are flattening statements for $\mu$. These results, which will be recalled in Section~\ref{Flattening}, have as consequence that for any $\gamma > 0$ and $x \in G$,
\begin{equation}\label{HighDimension}
	\mu^{*n}(B_{\delta}(x)) \ll \delta^{\dim G - \gamma}
\end{equation} for $\delta$ small enough depending on $\mu$ and $\gamma$ and $n \asymp_{\mu,\gamma} \log \frac{1}{\delta}$. The property \eqref{HighDimension} may be referred to as high dimension, since an absolutely continuous measure $\nu$ satisfies $\nu(B_{\delta}(x)) \asymp_{\nu} \delta^{\dim G}$. 

The proof of quasicompactness of $S_0$ comprises two steps. First we will show that the restricted operator $S_0|_{V_{\ell}}$ has small norm for all $\ell$ large enough, where $V_{\ell}$ is the Littlewood-Paley space introduced in \eqref{LittlewoodPayleySpaces}. The second step is to use the latter to deduce that $S_0$ restricted to $\bigoplus_{\ell \geq L} V_{\ell}$ has small norm for a suitable $L > 0$ and therefore is quasicompact. This exploits the first step and that the spaces $V_{\ell}$ are mutually orthogonal. Indeed, since the measure $\mu$ in question is supported close to the identity, the spaces $S_0 V_{\ell}$ and $V_{\ell'}$ are almost orthogonal too. 

For the first step, one uses that for $\varphi \in V_{\ell}$ the matrix coefficients $|\langle \rho_0(g) \varphi, \varphi \rangle|$ are small on average. Indeed, it is shown in Section~\ref{AverageMatrixCoeffSemisimple}, following \cite{LindenstraussVarju2016}, that
\begin{equation}\label{SmallAverageMatrixCoeff}
	\frac{1}{\Haarof{G}(B_R)} \int_{B_R} |\langle \rho_0(g) \varphi, \varphi \rangle| \, d\Haarof{G}(g) \ll 2^{-\ell/2} ||\varphi||_2.
\end{equation}
Since $\mu$ has high dimension, we are able to use \eqref{SmallAverageMatrixCoeff} to give strong estimates for $\langle S_0 \varphi, \varphi\rangle$ and therefore conclude a bound on the operator norm of $S_0|_{V_{\ell}}$.

In order to use \eqref{SmallAverageMatrixCoeff}, we ought to control the size of the support of $\mu^{*n}$ while ensuring that $\mu^{*n}$ has high dimension \eqref{HighDimension} quickly enough. Analogous to \cite{Bourgain2012} and \cite{BoutonnetIoanaSalehiGolsefidy2017}, this is where the $(c_1, c_2, \varepsilon)$-Diophantine property comes into play. Indeed, as $\varepsilon$ becomes smaller, a $(c_1, c_2, \varepsilon)$-Diophantine measure is increasingly rapidly non-concentrated on subgroups and therefore a strong flattening lemma applies (Lemma~\ref{SuperFlatteningLemma}). The latter holds while the measure is still close to the identity, which will allow us to conclude the claimed properties for $S_0$.

\subsection{Relation to Other Work} \label{OtherWork}
As mentioned in the introduction, the necessary results for $S_0$ are also proved in \cite{BoutonnetIoanaSalehiGolsefidy2017}. The main difference between  \cite{BoutonnetIoanaSalehiGolsefidy2017} and our proof is in the use of a different Littlewood-Paley decomposition. \cite{BoutonnetIoanaSalehiGolsefidy2017} develop a Littlewood-Paley decomposition on $G$, which leads to more general results as they are able to deal with all possible quotients of $G$, while we work with the Littlewood-Paley decomposition on $K$, leading to marginally stronger results.

For the $\mathrm{Isom}(\R^d)$ action on $\R^d$, a similar representation theoretic decomposition to \eqref{FourierInversionSymmetricSpace} holds for a suitable family of unitary representations $\rho_r : \mathrm{Isom}(\R^d) \to \mathscr{U}(L^2(\mathbb{S}^{d-1}))$ for $r \in \R$. In \cite{LindenstraussVarju2016}, a local limit theorem with strong error terms as in Theorem~\ref{StrongLCLT} is proved by just assuming that $S_0 = \rho_0(\mu)$ is quasicompact. Indeed they establish \eqref{StrongLCLTAssumption} for their setting by solely assuming that $S_0$ is quasicompact. It seems reasonable to believe that the same result may hold for a semisimple Lie group acting on its symmetric space, yet the proof of  \cite{LindenstraussVarju2016} is not transferable as several properties only applicable to $\mathrm{Isom}(\R^d)$ are used. 

We further mention that in \cite{Tolli2000} a Berry-Essen result is shown on $G$ for a probability measure with a smooth density of compact support.  

 \section{Preliminary Results}\label{SectionPreliminary}

\subsection{Quasicompact Operators}\label{SectionQuasicompactOperators} Throughout this section we denote by $\mathscr{B}$ a separable Banach space and the reader may recall the notations introduced in Section~\ref{Notations}. A bounded operator $A: \mathscr{B} \to \mathscr{B}$ is called quasicompact if $\rho_{\mathrm{ess}}(A) < \rho(A)$. In this secton we show that being quasicompact is an open property. We first state a useful lemma. 

\begin{lemma}\label{QuasicompactnessCharacterization}
	For any bounded operator $A: \mathscr{B} \to \mathscr{B}$ the following properties hold:
	\begin{enumerate}
		\item[(i)] 	$$\rho_{\mathrm{ess}}(A) = \inf_{U \text{ compact}} \rho(A - U).$$ 
		\item[(ii)] $A$ is quasicompact whenever $A^{*}$ is.  Moreover, $$\rho_{\mathrm{ess}}(A^{*}) = \rho_{\mathrm{ess}}(A).$$
		\item[(ii)] The set of spectral values of $A$ with modulus $> \rho_{\mathrm{ess}}(A)$ is at most countable and all of its accumulation points have modulus $\rho_{\mathrm{ess}}(A)$. 
	\end{enumerate}
\end{lemma}

\begin{proof}\label{SelfadjointQuasicompactnessCharacterization}
	(i) follows as the essential spectral radius is spectral radius of the image of $A$ in the Calkin algebra (c.f section 2.4 in the appendix of \cite{BenoistQuintRandomBook}) and (ii) as a bounded operator is Fredholm whenever its adjoint is (Corollary 2.12 of appendix B in \cite{BenoistQuintRandomBook}). Finally (iii) is contained in Proposition 2.14 of appendix B in \cite{BenoistQuintRandomBook}.
\end{proof}

\begin{corollary}\label{QuasicompactnessOpenProperty}
	Let $A_n : \mathscr{H} \to \mathscr{H}$ be a sequence of bounded operators on a Hilbert space $\mathscr{H}$ converging in operator norm to a bounded operator $A: \mathscr{H} \to \mathscr{H}$. If $A$ is quasicompact then so is $A_n$ for $n$ large enough and there is $\varepsilon > 0$ such that for $n$ large enough $\rho_{\mathrm{ess}}(A_n) < \rho_{\mathrm{ess}}(A) + \varepsilon < \rho(A) - \varepsilon < \rho(A_n)$.
\end{corollary}

\begin{proof}
	By Lemma~\ref{QuasicompactnessCharacterization} (i) for any $\varepsilon > 0$ there is a compact operator $U$ (depending on $\varepsilon$) such that $\rho(A - U) < \rho_{\mathrm{ess}}(A) + \varepsilon$. We choose a small $\varepsilon > 0$ such that  $\rho_{\mathrm{ess}}(A) + 2\varepsilon  < \rho(A) - 2\varepsilon$. Recall that the spectral radius is upper semi-continuous and since $A$ is quasicompact, $A$ is a continuity point for the spectral radius (cf. \cite{Newburgh1951}). Thus for large enough $n$ it holds that $\rho(A_n - U) < \rho(A - U) + \varepsilon$ and $\rho(A) - 2\varepsilon < \rho(A_n)$. Then for the above compact operator $U$,
	\begin{align*}
		\rho_{\mathrm{ess}}(A_n) \leq \rho(A_n - U) < \rho(A - U) + \varepsilon   < \rho_{\mathrm{ess}}(A) + 2\varepsilon  < \rho(A) - 2\varepsilon < \rho(A_n),
	\end{align*} showing the claim upon replacing $2\eps$ by $\eps$.
\end{proof}

\subsection{Strong Spectral Gap and Quasicompact Positive Operators}\label{SectionStrongSpectralGap}

We introduce the following definition of strong spectral gap.

\begin{definition}\label{StrongSpectralGapDef}
	Let $S : \mathscr{B} \to \mathscr{B}$ be a bounded operator on a Banach space $\mathscr{B}$. We say that $S$ has \textbf{strong spectral gap} if there are two operators $E,D : \mathscr{B} \to \mathscr{B}$ and a decomposition $S = \lambda E + D$ with $\lambda \in \C$ satisfying $|\lambda| = ||S||$ such that the following properties are satisfied:
	\begin{enumerate}
		\item[(i)] The operator $E$ is a projection onto its image and $\dim(\mathrm{Im}(E)) = 1$. 
		\item[(ii)] $E \circ D = D \circ E = 0$.
		\item[(iii)] $\rho(D) < ||S||$.
	\end{enumerate}
\end{definition}

The aim of this section is to prove Corollary~\ref{HilbertLatticeStrongSpectralGap} below on quasicompact operators which are positive in the sense of Banach lattices. We refer to the book \cite{SchaeferBanachBook} for the definition of a Banach lattice. For the convenience of the reader, we recall some further definitions from \cite{SchaeferBanachBook}.

Let $\mathscr{B}$ be a Banach lattice and denote by $\mathscr{B}_{+}$ the set of positive elements. We write $x \geq y$ whenever $x - y \in \mathscr{B}_+$ and further $x > y$ if and only if $x - y \in \mathscr{B}_+$ and $x\neq y$. We say that the bounded operator $A: \mathscr{B} \to \mathscr{B}$ is positive if $A(\mathscr{B}_{+}) \subset \mathscr{B}_{+}$, in notation $A \geq 0$. We write $A > 0$ if $A x > 0$ for $x > 0$.

We furthermore say that the operator $A$ has a strictly positive invariant form if there is a linear form $\eta$ that maps vectors  $> 0$ to real numbers $>0$ and that is invariant under $A$, i.e. $\eta \circ A = \eta$.

For an element $u \in \mathscr{B}_{+}$ we denote by $$I_u = \{ x \in \mathscr{B} \,:\, 0 \leq |x| \leq \lambda u \text{ for some } \lambda > 0  \}$$ the principal ideal generated by $u$, where as in \cite[II Definition 1.3]{SchaeferBanachBook} we write $|x| = \max(x,-x)$. The element $u$ is called quasi-interior if $I_u$ is dense in $\mathscr{B}$. 

A subspace $I$ of $\mathscr{B}$ is called an ideal if $I_u \subset I$ for all $u \in I$. An operator $A : \mathscr{B} \to \mathscr{B}$ is referred to as irreducible if the only $A$-invariant ideals are the trivial ideals $\{ 0 \}$ and $\mathscr{B}$.

\begin{theorem}(V 5.2 of \cite{SchaeferBanachBook}) \label{SchaeferStronglyIrreducible}
	Let $\mathscr{B}$ be a Banach lattice and let $A : \mathscr{B} \to \mathscr{B}$ be a positive irreducible bounded operator $> 0$  satisfying $\rho(A) = 1$ and with a strictly positive invariant form. Then the following properties hold:
	\begin{enumerate}
		\item[(i)] $1$ is an eigenvalue. The eigenspace of $1$ is one-dimensional and spanned by a quasi-interior element of $\mathscr{B}_{+}$.
		\item[(ii)] Every eigenvalue $\lambda$ of $A$ with $|\lambda| = 1$ is a root of unity and has a one dimensional eigenspace. Moreover, the latter eigenvalues form a group. 
		\item[(iii)] $1$ is the unique eigenvalue of $A$ with a positive eigenvector. 
	\end{enumerate}
\end{theorem}

Using this Theorem~\ref{SchaeferStronglyIrreducible}, we can draw the following corollary.

\begin{corollary}\label{HilbertLatticeStrongSpectralGap}
	Let $\mathscr{H}$ be a Hilbert lattice (i.e. a Hilbert space endowed with the structure of a Banach lattice) and let $A : \mathscr{H} \to \mathscr{H}$ be a positive quasicompact bounded operator $> 0$ and assume that $A^n$ are irreducible for every $n \geq 1$. Then $A$ has strong spectral gap. 
\end{corollary}

\begin{proof}
	Without loss of generality, we may replace $A$ by $A/\rho(A)$ and assume that $\rho(A) = 1$. The map $$E = \lim_{\lambda \to 1} (\lambda - 1) R(\lambda, A) $$ is the strictly positive projection to the one-dimensional eigenspace of $1$, where $R(\lambda, A)$ is the resolvent of $A$ at $\lambda$  (see corollary to Theorem V 5.2 of \cite{SchaeferBanachBook}). As the resolvent $R(\lambda, A)$ commutes with $A$, so does $E$. Note moreover that $E$ gives rise to a strictly positive invariant form. Indeed denote by $v_0 \in \mathscr{H}$ the by Theorem~\ref{SchaeferStronglyIrreducible} (i) unique eigenvector of $1$ with norm $||v_0|| = 1$ and consider $\eta(v) = \langle Ev, v_0 \rangle$ for $v\in \mathscr{H}$. Since $v_0$ is positive by Theorem~\ref{SchaeferStronglyIrreducible} (iii) it follows that $\eta$ is a strictly positive invariant form.
	
	Set $D = A -  E$. Then $E \circ D = D \circ E= 0$ as $A$ commutes with $E$ and we claim that $\rho(D) < 1$, which follows if we show that $1$ is the unique eigenvalue of $A$ on the circle of radius $1$. To show the latter, if $\lambda$ is an eigenvalue of $A$ with $|\lambda| = 1$ and eigenvector $v_{\lambda}$, then by Theorem~\ref{SchaeferStronglyIrreducible} (ii) $\lambda$ is a root of unity and hence $T^n v_{\lambda} = v_{\lambda}$ for some $n > 0$. Therefore by Theorem~\ref{SchaeferStronglyIrreducible} (i) applied to $T$ and $T^n$, it follows that $v_{\lambda}$ must be the unique positive $1$-eigenvector of $T$ and hence $\lambda = 1$. 
\end{proof}

We return to the operators $S_0$ and $S_0^{+}$ defined in \eqref{S_rDefinition}.

\begin{lemma}\label{S_rIrreducible}
	Let $G$ be a connected semisimple Lie group with finite center and let $\mu$ be a non-degenerate probability measure on $G$. Then $S_0$ and $S_0^{+}$ are positive bounded operators and $S_0^{n}$ and $(S_0^{+})^{n}$ are irreducible for all $n \geq 1$.
\end{lemma}

\begin{proof}
	We show that $S_0$ is irreducible and the same argument will apply to $S_0^n$ and $(S_0^+)^n$ for all $n \geq 1$ since $G$ is connected. By III Proposition 8.3 of \cite{SchaeferBanachBook}, it suffices to show for any $\varphi_1, \varphi_2 \in L^2(\Omega)$ with $\varphi_1 > 0$ and $\varphi_2 > 0$ that $\langle S_0^{\ell} \varphi_1, \varphi_2 \rangle$ is $> 0$ for some $\ell \geq 1$. Indeed, we may reduce to the case where $\varphi_1 = 1_{U_1}$ and $\varphi_2 = 1_{U_2}$ for $U_1$ and $U_2$ two sets of positive measure. Using that the support of $\mu$ generates a dense subgroup, we may choose $\ell$ large enough such that the support of $S_0^{\ell} 1_{U_1}$ has measure larger that $1 - \Haarof{\Omega}(U_2)/2$ and therefore $\langle S_0^{\ell} 1_{U_1}, 1_{U_2} \rangle > 0$. 
\end{proof} 

\subsection{Preliminaries on Representation Theory of Compact Lie Groups}
\label{ReptheoryCompacyLieGroups}

Recall the notation introduced in Section~\ref{Notations}. 

For $\gamma \in \overline{C} \cap I^{*}$, by Schur's Lemma, the operator $\pi_{\gamma}(\triangle)$ acts as a scalar. For functions on $K$, the operator $\lambda_G(\triangle)$ can be understood as the Laplacian. Therefore \eqref{PeterWeyl} is a decomposition into eigenfunctions of the Laplacian and on $M_{\gamma}$ the Laplacian has eigenvalue $\lambda_{\gamma} = \pi_{\gamma}(\triangle)$.

\begin{lemma}\label{DimensionLaplacianBound}
	For $\gamma \in \overline{C} \cap I^{*}$ denote $d_{\gamma} := \dim \pi_{\gamma}$ and $\lambda_{\gamma} := \pi_{\gamma}(\triangle)$. Then for $\gamma$ large enough it holds that $\lambda_{\gamma} \asymp ||\gamma||^2$ and $d_{\gamma} \ll ||\gamma||^{|R^+|}$ . Moreover, assuming that $K$ is semisimple, $ ||\gamma|| \ll d_{\gamma}.$
\end{lemma}

\begin{proof}	
	By Lemma 10.6 of \cite{HallLieBook}, $\lambda_{\gamma} := \pi_{\gamma}(\triangle) = \langle \gamma + \rho, \gamma + \rho \rangle - \langle \rho, \rho \rangle$, where $\rho = \frac{1}{2}\sum_{\alpha \in R^{+}} \alpha$ is the sum of positive half roots (notice that the multiplicity of each root is one cf. Theorem 7.23 of \cite{HallLieBook}).
	This easily implies $\lambda_{\gamma} \asymp ||\gamma||^2$. The upper bound on $d_{\gamma}$ follows by Weyl's dimension formula: $$d_{\gamma} = \prod_{\alpha \in R^{+}}  \frac{\langle \alpha, \gamma + \rho \rangle}{\langle \alpha, \rho \rangle} \leq \left(\prod_{\alpha \in R^{+}}  \frac{||\alpha||}{|\langle \alpha, \rho \rangle|} \right) ||\gamma + \rho||^{|R^+|} \ll_G ||\gamma||^{|R^+|}$$ for $||\gamma||$ large enough. For the lower bound we recall that in \cite{deSaxce2013}, also using the Weyl dimension formula, it is proved that $||\gamma||^{|R^+| - p} \ll d_{\gamma}$, where $p$ is the number of maximal elements of $R^{+}$ that are contained in one hyperplane. If $K$ is semisimple, the roots span the vector space $\mathfrak{t}^{*}$ and therefore $(|R^{+}| - p) \geq 1$.
\end{proof}

Recall the Sobolev spaces defined in \eqref{CompactGroupSobolevSpace}. We deduce a condition for a function being in $C^m(K)$ under an assumption on the decay of $||P_{\ell}\varphi||_2$. 
\begin{lemma}\label{SmoothnessCharacterization}
	Let $m \in \Z_{\geq 0}$, $s > m + \frac{1}{2}\dim K$ and let $\varphi 
	\in L^2(K)$. Assume that for all $\ell \in \Z_{\geq 0}$ large enough, $$||P_{\ell} \varphi||_2 \leq 2^{-(s+ 1)\ell}.$$ Then $\varphi \in H^s(K) \subset C^m(K)$.
\end{lemma}

\begin{proof}
	If $\varphi = \sum_{\gamma \in \overline{C} \cap I^{*}} \varphi_{\gamma}$, by the assumption for large enough $\ell$, $2^{2s\ell}||P_{\ell}\varphi||_2^2 = 2^{2s\ell}\sum_{2^{\ell - 1} \leq ||\gamma|| < 2^{\ell}} ||\varphi_{\gamma}||_2^2 \leq 2^{-2\ell}$ and hence using Lemma~\ref{DimensionLaplacianBound},
	$$\sum_{\gamma \in \overline{C} \cap I^{*}}  \lambda_{\gamma}^s ||\varphi_{\gamma}||_2^2 \ll \sum_{\ell \geq 0}  2^{2s\ell}\sum_{2^{\ell - 1} \leq ||\gamma|| < 2^{\ell}} ||\varphi_{\gamma}||_2^2 \ll \sum_{\ell \geq 0} 2^{-2\ell} < \infty,$$ showing that $\varphi \in H^s(K)$. The inclusion $H^s(K) \subset C^m(K)$ follows from the Sobolev embedding theorem (cf. \cite{AubinSomeBook} Theorem 2.10).
\end{proof}

\subsection{Flattening of $\mu^{*n}$}\label{Flattening}

In this section we state strong flattening results from \cite{BoutonnetIoanaSalehiGolsefidy2017} for $(c_1, c_2, \varepsilon)$-Diophantine measures. To introduce notation, denote $$P_{\delta} = \frac{1_{B_{\delta}}}{\Haarof{G}(B_{\delta})}$$ and for a measure $\nu$ and $g \in G$, we note that $(\nu * P_{\delta})(g) = \frac{\nu(B_{\delta}(g))}{\Haarof{G}(B_{\delta})}.$ We also use the notation $\nu_{\delta}  = (\nu)_{\delta} = \nu * P_{\delta}.$

We first relate the condition that a measure is $(c_1, c_2, \eps)$-Diophantine to the assumptions of several theorems in \cite{BoutonnetIoanaSalehiGolsefidy2017}. 

\begin{lemma}\label{ConditionLink}
	Let $c_1, c_2, \varepsilon > 0$ and let $\mu$ be a probability measure on $G$ satisfying $\mathrm{supp}(\mu) \subset B_{\varepsilon}$. Then $\mu$ is $(c_1,c_2, \varepsilon)$-Diophantine if and only if for $\delta$ small enough and $n = \frac{\log \frac{1}{\delta}}{c_1 \log \frac{1}{\varepsilon}}$, $$\sup_{H < G} \mu^{*n}(B_{\delta}(H)) \leq \delta^{\frac{c_2}{c_1}},$$ where $B_{\delta}(H) = \{ g \in G \, : \, d(g,H) < \delta \}$ and the supremum is taken over all closed subgroups of $G$.
\end{lemma}

\begin{proof}
	This follows from the fact that $\mu$ is $(c_1,c_2)$-Diophantine if and only if $\sup_{H < G}\mu^{*n}(B_{\delta}(H)) \leq \delta^{\frac{c_2}{c_1}}$ for $n = \frac{1}{c_1} \log \frac{1}{\delta}$.
\end{proof}

We state Corollary 4.2 from \cite{BoutonnetIoanaSalehiGolsefidy2017}.

\begin{theorem}(Flattening Lemma, Corollary 4.2 of \cite{BoutonnetIoanaSalehiGolsefidy2017})\label{SuperFlatteningLemma}
	Let $G$ be a connected simple Lie groups with finite center. Let $c_1, c_2 > 0$. Then for every $\gamma > 0$ there is $\varepsilon_0 = \eps_0(c_1,c_2, \gamma)  > 0$ and $C_0 = C_0(c_1,c_2, \gamma) > 0$ such that the following holds:
	
	If $0 < \varepsilon < \varepsilon_0$ and $\mu$ is a symmetric and $(c_1,c_2, \varepsilon)$-Diophantine probability measure on $G$, then for $\delta > 0$ small enough, $$||(\mu^{*n})_{\delta}||_{2} \leq \delta^{- \gamma} \quad\quad \text{ for any integer } \quad\quad n \geq C_0 \frac{\log \frac{1}{\delta}}{\log \frac{1}{\varepsilon}}.$$ 
\end{theorem}

\subsection{Estimate of Averages of Matrix Coefficients for Oscillating Functions}\label{AverageMatrixCoeffSemisimple}

In this subsection we prove the following proposition, which is inspired by \cite{LindenstraussVarju2016}. We denote $B_R = \{ g \in G \,:\, d(g,e) < R \}$.

\begin{proposition}\label{HighOscillatingAverage}
	Let $G$ be a non-compact semisimple Lie groups with finite center and maximal compact subgroup $K$. Recall the Littlewood-Paley decomposition \eqref{LittlewoodPayleyDecomposition} $L^2(K) = \bigoplus_{\ell \geq 0} V_{\ell}$ and assume further that $K$ is a semisimle Lie group. Then for any $r \in \R$ and $\ell \in \Z_{\geq 1}$, for $\varphi_1, \varphi_2 \in V_{\ell} \subset L^2(K)$, $$\frac{1}{\Haarof{G}(B_R)} \int_{B_R}  |\langle \rho_r^{+}(g)\varphi_1, \varphi_2 \rangle| \, d\Haarof{G}(g) \ll 2^{-\ell/2}||\varphi_1||_2 ||\varphi_2||_2,$$ where the representation $\rho_r^{+}$ is defined in \eqref{rhorplus}.
\end{proposition}

We recall the following lemma from \cite{LindenstraussVarju2016}.

\begin{lemma}(Proposition 5.1 of \cite{LindenstraussVarju2016})\label{LindenstraussVarjuRepresentationLemma}
	Let $(\pi, \mathscr{H})$ be a unitary representation of a compact group $K$ and let $D$ be the minimum of the dimension of all irreducible representations contained in $\pi$. Then for any vectors $u,v \in \mathscr{H}$, $$\left(  \int |\langle\pi(g)u,v \rangle|^2 \, d\Haarof{K}(k) \right)^{1/2} \leq \frac{||u||\, ||v||}{D^{1/2}}.$$
\end{lemma}

If $\pi$ is irreducible, then Lemma~\ref{LindenstraussVarjuRepresentationLemma} follows from Schur's Lemma (see \cite{KnappLieBook} Section I.5). For the general case one decomposes $\pi$ as a direct sum of irreducible representations.

\begin{proof}(of Proposition~\ref{HighOscillatingAverage})
	Denote $B_R' = B_R \cdot K$. By left invariance of the metric, it follows that $B_{R'} \subset B_{R + C}$ for $C$ an absolute constant and therefore $\Haarof{G}(B_{R'})\ll \Haarof{G}(B_R)$. Using Cauchy-Schwarz and that for $k \in K$ the operator $\rho_r^{+}(k)$ acts as the regular representation, it follows by Lemma~\ref{LindenstraussVarjuRepresentationLemma},
	\begin{align*}
		\int_{B_R}  |\langle \rho_r^{+}(g)\varphi_1, \varphi_2 \rangle| \, d\Haarof{G}(g) &\leq \int_{B_R'}  |\langle \rho_r^{+}(g)\varphi_1, \varphi_2 \rangle| \, d\Haarof{G}(g) \\ &= \int _{B_{R'}} \left(  \int_K  |\langle \rho_r^{+}(k)\varphi_1, \rho_r^{+}(g^{-1})\varphi_2 \rangle| \, d\Haarof{K}(k) \,  \right) d\Haarof{G}(g) \\
		&\leq \int _{B_{R'}} \left(  \int_K  |\langle \rho_r^{+}(k)\varphi_1, \rho_r^{+}(g^{-1})\varphi_2 \rangle|^2 \, d\Haarof{K}(k) \,  \right)^{1/2} d\Haarof{G}(g) \\
		&\leq \Haarof{G}(B_{R'}) \left( \min_{2^{\ell -1} \leq ||\gamma|| < 2^{\ell}} d_{\gamma} \right)^{-1/2} ||\varphi_1||\, ||\varphi_2|| \\
		&\ll  \Haarof{G}(B_R) 2^{-\ell/2} ||\varphi_1||\, ||\varphi_2||,
	\end{align*} having used in the last line that $||\gamma|| < d_{\gamma}$ from Lemma~\ref{DimensionLaplacianBound} under the assumption that $K$ is semisimple.
\end{proof}

\section{Proof of Local Limit Theorem}\label{SectionProofLocalLimitTheorem}

We fix throughout a non-compact semisimple Lie group $G$ with finite center. In this section we prove Theorem~\ref{QuasicompactnessImpliesLLT}, Theorem~\ref{LCLT} and Theorem~\ref{StrongLCLT}. The reader may recall the outline given in Section~\ref{SectionOutline}.

In Section~\ref{SectionSpectralPropertiesS(r)} we prove the necessary spectral properties for the operators $S_r$. Then in Section~\ref{ExistenceLimitMeasure} we prove the claimed properties of the limit measure as well as deduce \eqref{FourierBacktransform}. In Section~\ref{SectionHighFrequency} we deal with the high frequency term \eqref{HighFrequency} while in Section~\ref{SectionLowFrequency} we establish most of the necessary results to deal with the low frequency term \eqref{LowFrequency}. The proof of Theorem~\ref{LCLT} and Theorem~\ref{StrongLCLT} is then completed in Section~\ref{SectionStrongLCLT}, while Theorem~\ref{QuasicompactnessImpliesLLT} is deduced in Section~\ref{SectionLLT}.

\subsection{Spectral Properties of $S_r$}\label{SectionSpectralPropertiesS(r)}

In this section we discuss spectral results for the operators $S_0$ and $S_r$ and the function $r \mapsto \rho(S_r)$ under the assumption that $S_0$ is quasicompact and using the results developed in Section~\ref{SectionQuasicompactOperators} and Section~\ref{SectionStrongSpectralGap}. Notice that if $\mu$ is non-degenerate and $S_0$ is quasicompact, then by Lemma~\ref{S_rIrreducible} and Corollary~\ref{HilbertLatticeStrongSpectralGap} the operator $S_0$ has strong spectral gap.  

Before stating the first lemma, we mention that $|S_r \eta| \leq S_0 |\eta|$ for all $r \in \mathfrak{a}^{*}$ and $\eta \in L^2(\Omega)$, which implies $\rho(S_r) \leq ||S_0||$. Lemma~\ref{HighFreqeuncyNormLess} is concerned with improving the latter inequality to $\rho(S_r) < ||S_0||$ under suitable assumptions on $\mu$. 

\begin{lemma}\label{HighFreqeuncyNormLess}
	Let $\mu$ be a non-degenerate probability measure and assume that $S_0$ is quasicompact. Then for any non-zero $r\in \mathfrak{a}^{*}$, 
	\begin{equation}\label{HighFrequencyGap}
		\rho(S_r) < \rho(S_0) = ||S_0||.
	\end{equation}
	Moreover, for any $c_2 > c_1 > 0$ and $n$ large enough depending on $c_1$ and $c_2$, 
	\begin{equation}\label{SrPowerOperatorNormEstimate}
		\sup_{c_1\leq|r| \leq c_2} ||S_r^n||^{\frac{1}{n}} < ||S_0||.
	\end{equation}
\end{lemma}

\begin{proof}
	To prove \eqref{HighFrequencyGap}, we follow ideas from the proof of Theorem 3.9 of \cite{ConzeGuivarch2013}. Fix  a non-zero $r\in \mathfrak{a}^{*}$. We assume for a contradiction that $\rho(S_r) = \rho(S_0)$ and therefore there is $\lambda = e^{i \gamma} \rho(S_0) \in \mathrm{spec}(S_r)$ for $\gamma \in \R$. Then (cf. Section 12.1 of \cite{EinsiedlerWardFunctionalBook}) either $\lambda$ is in the discrete spectrum or in the approximate spectrum, i.e. there is a sequence $\eta_\ell \in \ker(S_r - \lambda \cdot \mathrm{Id})^{\perp}$ with $||\eta_\ell|| = 1$ and 
	\begin{equation}\label{ApproximateSpectrum}
		\lim_{\ell \to \infty} ||S_r\eta_\ell - \lambda \eta_\ell|| = 0.
	\end{equation}
	
	Note that as $S_0$ is quasicompact, $\rho(S_0) = ||S_0||$. We first treat the case where $\lambda$ is in the discrete spectrum, i.e. that there exists $\eta \in L^2(\Omega)$ such that $S_r \eta = \lambda \eta$. Then $||S_0||\, |\eta| = |S_r \eta| \leq S_0 |\eta|$ and thus $||\, S_0|\eta|\,|| = ||S_0||\, ||\eta||$. Denote by $\eta_0$ the $||S_0||$-eigenfunction of $S_0$ with unit norm. As $S_0$ has strong spectral gap (by Lemma~\ref{S_rIrreducible} and Corollary~\ref{HilbertLatticeStrongSpectralGap}), it follows that $\eta(\omega) = e^{i \theta(\omega)} \eta_0(\omega),$ for $\theta: \Omega \to \R$ a measurable function and $\omega \in \Omega$. 
	
	Then for almost all $\omega \in \Omega$ and $n \geq 1$, 
	\begin{align*}
		\int e^{-(\delta + ir)H(g^{-1}\omega) + i\theta(g^{-1}\omega)} \eta_0(g^{-1}\omega) \, d\mu^{*n}(g) &= (S_r^n \eta)(\omega) \\
		&= \lambda^n \eta(\omega) \\
		&= e^{in\gamma} ||S_0||^n e^{i \theta(\omega)} \eta_0(\omega) \\
		&= e^{i(n\gamma + \theta(\omega))} \int e^{-\delta H(g^{-1}\omega)} \eta_0(g^{-1}\omega) \, d\mu^{*n}(g).
	\end{align*}
	As $\eta_0$ is a quasi-interior element by Theorem~\ref{SchaeferStronglyIrreducible}, it must hold that $\eta_0(\omega) > 0$ for almost all $\omega \in \Omega$. Hence for almost all $\omega \in \Omega$ and $g \in \mathrm{supp}(\mu^{*n})$, $$e^{- i(rH(g^{-1}\omega) - \theta(g^{-1}\omega) + \theta(\omega) + n\gamma)}= 1.$$ 
	
	If $r \neq 0$, for a fixed $\omega \in \Omega$ and $n \geq 1$, we can choose $h_n \in G$ such that $e^{-irH(h_n^{-1}\omega)} = e^{i(n\gamma + \pi)}$ yet $e^{i(\theta(h_n^{-1}\omega) - \theta(\omega))} = 1$. Indeed, for a representative $\omega = k M$ for $k \in K$, we may choose $h_n = ka_nk^{-1}$ for an element $a_n \in A$ satisfying $e^{-irH(a_n^{-1})} = e^{i(n\gamma + \pi)}$ as then $H(h_n^{-1}k) = H(a_n^{-1})$ and $\theta(h^{-1}\omega) = \theta(\omega)$. We may choose the $h_n$ within a bounded region of $G$ and therefore upon replacing $h_n$ with a subsequence we may assume that $h_n$ converges to some element $h \in G$. Since $\mu$ is non-degenerate we can find some $n$ and $g \in \mathrm{supp}(\mu^{*n})$ such that $g$ becomes arbitrarily close to $h$ and hence for $n$ large enough also to $h_n$. This is a contradiction.
	
	It remains to assume that $\lambda$ is in the approximate spectrum. Let $\eta_{\ell}$ as in \eqref{ApproximateSpectrum}. Since $\langle S_r \eta_{\ell}, \lambda \eta_{\ell} \rangle = \langle S_r \eta_{\ell} - \lambda \eta_{\ell} , \lambda \eta_{\ell} \rangle + ||S_0||^2$, it follows that $\langle S_r \eta_{\ell}, \lambda \eta_{\ell} \rangle \xrightarrow{\ell \to \infty} ||S_0||^2$ and furthermore exploiting $|\langle S_r \eta_{\ell}, \lambda \eta_{\ell} \rangle| \leq \langle S_0 |\eta_{\ell}|, ||S_0|| \, |\eta_{\ell}| \rangle$ one concludes $$\lim_{\ell \to \infty} \langle S_0 |\eta_{\ell}|, ||S_0|| \, |\eta_{\ell}| \rangle = ||S_0||^2$$ and hence $||S_0 |\eta_{\ell}| - ||S_0|| \, |\eta_{\ell}| \, ||^2 \leq  2||S_0||^2 - 2 \langle S_0 |\eta_{\ell}|, ||S_0|| \, |\eta_{\ell}| \rangle \xrightarrow{\ell \to \infty}  0.$
	
	Denote $\psi_ {\ell}= |\eta_{\ell}| - \langle |\eta_{\ell}|, \eta_0 \rangle \eta_0 \in \langle \eta_0 \rangle^{\perp} \subset L^2(\Omega)$. Then it holds that $$||(S_0 - ||S_0||) \psi_{\ell}||_2 = ||S_0 |\eta_{\ell}| - ||S_0|| \, |\eta_{\ell}| \, ||_2 \xrightarrow{\ell \to \infty}  0.$$ Since $\psi_{\ell}$ in $\langle \eta_0 \rangle^{\perp}$ and $S_0 - ||S_0||$ is invertible on $\langle \eta_0 \rangle^{\perp}$ it follows that $||\psi_\ell||_2 \to 0$. Notice that $||\psi_{\ell}||_2^2 = 1 - \langle |\eta_{\ell}|, \eta_0 \rangle^2$ and hence $\langle |\eta_{\ell}| , \eta_0 \rangle \to 1$ and further $|| \, |\eta_{\ell}| - \eta_0 ||_2 \to 0.$ Upon replacing $\ell$ by a subsequence, we can assume that $|\eta_{\ell}|$ converges pointwise to $\eta_0$ almost everywhere.
	
	We further note that for all $n \geq 1$, $\langle S_r^n \eta_{\ell}, \lambda^{n}\eta_{\ell} \rangle \to ||S_0||^{2n}$ as $\ell \to \infty$. Indeed this follows by induction as 
	\begin{align*}
		&\langle S_r^n \eta_{\ell} - S_r^{n-1} \lambda \eta_{\ell} + S_r^{n-1}\lambda \eta_{\ell}, \lambda^{n}\eta_{\ell} \rangle \\ &= \langle S_r^{n-1}(S_r \eta_{\ell} - \lambda \eta_{\ell}), \lambda^n \eta_{\ell} \rangle + ||S_0||^2 \langle S_r^{n-1}\eta_{\ell}, \lambda^{n-1} \eta_{\ell} \rangle \to ||S_0||^{2n}.
	\end{align*}
	
	Write $\lambda = e^{i\gamma}||S_0||$ and $\eta_{\ell}(\omega) = e^{i\theta_{\ell}(\omega)} |\eta_{\ell}|(\omega)$ for $\theta_\ell: \Omega \to \R$ a measurable function and $\omega \in \Omega$. Notice that $\langle S_r^n \eta_{\ell} , \lambda^n \eta_{\ell} \rangle$ equals $$\int \int e^{-(\delta + ir)H(g^{-1}\omega) + i(\theta_\ell(g^{-1}\omega) - \theta_\ell(\omega) - n\gamma)} ||S_0||^n \, |\eta_{\ell}|(g^{-1}\omega) |\eta_{\ell}|(\omega) \, d\mu^{*n}(g)d\Haarof{\Omega}(\omega)$$ and on the other hand $$\langle S_0^n \eta_0, ||S_0||^n \eta_0 \rangle = \int \int e^{-\delta H(g^{-1}\omega)} ||S_0||^n\eta_0(g^{-1}\omega)\eta_0(\omega) \, d\mu^{*n}(g)d\Haarof{\Omega}(\omega).$$ As $\langle S_r^n \eta_{\ell}, \lambda^n \eta_{\ell} \rangle \xrightarrow{\ell \to \infty} ||S_0||^{2n} = \langle S_0^n\eta_0, ||S_0||^n \eta_0 \rangle$ and since almost surely $|\eta_{\ell}| \to \eta_0$, we conclude that for almost all $g \in \mathrm{supp}(\mu^{*n})$ and $\omega \in \Omega$, $$ \lim_{\ell \to \infty} e^{i(rH(g^{-1}\omega) - \theta_{\ell}(g^{-1}\omega) + \theta_{\ell}(\omega) + \gamma)} = 1.$$ This leads to a contradiction by a similar argument to the case of the discrete spectrum.
	
	To prove \eqref{SrPowerOperatorNormEstimate}, we notice that for an operator $T$ on a Hilbert space $\mathscr{H}$ with $||T|| \leq 1$, the value of $||T^n||^{\frac{1}{n}}$ for a given $n$ controls $||T^k||^{\frac{1}{k}}$ for any $k \geq n$. Indeed (cf. \cite{RemlingMathOverflow}) if $k = \ell n + j$ for $0 \leq j \leq n-1$ then it holds that 
	\begin{equation}\label{highkcontrolled}
		||T^k||^{\frac{1}{k}} \leq (||T^{\ell n}||^{\frac{1}{\ell n}})^{\frac{\ell n}{k}} ||T||^{\frac{j}{k}} \leq (||T^n||^{\frac{1}{n}})^{1 - \frac{j}{k}}  ||T||^{\frac{j}{k}}.
	\end{equation}
	Therefore for $k$ large enough in terms of $n$, $||T^k||^{\frac{1}{k}}$ is at most slightly larger than $||T^n||^{\frac{1}{n}}$. Assume now for a contradiction that \eqref{SrPowerOperatorNormEstimate} does not hold. Then there is a sequence $(n_i)_{i \geq 1}$ with $n_i \to \infty$ and for each $i$ there is $r_i$ with $ ||S_{r_i}^{n_i}||^{\frac{1}{n_i}} = ||S_0||$. As the set $\{c_1 \leq |r| \leq c_2\}$ is compact, we may choose a subsequence of the $i$ such that $r_i$ converges to $r \in \mathfrak{a}^{*}$ with $c_1 \leq |r| \leq c_2$. We arrive at a contradiction as by \eqref{highkcontrolled}, $||S_{r_i}^{n_i}||^{\frac{1}{n_i}}$ is at most marginally larger than $||S_r^n||^{\frac{1}{n}}$ for $r_i$ close enough to $r$. Indeed, choose $\varepsilon > 0$ small enough such that $\rho(S_r) + 3 \varepsilon < ||S_0||$ and fix $n$ large enough such that $||S_r^{n}||^{\frac{1}{n}} \leq \rho(S_r) + \varepsilon$. Then for $r_i$ close enough to $r$,  $||S_{r_i}^n||^{\frac{1}{n}} \leq \rho(S_r) + 2\varepsilon$ and hence by \eqref{highkcontrolled}, choosing $i$ sufficiently large, $||S^{n_i}_{r_i}||^{\frac{1}{n_i}} \leq \rho(S_r) + 3\varepsilon < ||S_0||$, a contradiction to the assumption. 
\end{proof}

\begin{proposition}\label{S_rStrongSpectralGap}
	Let $\mu$ be a  non-degenerate probability measure with finite second moment and assume that $S_0$ is quasicompact. Then there is $\delta_0 = \delta_0(\mu) > 0$ such that for any $r \in \mathfrak{a}^{*}$ with $|r| \leq \delta_0$ the operators $S_r$ and $S_r^{*}$ have strong spectral gap. 
	
	More precisely there is $0 < \delta_0 < 1$ small enough satisfying the following properties. For $|r| \leq \delta_0$ we can write
	\begin{equation} \label{StrongSpectralGapsmallr}
		S_r = \lambda(r)E_r + D_r \quad\quad \text{and} \quad\quad S_{r}^{*} = \overline{\lambda(r)} E_r^{*} + D_r^{*}
	\end{equation} where $\lambda(r)$, $E_r$ and $D_r$ and equally $\overline{\lambda(r)}, E_r^{*}$ and $D_r^{*}$ satisfy the assumptions of Definition~\ref{StrongSpectralGapDef}, and the following properties hold:
	\begin{enumerate}
		\item[(i)] $\sup_{|r| \leq \delta_0} ||D_r|| \leq (1 - c)||S_0||$ for $c = c(\mu) > 0$.
		\item[(ii)] $||E_r - E_0|| \ll_{\mu} |r|^2$ and $||E_r^{*} - E_0^{*}|| \ll_{\mu} |r|^2$ for $|r| \leq \delta_0$. 
		\item[(iii)] Let $\eta_r$ be the unique $\lambda(r)$-eigenfunction of $S_r$ with unit norm. Then for small enough $r$ there exists a unique $\overline{\lambda(r)}$-eigenfunction $\eta_r'$ of $S_r^{*}$ satisfying $\langle \eta_r', \eta_r \rangle = 1$. Then for $\varphi \in L^2(\Omega),$ $$E_r\varphi = \langle \varphi, \eta'_r \rangle \eta_r.$$
		\item[(iv)] Moreover, $$||\eta_r - \eta_0||_2 \ll_\mu |r|^2, \quad\text{ and }\quad  ||\eta_r' - \eta_0'|| \ll_{\mu} |r|^2$$ for $|r| \leq \delta_0$.
	\end{enumerate}
\end{proposition}

\begin{proof}
	As $\mu$ has finite second moment, the directional derivatives of second order of the family of operators $S_r$ and $S_r^{*}$ exist. Therefore the function $r \mapsto ||S_r - S_0||$ is $C^2$. Since $\overline{S_r \varphi} = S_{-r}\overline{\varphi}$ for $\varphi \in L^2(\Omega)$, it follows by Taylor's theorem that $||S_r - S_0|| \ll_{\mu} |r|^2$ for small $r$. By Corollary~\ref{QuasicompactnessOpenProperty} and Corollary~\ref{HilbertLatticeStrongSpectralGap}, $S_0$ has strong spectral gap and $S_r$ is quasicompact for small $r$. Equally by Lemma~\ref{QuasicompactnessCharacterization} (ii) and since $S_0^{*} = \int \rho_0(g^{-1}) \, d\mu(g)$ is a positive operator too, it follows that $S_0^{*}$ has strong spectral gap and $S_r^{*}$ is quasicompact for small $r$. 
	
	We show that there is $\delta_0, c > 0$ small enough such that for $|r| \leq \delta_0$ and two orthogonal functions of unit norm $\varphi_1, \varphi_2 \in L^2(\Omega)$ it must hold for either $i = 1$ or $i = 2$ that 
	\begin{equation}\label{SrGapMainEstiamte}
		||S_r \varphi_i||_2 \leq (1 - c)||S_0||.
	\end{equation} Indeed, assume for a contradiction that \eqref{SrGapMainEstiamte} does not hold. Then $||S_0 \varphi_i||_2 \geq ||S_r \varphi_i||_2 - ||(S_r - S_0)\varphi_i||_2 \geq (1 - c)\lambda(0) + O_{\mu}(|r|^2) \geq (1 - 2c)||S_0||$ for $r$ small enough. For $c$ small enough, as $S_0$ has strong spectral gap and $\langle \varphi_1, \varphi_2 \rangle = 0$, the latter is a contradiction.  
	
	Therefore we have shown for $|r| \leq \delta_0$ that the $\lambda(r)$-eigenspace of $S_r$ is one dimensional and on its complement the norm of $S_r$ is bounded by $(1-c)||S_0||$. Choose $\delta_0 > 0$ in addition small enough such $||S_0||(1 - \frac{c}{2}) < \inf_{|r| \leq \delta_0} \lambda(r)$. Denote by $\gamma_1 : \mathbb{S}^1 \to \C$ a smooth parametrization of the closed circle of radius $||S_0||(1 - \frac{c}{2})$ around zero and by $\gamma_2: \mathbb{S}^1 \to \C$ a smooth parametrization of the circle of radius $\frac{||S_0||c}{2}$ around $||S_0||$. Consider the operators
	\begin{equation}
		P_r = -\frac{1}{2\pi i}\int_{\gamma_1} R(z,S_r) \, dz, \quad\quad \text{and} \quad\quad E_r = -\frac{1}{2\pi i}\int_{\gamma_2} R(z,S_r) \, dz,
		\label{Erexpression}
	\end{equation} for $R(z,S_r) = (S_r - z \cdot \mathrm{Id})^{-1}$ the resolvent of $S_r$ at $z$.  Then by Theorem 6.17 of Chapter 3 in \cite{KatoPertubationBook}, the operators
	$E_r$ and $P_r$ are commuting projections with $\mathrm{Id} = E_r + P_r$ and where  $\mathrm{ker}(P_r) = \mathrm{Im}(P_r)$ is the one dimensional eigenspace of $S_r$ with eigenvalue $\lambda(r)$. By setting $D_r = S_r P_r$ , we therefore have shown that $S_r = S_r(E_r + P_r) = \lambda(r)E_r + D_r$ has strong spectral gap and that (i) holds.
	
	We claim that the operators $E_r$ and $P_r$ are also $C^2$. Indeed by Lemma 3 of Chapter VII.6 of \cite{DunfordSchwartzLinearBook}, it holds that whenever $||S_r - S_0|| < ||R(z,S_0)||^{-1}$, then for any $z$ in the resolvent set of $S_0$ that $z$ is also in the resolvent set for $S_r$ and that $$R(z,S_r) = R(z,S_0)\sum_{n = 0}^{\infty} (S_r - S_0)^n R(z,S_0)^n. $$ Since $S_r$ is $C^2$ it therefore follows that for $r$ small enough $R(z,S_r)$ is also $C^2$ on $\gamma_1$ and $\gamma_2$. Thus $||P_r - P_0|| \ll_{\mu} |r|^2$ and $||E_r - E_0||  \ll_{\mu} |r|^2$ and the claim for $E_r^{*}$ is established similarly. 
	
	To show (iii), first assume that such an $\eta_r'$ exists. Then as $E_r \varphi = \langle \varphi, \psi \rangle \eta_r$ for some $\psi \in L^2(\Omega)$ with $S_r E_r = E_r S_r$ and $E_r^2 = E_r$ it follows that $S_r^*\psi = \overline{\lambda(r)}\psi$ and that $\langle \eta_r, \psi \rangle = 1$, which implies that $\psi = \eta_r'$. By the above, it follows that there is a unique $\overline{\lambda(r)}$-eigenfunction of $S_r^{*}$ with unit norm for $|r| \leq \delta_0$, yet we need to show that there exists one with $\langle \eta_r, \eta_r' \rangle = 1$. For $r = 0$ this holds as both eigenfunctions are positive almost surely and for small $r$ we apply (iv) (for $\eta_r'$ with a fixed norm) to show that there is a $\overline{\lambda(r)}$-eigenfunction $\eta_r'$ of $S_r^{*}$ satisfying $\langle \eta_r, \eta_r' \rangle \neq 0$ and therefore upon normalizing $\eta_r'$ the claim follows. 
	
	To conclude, we show (iv) for $||\eta_r - \eta_0||_2$ and note that the same argument applies to $||\eta_r' - \eta_0'||_2$. The claim is deduced from (ii) by noticing that for $\delta_0$ small enough, $\eta_r = \frac{E_r \eta_0}{||E_r \eta_0||}$. Indeed, $||E_r \eta_0|| = |||E_r \eta_0 - \eta_0 + \eta_0|| \geq ||\eta_0|| - ||(E_r - E_0) \eta_0|| > \frac{1}{2}$ for $\delta$ small enough. To prove (iv), notice 
	\begin{align*}
		||\eta_r - \eta_0||_2 &\leq ||\tfrac{E_r \eta_0}{||E_r\eta_0||} - \tfrac{E_0 \eta_0}{||E_r\eta_0||}||_2 + ||\tfrac{E_0 \eta_0}{||E_r\eta_0||} - \eta_0||_2 \\
		&\ll_{\mu} ||E_r - E_0|| + | \tfrac{1}{||E_r\eta_0||} - 1| \ll_{\mu} |r|^2,
	\end{align*} using that $1 = ||E_0 \eta_0||$ and $|\tfrac{1}{||E_r\eta_0||} - 1| \leq |\tfrac{||E_0\eta_0|| - ||E_r\eta_0||}{||E_r\eta_0||}| \ll_{\mu} ||E_r - E_0|| \ll_{\mu} |r|^2$.
\end{proof}

\begin{proposition}\label{s(r)behavioursmallr}
	Let $\mu$ be a  non-degenerate probability measure with finite second moment and assume that $S_0$ is quasicompact. Then $\lambda(r)$ is a $C^2$-function and the Hessian $H_{\lambda,0}$ of $\lambda$ at $0$ is a negative definite sesquilinear form. 
\end{proposition}

\begin{proof}
	Using the notation of the proof of Proposition~\ref{S_rStrongSpectralGap}, it holds that $\lambda(r) = \frac{\langle S_r E_r \eta_0, \eta_0' \rangle}{\langle E_r \eta_0, \eta_0' \rangle}$ and therefore for $r$ small enough it follows that $r \mapsto \lambda(r)$ is a $C^2$-function. 
	
	For the remainder we follow roughly the proof of Proposition 2.2.7 of \cite{Bougerol1981}. To show that $H_{\lambda,0}$ is negative definite, we fix a non-zero element $r \in \mathfrak{a}^{*}$ and prove that the function $\xi(t) = \lambda(tr)$ has strictly negative second derivative at zero. Consider the function $h_n(t) = \langle D_{tr}^n \eta_0, \eta_0' \rangle$. As $D_{tr}^n = (\mathrm{Id} - E_{tr})D_{tr}^n(\mathrm{Id} - E_{tr})$ it holds that 
	\begin{align*}
		|h_n(t)| &= |\langle D_{tr}^n(\mathrm{Id} - E_{tr})\eta_0, (\mathrm{Id} - E_{tr})^{*} \eta_0' \rangle| \\
		&\leq ||D_{tr}||^n ||(\mathrm{Id} - E_{tr}) \eta_0|| \, ||(\mathrm{Id} - E_{tr})^{*} \eta_0'|| \\
		&\leq ||D_{tr}||^n ||(E_0 - E_{tr}) \eta_0|| \, ||(E_0 - E_{tr})^{*} \eta_0'||  \\ &\leq ||D_{tr}||^n ||(E_0 - E_{tr})||\, ||(E_0^{*} - E_{tr}^{*})|| \ll ||D_{rt}^n|| t^2,
	\end{align*} using Proposition~\ref{S_rStrongSpectralGap} (ii). In particular, using Proposition~\ref{S_rStrongSpectralGap} (i), $\lambda(0)^{-n}|h_n(t)| \ll_{\mu,r} t^2$ for all $n \geq 1$ and small $t$ and therefore $\lambda(0)^{-n}h_n''(0)$ is bounded for all $n \geq 1$ as otherwise Taylor's theorem would yield a contradiction. 
	
	As $\xi(0) = \lambda(0)$ and $\xi'(0) = 0$, it follows that
	\begin{align}\nonumber
		\frac{d^2}{d^2 t}\bigg|_{t = 0} \left( \lambda(0)^{-n} \langle S_{tr}^n \eta_0, \eta_0 \rangle \right) &= 	\frac{d^2}{d^2 t}\bigg|_{t = 0} \left( \lambda(0)^{-n} \xi(t)^n\langle E_{tr} \eta_0, \eta_0 \rangle + \lambda(0)^{-n}h_n(t) \right) \\
		&=n\lambda(0)^{-1}\xi''(0) + \frac{d^2}{d^2 t}\bigg|_{t = 0} \langle E_{tr} \eta_0, \eta_0 \rangle + \lambda(0)^{-n}h_n''(0). \label{DerivativeExpression}
	\end{align} Note that $ \frac{d^2}{d^2 t}|_{t = 0} \langle E_{tr} \eta_0, \eta_0 \rangle$ is also bounded as by Proposition~\ref{S_rStrongSpectralGap},  $| \langle E_{tr} \eta_0, \eta_0 \rangle| \ll_{\mu,r} 1 + t^2$.
	
	We finally consider the functions $f_n(t) = \lambda(0)^{-n}\langle S_{tr}^n \eta_0, \eta_0 \rangle$ for $n \geq 1$. We claim that the function $f_n(t)$ is positive definite. Indeed,  for $t_1, \ldots , t_m \in \R$ and $\alpha_1, \ldots , \alpha_m \in \C$, 
	\begin{align*}
		&\lambda(0)^n \sum_{k, \ell} \alpha_k \overline{\alpha_{\ell}} f_{n}(t_k - t_\ell) = \sum_{k, \ell} \langle S_{(t_k - t_\ell)r}^{n} \alpha_k \eta_0 , \alpha_\ell \eta_0\rangle \\ &= \sum_{k, \ell}\int \alpha_k \overline{\alpha_\ell} e^{-i(t_k - t_\ell)rH(g^{-1}k)} e^{-\delta H(g^{-1}k)} \eta_0(g^{-1}.k) \eta_0(k) \, d\mu(g)d\Haarof{\Omega}(k) \\
		&= \int \bigg|\sum_{k} e^{-i t_k r H(g^{-1}k)}\alpha_k  \bigg|^2 e^{-\delta H(g^{-1}k)} \eta_0(g^{-1}.k) \eta_0(k) \, d\mu(g)d\Haarof{\Omega}(k),
	\end{align*} which is positive as $\eta_0 \geq 0$.
	Therefore by Bochner's theorem and since $f_n(0) = 1$ one may expresses $f_n$ as the Fourier transform of a real valued random variable $X_n$, i.e. $f_n(t) = \int e^{itx} \, d\mu_{X_n}(x)$. Denote by $v_n = -i f_n'(0)$ the expected value of $X_n$ and by $\sigma_n^2 = -f_n''(0)$ its variance. For any given $c > 0$ we notice that $P[|X_n - v_n| < c] \to 0$ as $n \to \infty$ since by Lemma~\ref{HighFreqeuncyNormLess} it holds that $f_n(t) \to 0$ for $t \neq 0$ as $n \to \infty$ and therefore $\mu_n$ weakly converges to the zero measure. Applying Chebyschev's inequality, $$ 1 - \frac{\sigma_n^2}{c^2} \leq 1 - P[|X_n - v_n| \geq c]  = P[|X_n - v_n| < c] \to 0$$ and hence $\sigma_n^2 \geq c^2/2$ for any large enough $n$. Thus $f_n''(0) \to - \infty$ which by \eqref{DerivativeExpression} can only happen if $\xi''(0) < 0$. This concludes the proof.  
\end{proof}

\subsection{The Limit Measure}
\label{ExistenceLimitMeasure}

In this section we establish the claimed properties of the functions $\psi_{\mu,r}$ as stated in \eqref{FourierBacktransform}. A multiple of $\psi_{\mu,0}$ is the limit function of Theorem~\ref{QuasicompactnessImpliesLLT}. 

The main lemma of this section may be viewed as a Lie group analogue of \eqref{RealFourierBackTransform}.

\begin{lemma}\label{MuStationaryMeasure}
	Let $\mu$ and $\delta_0 \in (0,1)$ be as in Proposition~\ref{S_rStrongSpectralGap}. Denote for $|r| \leq \delta_0$ by $\eta_r$ the unique $\lambda(r)$-eigenfunction of $S_r$ with unit norm and by $\eta_r'$ the $S_r^{*}$-eigenfunction with eigenvalue $\overline{\lambda(r)}$ satisfying $\langle \eta_r', \eta_r \rangle = 1$. Then the continuous function 
	\begin{equation}\label{DefPsiMuR}
		\psi_{\mu,r}(g) = \langle \eta_{r},   \rho_{r}(g)\eta_{r}' \rangle
	\end{equation}
	satisfies  $\mu * \psi_{\mu,r} = \psi_{\mu,r} * \mu  = \lambda(r) \psi_{\mu,r}$. Moreover, for any $f\in \Schwartz{X}$ and $h \in G$, \begin{equation}
		\int f  \cdot \rho_G(h) \psi_{\mu,r} \, d\Haarof{G} = \int_{\Omega}  \widehat{f}(r,\omega) (E_r\rho_r(h^{-1})1)(\omega) \, d\Haarof{\Omega}(\omega),
		\label{SymmSpaceIntLimitMeas}
	\end{equation} where $\rho_G$ is the right regular representation of $G$ and we view $f$ as a right $K$-invariant eigenfunction on $G$. 
\end{lemma}

\begin{proof}
	The relation \eqref{SymmSpaceIntLimitMeas} follows as for $f \in \Schwartz{X}$ and $h \in G$,
	\begin{align*}
		\int f  \cdot \rho_G(h)\psi_{\mu,r} \, d\Haarof{G}  &= \langle \eta_{r}, \rho_{r}(f) \rho_{r}(h)\eta_{r}'  \rangle \\
		&= \langle \eta_{r},  \rho_{r}(f) \rho_{r}(\Haarof{K}) \rho_{r}(h)\eta_{r}' \rangle \\
		&= \langle \eta_{r},  \rho_{r}(f) \langle\eta_{r}',\rho_{r}(h^{-1})1 \rangle 1 \rangle \\
		&= \langle \langle \rho_{r}(h^{-1}) 1,\eta_{r}' \rangle\eta_{r},  \rho_{r}(f) 1 \rangle \\
		&=\int_{\Omega} \widehat{f}(r,\omega)(E_{r}\rho_{r}(h^{-1}) 1)(\omega) \, d\Haarof{\Omega}(\omega),
	\end{align*}
	having used in the last line that $\widehat{f}(r,k) = \rho_{-r}(f)(1) = \overline{\rho_r(f)(1)}.$
	
	To show that $\mu * \psi_{\mu,r} = \lambda(r) \psi_{\mu,r}$, we calculate for $g \in G$
	\begin{align*}
		(\mu  * \psi_{\mu,r})(g)  &= \int \psi_{\mu,r}(h^{-1}g) \, d\mu(h) \\ &= \langle \eta_{r},  S_{r}^* \rho_{r}(g) \eta_{r}' \rangle \\&= \langle S_{r}\eta_{r},  \rho_{r}(g) \eta_{r}' \rangle= \lambda(r) \psi_{\mu,r}(g).
	\end{align*} A similar argument shows that $\psi_{\mu,r} * \mu = \lambda(r) \psi_{\mu,r}$.
\end{proof}

For later reference we show the following lemma.

\begin{lemma}\label{PsimurEstimate}
	Let $\mu$ be a non-degenerate probability measure on $G$ with finite second moment and assume that $S_0$ is quasicompact. Denote by $\delta_0$ the constant obtained from Proposition~\ref{S_rStrongSpectralGap}. Then for $|r| \leq \delta_0$ with $\delta_0$ small enough, and $g \in G$,
	\begin{equation}\label{PsimuDistance}
		|\psi_{\mu,r}(g) - \psi_{\mu,0}(g)| \ll |r|(1 + ||g||).
	\end{equation}
	Moreover, for $|r| \leq \delta_0$  and $g \in G$,
	\begin{equation}\label{PsimuTrickDistance}
		\bigg|\frac{\psi_{\mu,r}(g) +\psi_{\mu,-r}(g) }{2} - \psi_{\mu,0}(g)\bigg| \ll |r|^2(1 + ||g||^2).
	\end{equation} 
\end{lemma}

\begin{proof}
	Observe that
	\begin{align*}
		|\psi_{\mu,r}(g) - \psi_{\mu,0}(g)| &=  |\langle \eta_{r},   \rho_{r}(g)\eta_{r}' \rangle - \langle \eta_{0},   \rho_{0}(g)\eta_{0}' \rangle | \\ &= |\langle \rho_{r}(g^{-1})\eta_{r},   \eta_{r}' \rangle - \langle \rho_{0}(g^{-1})\eta_{0},   \eta_{0}' \rangle |   \\&\leq  |\langle \rho_r(g^{-1})\eta_r, \eta_r' - \eta_0' \rangle| + |\langle \rho_r(g^{-1})\eta_r - \rho_0(g^{-1})\eta_0, \eta_0' \rangle | \\ &\ll_{\mu} ||\eta_r' - \eta_0'||_2 + ||(\rho_r(g^{-1}) - \rho_0(g^{-1})) \eta_0||_2.
	\end{align*}
	Thus in order to prove \eqref{PsimuDistance}, by using Proposition~\ref{S_rStrongSpectralGap} (iii) it suffices to deal with $||(\rho_r(g^{-1}) - \rho_0(g^{-1})) \eta_0||_2$. One calculates that for $g \in G$ and $\omega \in \Omega$,
	\begin{align}
		|(\rho_r(g^{-1}) - \rho_0(g^{-1})) \eta_0(\omega)| &= |(e^{-irH(g\omega)} - 1)|  |e^{-\delta H(g\omega)} \eta_0(g\omega)| \nonumber \\
		&\ll |r| \, ||g|| \, |e^{-\delta H(g\omega)} \eta_0(g\omega)|. \label{rho_rDistanceCoreEstimate}
	\end{align} Equation \eqref{PsimuDistance} therefore follows by squaring the latter term, integrating over $\Omega$ and using that $||\rho_0(g) \eta_0||_2 = ||\eta_0||_2 = 1$. For \eqref{PsimuTrickDistance} one performs the same calculation and notices that $$\bigg|\left(\frac{\rho_r(g) + \rho_{-r}(g)}{2} - \rho_0(g)\right) \eta_0(\omega)\bigg| = |(\cos(rH(g^{-1}\omega)) - 1)|  |e^{-\delta H(g^{-1}\omega)} \eta_0(g^{-1}\omega)|.$$ Then \eqref{PsimuTrickDistance} follows by using that $|(\cos(rH(g^{-1}\omega)) - 1)| \ll |r|^2 ||g||^2$.
\end{proof}

\subsection{High Frequency Estimate}\label{SectionHighFrequency}

For a Schwartz function $f \in \Schwartz{X}$, we say that the Fourier transform $\widehat{f}: \mathfrak{a} \times K \to \C$ has compact support if there is $R \geq 0$ such that $\widehat{f}(r,\omega) = 0$ for $r \geq R$ and all $\omega \in \Omega$. In this section with make no notational difference between a function $f \in \Schwartz{X}$ and its $G$-lift. We first prove a preliminary lemma on the Fourier transform.

\begin{lemma}\label{FourierL1Bound}
	For $f \in \Schwartz{X}$, $$||\widehat{f}(r,\cdot)||_{L^2(\Omega)} \leq ||f||_1$$
\end{lemma}

\begin{proof}
	We calculate for $r\in \mathfrak{a}$ and $\omega \in \Omega$ that
	\begin{align*}
		|\widehat{f}(r,\omega)|^2 &= \bigg| \int_G f(g) (\rho_{-r}(g)1)(\omega) \, d\Haarof{G}(g) \bigg|^2 \\ &\leq \bigg| \int_G |f(g)| \, |(\rho_{-r}(g)1)(\omega)| \, d\Haarof{G}(g)  \bigg|^2  \\
		&\leq \bigg| \int_G |f(g)| \, |(\rho_0(g)1)(\omega)| \, d\Haarof{G}(g)  \bigg|^2. 
	\end{align*}
	Set $f_1  = \frac{|f|}{||f||_1}$ so that it follows that $$|\widehat{f}(r,\omega)|^2 \leq ||f||_1^2 \cdot  \bigg|  \int_G  (\rho_0(g)1)(\omega)   \, f_1(g) \, d\Haarof{G}(g)   \bigg|^2.$$ Recall that if $X$ is a random variable on a probability space then by Jensen's inequality $E[X]^2 \leq E[X^2]$. By construction $f_1 d\Haarof{G}$ is a probability measure and hence it follows that 
	\begin{align*}
		|\widehat{f}(r,\omega)|^2 &\leq ||f||_1^2 \int_G (\rho_0(g)1)(\omega)^2 \,  f_1(g)\, d\Haarof{G}(g) \\
		&\leq ||f||_1^2 \int_G \frac{d(\alpha_g)_*\Haarof{\Omega}}{d\Haarof{\Omega}}(\omega) \,  f_1(g)\, d\Haarof{G}(g)
	\end{align*}
	Thus we conclude that 
	\begin{align*}
		||\widehat{f}(r,\cdot)||_2^2 &\leq ||f||_1^2 \int_G \left( \int_{\Omega} \frac{d(\alpha_g)_*\Haarof{\Omega}}{d\Haarof{\Omega}}(\omega) \, d\Haarof{\Omega}(\omega) \right) \, f_1(g) \, d\Haarof{G}(g) \\
		&\leq ||f||_1^2.
	\end{align*}
\end{proof}

\begin{lemma}\label{HighFrequencyFourierCompactlySupported}
	Let $\mu$ be a non-degenerate probability measure on $G$ assume that $S_0$ is quasicompact and let $\delta_0 \in (0,1)$ be the constant from Proposition~\ref{S_rStrongSpectralGap}. Let $R \geq 1$ and let $f\in \Schwartz{X}$ be a Schwartz function whose Fourier transform satisfies $\widehat{f}(r,\omega) = 0$ for all $|r| \geq R$ and $\omega \in \Omega$. Then there is $c_R = c_R(\mu) >0$ depending on $\mu$ and $R$ such that for $n \geq 1$, $$\bigg|\frac{n^{\ell/2}}{\sigma^n} \int_{|r| \geq \delta_0} \int_{\Omega} \widehat{f}(r, \omega) \, (S_r^n\rho_r(h_0)1)(\omega) \, d\Haarof{\Omega}(\omega)d\SphericalPlancharelMeasure(r)\bigg| \ll_{\mu} R^{\dim X} e^{-c_R n}||f||_1.$$
\end{lemma}

\begin{proof}
	Choose $R$ such that $\widehat{f}(r, \omega) = 0$ for $r \geq R$ and $\omega \in \Omega$. Then using Cauchy-Schwarz and Lemma~\ref{HighFreqeuncyNormLess},
	\begin{align*}
		&\bigg|\frac{n^{\ell/2}}{\sigma^n} \int_{\delta_0 \leq |r| \leq R} \int_{\Omega} \widehat{f}(r, \omega) \, (S_r^n\rho_r(h_0)1)(\omega) \, d\Haarof{\Omega}(\omega)d\SphericalPlancharelMeasure(r)\bigg| \\
		&\leq \frac{n^{\ell/2}}{\sigma^n} \int_{\delta_0 \leq |r| \leq R} ||\widehat{f}(r, \cdot)||_{L^2(\Omega)} ||S_r^n\rho_r(h_0)1||_2 \, d\SphericalPlancharelMeasure(r) \\
		&\leq \frac{n^{\ell/2}}{\sigma^n} \sup_{\delta_0 \leq |r| \leq R} ||S_r^n|| \int_{1 \leq |r| \leq R} ||\widehat{f}(r, \cdot)||_{L^2(\Omega)} \, d\SphericalPlancharelMeasure(r) \\
		&\leq e^{-c_Rn} \int_{\delta_0 \leq |r| \leq R} ||\widehat{f}(r, \cdot)||_{L^2(\Omega)} \, d\SphericalPlancharelMeasure(r)  \\
		&\ll_{\mu} e^{-c_R n}||f||_1  \int_{|r| \leq R} |c(r)|^{-2} \, d\Haarof{\mathfrak{a}^{*}}(r) \\
		&\ll_{\mu} e^{-c_R n}||f||_1  \int_{|r| \leq R} (1 + |r|^{\dim N}) \, d\Haarof{\mathfrak{a}^{*}}(r) \ll_{\mu} R^{\dim X} e^{-c_R n}||f||_1,
	\end{align*} using \eqref{SrPowerOperatorNormEstimate} in order to choose a constant $c_R > 0$ depending on $\mu$ and $R$ such that $\left( \frac{n^{\ell/2}}{\sigma^n} \sup_{\delta_0 \leq |r| \leq R} ||S_r^n||\right) \leq e^{-c_Rn}$ for $n$ large enough and Proposition 7.2 of chapter IV in \cite{HelgasonGroupsBook}, asserting that $|c(r)|^{-2} \ll 1 + |r|^{\dim N}$ for any $r \in \mathfrak{a}^{*}$
\end{proof}

Towards proving Theorem~\ref{StrongLCLT}, we strengthen Lemma~\ref{HighFrequencyFourierCompactlySupported} under strong assumptions on $||S_r||$.

\begin{lemma}\label{HighFrequencyUnderS_rStrong}
	Let $\mu$ be a non-degenerate probability measure on $G$. Assume that $S_0$ is quasicompact and that $\left(\sup_{|r| \geq 1} ||S_r||\right) < ||S_0||.$ Let $\delta_0$ be the constant from Proposition~\ref{S_rStrongSpectralGap}. Then for $f \in \Schwartz{X}$, $s > \frac{1}{2}\dim X$  and $n \geq 1$,
	$$\bigg|\frac{n^{\ell/2}}{\sigma^n} \int_{|r| \geq \delta_0} \int_{\Omega} \widehat{f}(r, \omega) \, (S_r^n\rho_r(h_0)1)(\omega) \, d\Haarof{\Omega}(\omega)d\SphericalPlancharelMeasure(r)\bigg| \ll_{\mu,s} e^{-cn}||f||_{H^s}.$$
\end{lemma}

\begin{proof}
	The left hand side of the claimed equation is bounded by
	\begin{align*}
		&\leq \frac{n^{\ell/2}}{\sigma^n} \int_{|r| \geq \delta_0} ||\widehat{f}(r, \cdot)||_{L^2(\Omega)} ||S_r^n\rho_r(h_0)1||_2 \, d\SphericalPlancharelMeasure(r) \\
		&\leq e^{-cn} \int_{|r| \geq \delta_0} ||\widehat{f}(r, \cdot)||_{L^2(\Omega)} |r|^s |r|^{-s} \, d\SphericalPlancharelMeasure(r) \\
		&\leq  e^{-cn}\sqrt{\int_{|r| \geq \delta_0} |r|^{-2s} \, d\SphericalPlancharelMeasure(r)}  \sqrt{\int_{|r| \geq \delta_0} ||\widehat{f}(r, \cdot)||_{L^2(\Omega)}^2 |r|^{2s}  \, d\SphericalPlancharelMeasure(r)} \\ &\ll_{\delta_0,s} e^{-cn}||f||_{H^s},
	\end{align*} for $n$ large enough and choosing $s$ sufficiently large such that $\int_{|r| \geq 1} |r|^{-2s} \, d\SphericalPlancharelMeasure(r)$ is bounded. Indeed, by Proposition 7.2 of chapter IV in \cite{HelgasonGroupsBook}, it holds that $|c(r)|^{-2} \ll 1 + |r|^{\dim N}$ for any $r \in \mathfrak{a}^{*}$ and therefore $|c(r)|^{-2} \ll_{\delta_0} |r|^{\dim N}$ for $|r| \geq \delta_0$. Thus $\int_{|r| \geq \delta_0} |r|^{-2s} \, d\SphericalPlancharelMeasure(r) \ll_{\delta_0} \int_{|r| \geq \delta_0} |r|^{\dim N - 2s} \, d\Haarof{\mathfrak{a}^{*}}(r)$ and the latter term is $< \infty$ whenever $\dim N - 2s < -\dim A$. 
\end{proof}

\subsection{Low Frequency Estimate}\label{SectionLowFrequency}

Throughout this section we assume that $S_0$ is quasicompact and denote by $\delta_0 \in (0,1)$ the constant from Proposition~\ref{S_rStrongSpectralGap}. In this section we deal with the some preliminary estimates for the frequency range $|r|\leq \delta_0$. We recall that by Proposition~\ref{S_rStrongSpectralGap} for $|r| \leq \delta_0$ we have a decomposition $$S_r = \lambda(r)E_r + D_r,$$ where $E_r$ and $D_r$ satisfy the properties of Definition~\ref{StrongSpectralGapDef}. We first show that we can ignore the contribution of $D_r$.

\begin{lemma}
	Let $\mu$ be a non-degenerate probability measure on $G$ and assume that $S_0$ is quasicompact. There exists a constant $c > 0$ depending on $\mu$ such that for all $f\in \Schwartz{X}$ and $h_0 \in G$,  
	$$\bigg|\frac{n^{\ell/2}}{\sigma^n}\int_{|r|\leq \delta_0} \int_{\Omega}  \widehat{f}(r,\omega) (D_{r}^n \rho_r(h_0) 1)(\omega) \, d\Haarof{\Omega}(\omega)d\SphericalPlancharelMeasure(r)\bigg| \ll ||f||_{1} e^{-cn}.$$
	\label{LowFreqProj}
\end{lemma}

\begin{proof}
	Using Proposition~\ref{S_rStrongSpectralGap}, we deduce $\frac{n^{\ell/2}}{\sigma^n} \sup_{|r|\leq \delta_0}  ||D_r^n \rho_r(h_0) 1|| \ll e^{-cn}$ for $c > 0$ a constant depending on $\mu$. Using Cauchy-Schwarz the term in question is bounded by $$ \frac{n^{\ell/2}}{\sigma^n}\int_{|r|\leq \delta_0} ||\widehat{f}(r,\cdot)||_{L^2(\Omega)} \, ||D_r^n\rho_r(h_0) 1||_2 \, d\nu_{\mathrm{sph}}(r).$$ The lemma follows as $||\widehat{f}(r,\cdot)||_{L^2(\Omega)} \leq ||f||_{1}$ by Lemma~\ref{FourierL1Bound} and by estimating $\int_{|r|\leq \delta_0} 1 \, d\nu_{\mathrm{sph}}(r) \ll 1$ since $\delta_0 \leq 1$.
\end{proof}

Therefore, up to an exponential error term, we only need to deal with
\begin{equation} \label{RemainingTerm}
	\frac{n^{\ell/2}}{\sigma^n}\int_{|r|\leq \delta_0} \lambda(r)^n \int_{\Omega}  \widehat{f}(r,\omega) (E_{r} \rho_r(h_0) 1)(\omega) \, d\Haarof{\Omega}(\omega)d\SphericalPlancharelMeasure(r).
\end{equation} 

Recall that $\ell = 2p + d$ for $d$ the rank of $G$, where the rank is defined as the real dimension of $\mathfrak{a}$.  We therefore may rewrite \eqref{RemainingTerm} by replacing $r$ by $\frac{r}{\sqrt{n}}$ as
\begin{equation}\label{RemainingTermFubini}
	\frac{n^{p}}{\sigma^n} \int_{|r| \leq \delta_0 \sqrt{n}} \lambda(\tfrac{r}{\sqrt{n}})^n|c(\tfrac{r}{\sqrt{n}})|^{-2} \int_{\Omega}  \widehat{f}(\tfrac{r}{\sqrt{n}},\omega) (E_{\frac{r}{\sqrt{n}}} \rho_{\tfrac{r}{\sqrt{n}}}(h_0) 1)(\omega) \, d\Haarof{\Omega}(\omega)d\Haarof{\mathfrak{a}^{*}}(r).
\end{equation}

Towards proving the local limit theorem, we first replace $\frac{\lambda(r/\sqrt{n})^n}{\sigma^n}$ by a suitable function. Before doing so we give some elementary calculative results.

\begin{lemma}
	The following inequalities hold:
	\begin{enumerate}
		\item[(i)] For any $A,B \in \mathbb{R}$,  $$|e^A - e^B|\leq |A-B|\max\{ e^A, e^B \}$$
		\item[(ii)] For any $c > 0, r \neq 0$ and $n \geq 1$, $$ne^{-cnr^2} \leq \frac{2}{c} e^{-cnr^2/2}r^{-2}.$$
	\end{enumerate}
	\label{CalcLemma}
\end{lemma}

\begin{proof}
	For the first inequality by assuming without loss of generality that $A\geq B$ we deduce that $|e^A - e^B|\leq e^A |1 - e^{B-A}|$ and hence reduce to showing that $|1 - e^{B-A}| \leq |A-B|$. For this we use that $e^{x} \geq 1 + x$ and hence as $B-A$ is negative, $|1 - e^{B-A}| = 1 - e^{B-A} \leq -(B-A) = |A-B|$. 
	
	For the second inequality we apply the observation that $e^{-x} \leq \frac{1}{x}$ to deduce that $ne^{-cnr^2/2} \leq n\frac{2}{cnr^2} = \frac{2}{cr^2}$ which implies the claim by multiplication with $e^{-cnr^2/2}$. 
\end{proof}

\begin{lemma}\label{s(r)Estimate}
	Assume that $\mu$ has finite fourth moment. There are constants $c_2, c^{*} > 0$ and a positive definite sesquilinear form $Q$ on $\mathfrak{a}$ such that for $|r| \leq \delta_0$,  $$\bigg| \frac{\lambda(r)^n}{\sigma^n} - e^{-c_2nQ(r,r)}  \bigg| \ll_{\mu}  e^{-c^{*}n|r|^2} |r|^2.$$ In particular, for $|r| \leq \delta_0 \sqrt{n}$,  $$\bigg| \frac{\lambda(r/\sqrt{n})^n}{\sigma^n} - e^{-c_2Q(r,r)}  \bigg| \ll_{\mu}  n^{-1}e^{-c^{*}|r|^2} |r|^2.$$
\end{lemma}

\begin{proof}
	As in the proof of Proposition~\ref{s(r)behavioursmallr} one shows that $\lambda(r)$ is $C^4$ if $\mu$ has finite fourth moment. Indeed, by conducting a Taylor expansion of $\lambda$, for small $r$, $$\lambda(r) = \lambda(0) - Q(r,r)  + O_G(|r|^4),$$ where $Q(r,r) = -H_{\lambda,0}(r,r)/2$ for $H_{\lambda,0}$ the Hessian of $\lambda$ at $0$. By Proposition~\ref{s(r)behavioursmallr} the sesquilinear form $Q$ is positive definite. Moreover, we may choose for small enough $r$ a constant $c_{*} > 0$ such that $|\lambda(r)| \leq \lambda(0)(1 - c^{*}|r|^2).$  Using that $\ln(1 + x) \leq x$, it therefore follows that $$n\ln(\tfrac{\lambda(r)}{\lambda(0)}) \leq -c^{*}n|r|^2.$$
	Throughout set $c_2 = \frac{1}{\lambda(0)}$ and choose $c^{*} \leq c_2$. Then $$\max\{ e^{-c_2nQ(r,r)}, e^{n\ln(\frac{\lambda(r)}{\lambda(0)}) } \} \leq e^{-c^{*}n|r|^2}.$$ Using Lemma~\ref{CalcLemma} (i) it follows that 
	\begin{align*}
		| \tfrac{\lambda(r)^n}{\lambda(0)^n}   - e^{-c_2nQ(r,r)}|  &= |e^{n\ln(\frac{\lambda(r)}{\lambda(0)})}  - e^{-c_2nQ(r,r)}| \\ &\leq \max\{ e^{-c_2nQ(r,r)}, e^{n\ln(\frac{\lambda(r)}{\lambda(0)})} \} |n\ln(\tfrac{\lambda(r)}{\lambda(0)}) + c_2nQ(r,r)| \\&\ll e^{-c^{*}nQ(r,r)} n|r|^4 \\
		&\ll e^{-c^{*}nQ(r,r)}|r|^2,
	\end{align*} by using Lemma~\ref{CalcLemma} (ii) in the last line by changing the constant $c^{*}$.
\end{proof}

Recall by the definition of the $c$-function:
\begin{align}
	|c(r)|^{-2}
	&= \frac{1}{I(\delta)} \left(\prod_{\ell = 1}^p \bigg| B\left( \frac{m(r_i)}{2}, \frac{i \langle r, r_\ell \rangle}{\langle r_\ell , r_\ell \rangle} \right)  \bigg| \right)^{-2} \cdot \left(\prod_{\ell = p + 1}^k  \bigg| B\left( \frac{m(r_\ell)}{2}, \frac{m(r_\ell/2)}{4}  +  \frac{i\langle r, r_\ell \rangle}{\langle r_\ell , r_\ell \rangle} \right) \bigg| \right)^{-2} \nonumber \\
	&= \frac{1}{I(\delta)} \left(\prod_{\ell = 1}^p \frac{|\Gamma(\tfrac{m(r_\ell)}{2} + \tfrac{i\langle r, r_\ell \rangle}{\langle r_\ell, r_\ell \rangle})|^2}{|\Gamma(\tfrac{m(r_\ell)}{2})|^2|\Gamma(\tfrac{i\langle r, r_\ell \rangle}{\langle r_\ell, r_\ell \rangle})|^2} \right) \cdot \left(\prod_{\ell = p + 1}^k  \frac{|\Gamma(\tfrac{m(r_\ell)}{2} + \tfrac{m(r_\ell/2)}{4} + \tfrac{i\langle r, r_\ell \rangle}{\langle r_\ell, r_\ell \rangle})|^2}{|\Gamma(\tfrac{m(r_\ell)}{2})|^2|\Gamma(\tfrac{m(r_\ell/2)}{4} + \tfrac{i\langle r, r_\ell \rangle}{\langle r_\ell, r_\ell \rangle})|^2}   \right), \label{c(r)GammaExpression}
\end{align}where $B(x,y) = \int_0^1 t^{x-1}(1- t)^{y-1} \, dt$ is the Beta function satisfying $B(x,y) = \frac{\Gamma(x)\Gamma(y)}{\Gamma(x+y)}$.

\begin{lemma}\label{c(r)Estimate}
	There is a constant $c_G$ depending only on $G$ such that for $|r| \leq \delta_0$, $$|c(r)|^{-2} = c_G \prod_{\ell = 1}^p |\langle r, r_\ell \rangle|^2 + O\left( |r|^{2} \right).$$ In particular, for $|r| \leq \delta_0 \sqrt{n}$, $$\bigg|n^p|c(\tfrac{r}{\sqrt{n}})|^{-2} - c_G \prod_{\ell = 1}^p |\langle r, r_\ell \rangle|^2\bigg| \ll n^{-1}|r|^{2}.$$
\end{lemma}

\begin{proof}
	As the singularities of the $\Gamma$ function are at $0,-1,-2, \ldots$ and $\Gamma(z)$ behaves around $0$ like $\frac{1}{z}$, it holds that $| \frac{1}{|\Gamma(ix)^2|} - x^2| \ll x^4$ and $|\Gamma(\tfrac{n}{2} +ix)^2 - \Gamma(\tfrac{n}{2})^2| \ll x^2$. Therefore, $$\bigg| \frac{|\Gamma(\tfrac{m(r_\ell)}{2} + \tfrac{i\langle r, r_\ell \rangle}{\langle r_\ell, r_\ell \rangle})|^2}{|\Gamma(\tfrac{m(r_\ell)}{2})|^2|\Gamma(\tfrac{i\langle r, r_\ell \rangle}{\langle r_\ell, r_\ell \rangle})|^2} - \frac{|\Gamma(\tfrac{m(r_\ell)}{2} )|  \, |\langle r,r_{\ell} \rangle|^2}{|\Gamma(\tfrac{m(r_\ell)}{2})|^2\, |\langle r_{\ell}, r_{\ell} \rangle|^2}   \bigg| \ll |r|^2$$ and similarly 
	$$\bigg|\frac{|\Gamma(\tfrac{m(r_\ell)}{2} + \tfrac{m(r_\ell/2)}{4} + \tfrac{i\langle r, r_\ell \rangle}{\langle r_\ell, r_\ell \rangle})|^2}{|\Gamma(\tfrac{m(r_\ell)}{2})|^2|\Gamma(\tfrac{m(r_\ell/2)}{4} + \tfrac{i\langle r, r_\ell \rangle}{\langle r_\ell, r_\ell \rangle})|^2}  - \frac{|\Gamma(\tfrac{m(r_\ell)}{2} + \tfrac{m(r_\ell/2)}{4})|^2}{|\Gamma(\tfrac{m(r_\ell)}{2})|^2|\Gamma(\tfrac{m(r_\ell/2)}{4} )|^2}    \bigg| \ll |r|^2.$$ Using these two estimates in \eqref{c(r)GammaExpression} the lemma follows for a suitable constant $c_G$.
\end{proof}

Denote by $$\gamma(r) = c_Ge^{-c_2Q(r,r)}\prod_{\ell = 1}^p |\langle r,r_{\ell} \rangle|^2$$ for $c_G$ the constant from Lemma~\ref{c(r)Estimate}. We then may draw the following corollary.

\begin{corollary}\label{s(r)c(r)EstimateCombined} Assume that $\mu$ has finite fourth moment. For $|r| \leq \delta_0\sqrt{n}$ and $c' > 0$ a constant depending on $\mu$,
	$$\bigg| \frac{n^{p}}{\sigma^n}  \lambda(\tfrac{r}{\sqrt{n}})^n|c(\tfrac{r}{\sqrt{n}})|^{-2} - \gamma(r) \bigg| \ll_{\mu}  n^{-1} e^{-c'|r|^2}.$$
\end{corollary}

\begin{proof}
	Combining Lemma~\ref{s(r)Estimate} and Lemma~\ref{c(r)Estimate}, for a suitable constant $c' > 0$, \begin{align*}
		\bigg| \frac{n^{p}}{\sigma^n}  \lambda(\tfrac{r}{\sqrt{n}})^n|c(\tfrac{r}{\sqrt{n}})|^{-2} - \gamma(r) \bigg| &\leq \bigg| \frac{\lambda(\tfrac{r}{\sqrt{n}})^n}{\sigma^n}  -e^{-c_2Q(r,r)} \bigg| |n^pc(\tfrac{r}{\sqrt{n}})^{-2}| \\ &+  \bigg|n^p|c(\tfrac{r}{\sqrt{n}})|^{-2} - c_G \prod_{\ell = 1}^p |\langle r, r_\ell \rangle|^2\bigg| e^{-c_2Q(r,r)} \\
		&\ll_{\mu} n^{-1} e^{-c'|r|^2},
	\end{align*}  using that $|c(\tfrac{r}{\sqrt{n}})|^{-2} - \gamma(r)| \ll_{\mu} |r|^{O(1)}$ which equally follows by Lemma~\ref{c(r)Estimate}.
\end{proof}

\subsection{Proof of Theorem~\ref{LCLT} and Theorem~\ref{StrongLCLT}}\label{SectionStrongLCLT}

Throughout this section assume that $\mu$ has finite fourth moment. We are now in a suitable position to prove Theorem~\ref{LCLT} and Theorem~\ref{StrongLCLT}. Let $f \in \Schwartz{X}$. Recall that we expressed in \eqref{LLTFourierSide}  the term in question $\frac{n^{\ell/2}}{\sigma^n}\int f(g.x_0) \, d\mu^{*n}(g)$ for $x_0 = h_0K$ by using the Fourier inversion formula as $$\frac{n^{\ell/2}}{\sigma^n} \int_{\mathfrak{a}^{*}} \int_{\Omega}   \widehat{f}(r,\omega) (S_r^n \rho_r(h_{0})1)(\omega) \, d\Haarof{\Omega}(\omega)d\SphericalPlancharelMeasure(r).$$ The latter term is decomposed into the high frequency \eqref{HighFrequency} and low frequency \eqref{LowFrequency} component for $\delta_0 \in (0,1)$ small enough such that Lemma~\ref{S_rStrongSpectralGap} holds. Under the assumption $\sup_{|r| \geq 1} ||S_r||$, the high frequency term \eqref{HighFrequency} is dealt with by Lemma~\ref{HighFrequencyUnderS_rStrong} collecting an error term of size $O_{\mu}(e^{-cn}||f||_{H^s})$ for $s = \frac{1}{2}(\dim X + 1)$. Without this assumption, one requires that the Fourier transform of $f$ is compactly supported yielding by Lemma~\ref{HighFrequencyFourierCompactlySupported} an error term of size $O_{\mu,f}(e^{-c_f n}||f||_1)$.

For the low frequency term, one applies Lemma~\ref{LowFreqProj}, thereby collecting an error term of size $O_{\mu}(e^{-cn}||f||_1)$. It remains to deal with \eqref{RemainingTerm}, which after the substitution $r$ to $\frac{n}{\sqrt{r}}$ is of the form \eqref{RemainingTermFubini}. Using Lemma~\ref{MainEstaimteNormonVell} and Corollary~\ref{s(r)c(r)EstimateCombined}, we arrive at the term 
\begin{align*}
	&\int_{|r| \leq \delta_0 \sqrt{n}} \gamma(r) \int_{\Omega}  \widehat{f}(\tfrac{r}{\sqrt{n}},\omega) (E_{\frac{r}{\sqrt{n}}} \rho_{\tfrac{r}{\sqrt{n}}}(h) 1)(\omega) \, d\Haarof{\Omega}(\omega)d\Haarof{\mathfrak{a}^{*}}(r) \\
	&= \int f(g)\rho_G(h^{-1}) \left( \int_{|r| 
		\leq \delta_0\sqrt{n}} \gamma(r) \psi_{\mu, \frac{r}{\sqrt{n}}}(g) \, d\Haarof{\mathfrak{a}^{*}}(r)\right) d\Haarof{G}(g)
\end{align*} admitting an additional error term of size $$\ll_{\mu} n^{-1} ||f||_1\int_{|r| \leq \delta_0 \sqrt{n}} e^{-c'|r|^2} \, d\Haarof{\mathfrak{a}^{*}}(r) \ll_{\mu} n^{-1} ||f||_1,$$ using that the latter integral converges.

We define for $n \geq 1$ the continuous real-valued functions on $G$, 
\begin{align}
	\psi_{n}(g) &= \int_{|r| 
		\leq\delta_0 \sqrt{n}} \gamma(r) \psi_{\mu, \frac{r}{\sqrt{n}}}(g) \, d\Haarof{\mathfrak{a}^{*}}(r) \quad\quad \text{and}  \label{DefPsin} \\
	\psi_0(g) &= c_{\mu} \cdot \psi_{\mu,0}(g) \quad\quad\text{for}\quad\quad c_{\mu}  = \int_{r \in \mathfrak{a}^{*}} \gamma(r) \, d\Haarof{\mathfrak{a}^{*}}(r).\nonumber
\end{align} While $\psi_{\mu,\frac{r}{\sqrt{n}}}$ is not necessarily real-valued, the function $\psi_n$ is as $\overline{\psi_{\mu,\frac{r}{\sqrt{n}}}} = \psi_{\mu,-\frac{r}{\sqrt{n}}}$ and the definition of $\psi_n$ is invariant under $r \mapsto -r$.

We have so far collected a total error of size $$O_{\mu}(n^{-1}||f||_1 + e^{-cn}||f||_{H^s})$$ under the assumption $\sup_{|r| \geq 1} ||S_r|| < ||S_0||$ and for $f\in \mathscr{S}(X)$ and $$O_{\mu}(n^{-1}||f||_1)  + O_{\mu, f}(e^{-c_fn}||f||_1)$$ without the latter assumption yet requiring that the Fourier transform of $f$ has compact support.  To conclude the proof, we show the following lemma. 

\begin{lemma}\label{Psi_nToPsi}
	For $g \in G$ and $n \geq 1$, $$|\psi_{n}(g) - \psi_0(g)| \ll_{\mu} n^{-1}(1 + ||g||^2).$$ 
\end{lemma} 

\begin{proof}
	Since $\gamma(r) \ll_{\mu} e^{-c'|r|^2}$ for a suitable constant $c'$ it follows that $$|\psi_{\mu,0}(g)|\int_{|r| > \delta_0 \sqrt{n}} \gamma(r) \, d\Haarof{\mathfrak{a}^{*}}(r)$$ decays exponentially fast in $n$ (using that $|\psi_{\mu,0}(g)| = |\langle  \eta_0, \rho_0(g)\eta_0' \rangle| \ll_{\mu} 1$) and therefore we need to deal with 
	\begin{equation}\label{Psi_nToPsiRemainingTerm}
		\bigg|\psi_{n}(g) -   \int_{|r| 
			\leq \delta_0 \sqrt{n}} \gamma(r) \psi_{\mu,0}(g) \, d\Haarof{\mathfrak{a}^{*}}(r)  \bigg|.
	\end{equation}
	
	By Lemma~\ref{PsimurEstimate} it holds that $$\bigg|\frac{\psi_{\mu,\frac{r}{\sqrt{n}}}(g) + \psi_{\mu,-\frac{r}{\sqrt{n}}}(g)}{2} - \psi_{\mu,0}(g)\bigg| \ll_{\mu} n^{-1}|r|^2(1 + ||g||^2)$$ and therefore using again that $\gamma(r) \ll_{\mu} e^{-c^{*}|r|^2}$ and as the defining integral of $\psi_{n}$ is invariant under replacing $r$ by $-r$,
	\begin{align*}
		\eqref{Psi_nToPsiRemainingTerm} &\ll \int_{|r| 
			\leq \delta_0 \sqrt{n}} \gamma(r) |\psi_{\mu,\frac{r}{\sqrt{n}}}(g) - \psi_{\mu,0}(g)| \, d\Haarof{\mathfrak{a}^{*}}(r)   \\
		&\ll_{\mu} n^{-1}(1 + ||g||^2)\int_{|r| \leq  \delta_0 \sqrt{n}} \gamma(r)|r|^2 \, d\Haarof{\mathfrak{a}^{*}}(r) \\
		&\ll_{\mu} n^{-1}(1 + ||g||^2).
	\end{align*}
\end{proof}

Recall that we have defined $$||f||_{*} = \int |f(x)|(1 + d_X(x,o)^2) \, d\Haarof{X}(x) = \int |f(g)|(1 + ||g||^2) \, d\Haarof{G}(g),$$ where we make no notational difference between $f$ and its lift to $G$. To conclude the proof of \eqref{LCLTQuantitativeFourierCompact} and \eqref{StrongLCLTFormula} we estimate
\begin{align*}
	&\bigg| \int f(g.x_0) \psi_n(g) \, d\Haarof{G}(g) - \int f(g.x_0) \psi_0(g) \, d\Haarof{G}(g) \bigg| \\ &\leq \int |f(g)| |\psi_n(gh_0^{-1}) - \psi_0(gh_0^{-1})| \, d\Haarof{G}(g) \\
	&\ll_{\mu} n^{-1} \int |f(g)| (1 + ||gh_0^{-1}||^2) \, d\Haarof{G}(g) \\
	&\ll_{\mu} n^{-1} \int |f(g)| (1 + ||g||^2 + ||h_0||^2) \, d\Haarof{G}(g) \\
	&\ll_{\mu} n^{-1}||f||_{*} + n^{-1}d_X(x_0,o)^2 ||f||_1, 
\end{align*} having used in the penultimate line that $||gh_0^{-1}|| \leq ||g|| + ||h_0^{-1}||$ by Corollary 7.20 of \cite{BenoistQuintRandomBook} as $G$ is connected. This concludes the proof of Theorem~\ref{LCLT} and of \eqref{StrongLCLTFormula}. The final claim of Theorem~\ref{StrongLCLT} is proved in the following lemma.

\begin{lemma}\label{SpreadOurKinvSpectral}
	Let $G$ be a non-compact connected semisimple Lie group with finite center and let $\mu$ be a  non-degenerate probability measure on $G$ with finite second moment. Assume that $\mu$ satisfies one of the following properties:
	\begin{enumerate}
		\item[(i)] $\mu$ is spread out.
		\item[(ii)] $\mu$ is bi-$K$-invariant, i.e. $\Haarof{K} * \mu * \Haarof{K}$. 
	\end{enumerate} Then $S_0$ is quasicompact and $\left( \sup_{|r| \geq 1} ||S_r|| \right) < ||S_0||$.
\end{lemma}

\begin{proof}
	The claim of the lemma was established for spread out measures in Section 2.2 of \cite{Bougerol1981}. It remains to treat the case where $\mu$ is bi-$K$-invariant. Note that as $S_r = \rho_r(\Haarof{K}) * S_r * \rho_r(\Haarof{K})$ it holds that $S_r 1_{\Omega} = \lambda(r)1_{\Omega}$ and $S_r \langle 1_{\Omega}\rangle^{\perp} = \{ 0 \}$ and therefore $\lambda(r) = \int \phi_r(g) \, d\mu(g)$. The claim now follows as $\phi_r(g) \to 0$ (cf. for example appendix A of \cite{FinisMatz2021}) for fixed $g \in G \backslash K$ and $r \to \infty$ and using that $\mu(G \backslash K) > 0$ as $\mu$ is non-degenerate. 
\end{proof}

\subsection{Proof of Theorem~\ref{QuasicompactnessImpliesLLT}}\label{SectionLLT}

\begin{lemma}\label{LLTSchwartzCompactSupport}
	Let $G$ and $\mu$ be as in Theorem~\ref{QuasicompactnessImpliesLLT}. Let $f\in \Schwartz{X}$ be a Schwartz function whose Fourier transform is compactly supported. Then  $$\lim_{n \to \infty}\frac{n^{\ell/2}}{\sigma^n}\int f(g.x_0) \, d\mu^{*n}(g) = \int f(g.x_0) \psi_{0}(g) \, d\Haarof{G}(g).$$ 
\end{lemma}

\begin{proof}
	The proof is as the one of Theorem~\ref{StrongLCLT} expect that we cannot use Lemma~\ref{s(r)Estimate}. Revising the argument of Lemma~\ref{s(r)Estimate}, it follows that for the positive definite quadratic form $Q$ from Lemma~\ref{s(r)Estimate}, under the assumption that $\mu$ has finite second moment, it holds that $\lambda(r) = \lambda(0) - Q(r,r) + o(|r|^2)$ and therefore for $|r| \leq \delta_0 \sqrt{n}$, $$\lim_{n \to \infty}  \frac{\lambda(r/\sqrt{n})^n}{\sigma^n} = e^{-c_2Q(r,r)} \quad\quad \text{and} \quad\quad   \frac{\lambda(r/\sqrt{n})^n}{\sigma^n}  \ll e^{c'|r|^2}$$ for a suitable constant $c'>0$. Similarly to Lemma~\ref{s(r)c(r)EstimateCombined}, $$\lim_{n \to \infty} \frac{n^p}{\sigma^n} \lambda(\tfrac{r}{\sqrt{n}})^n |c(\tfrac{r}{\sqrt{n}})|^{-2} = \gamma(r).$$ Arguing as in the proof of Theorem~\ref{LCLT}, it therefore follows by dominated 
	convergence,
	\begin{align*}
		&\lim_{n \to \infty}\frac{n^{\ell/2}}{\sigma^n}\int f(g.x_0) \, d\mu^{*n}(g)   = \lim_{n \to \infty} \eqref{RemainingTerm} \\
		&= \int_{r \in \mathfrak{a}^{*}} \gamma(r) \int_{\Omega} \widehat{f}(0,\omega) (E_{0} \rho_0(h_0)1)(\omega) \, d\Haarof{\Omega}(\omega) d\Haarof{\mathfrak{a}^{*}}(r) \\
		&= \int f(g.x_0) \psi_0(g) \, d\Haarof{G}(g).
	\end{align*}
\end{proof}

\begin{lemma}\label{LLTBound}
	Let $f \in \Schwartz{X}$. Then $$\limsup_{n \to \infty}\bigg|\frac{n^{\ell/2}}{\sigma^n} \int f(g.x_0) \, d\mu^{*n}(g)\bigg|  \ll ||f||_1,$$ where the implied constant depends only on $G$.
\end{lemma}

\begin{proof}
	One may reduce to functions $f \geq 0$. By covering the latter function suitably by a linear combination of characteristic functions, it suffices to show the claim for $f = 1_{B_{\varepsilon}(x)}$ with $\varepsilon > 0$ small and $x \in X$.  By Theorem 5.7 of \cite{Andersen2004} there is a positive function $h \in \mathscr{S}(X)$, whose Fourier transform has compact support, satisfying $1_{B_{\varepsilon}(x)} \leq h$ and $||h||_1 \ll \mathrm{vol}_X(B_{\varepsilon})$. The lemma follows by applying Lemma~\ref{LLTSchwartzCompactSupport} to $h$.
\end{proof}

\begin{proof}(of Theorem~\ref{QuasicompactnessImpliesLLT})
	Let $\delta_{\ell} \in \mathscr{S}(X)$ be an approximation to the identity on $G$ that is bi-$K$-invariant and whose Fourier transform has compact support. Such functions exists by choosing a sequence $\omega_{\ell}$ of smooth bi-$K$-invariant approximations to the identity that are supported on smaller and smaller balls around $e\in G$. As a Schwartz function is characterized by its Fourier transform, it suffices to determine $\widehat{\delta_{\ell}}.$ Indeed one may choose $\widehat{\delta_{\ell}}$ to be equal to $\widehat{\omega_{\ell}}$ in a sufficiently large ball around the identity and to decay to zero rapdily outside of it. One then readily checks that $\delta_{\ell}$ satisfies the required properties. 
	
	Then for $f \in \Schwartz{X}$, it holds for $r \in \mathfrak{a}^{*}$ and $k \in \Omega$,  
	\begin{align*}
		\widehat{f * \delta_{\ell}}(r,k) = (\rho_{-r}(f * \delta_{\ell})1)(k) = (\rho_{-r}(f)\rho_{-r}(\delta_{\ell})1)(k) = \widehat{f}(r,k)\widehat{\delta_{\ell}}(r).
	\end{align*} Therefore the Fourier transform of $f * \delta_{\ell}$ has compact support.
	
	Combining Corollary~\ref{LLTSchwartzCompactSupport} and Lemma~\ref{LLTBound}, for $f \in \mathscr{S}(X)$,
	\begin{align*}
		&\frac{n^{\ell/2}}{\sigma^n}\int f(g.x_0) \, d\mu^{*n}(g) \\ &=  \frac{n^{\ell/2}}{\sigma^n}\int (f * \delta_{\ell})(g.x_0) \, d\mu^{*n}(g) +  \frac{n^{\ell/2}}{\sigma^n}\int (f- f * \delta_{\ell})(g.x_0) \, d\mu^{*n}(g) \\
		&= \int f(g.x_0) \psi_{0}(g) \, d\Haarof{G}(g) + O_{\mu}(||f - f*\delta_{\ell}||_1) + o_{f,\ell}(1)
	\end{align*} having used Lemma~\ref{LLTBound} and that $| \int (f- f*\delta_{\ell})(g) \psi_{0}(gh_0^{-1}) \, d\Haarof{G}(g) | \ll_{\mu} ||f- f*\delta_{\ell}||_1$ as $\psi_0$ is bounded. The claim follows by choosing $\ell$ sufficiently slowly increasing  in $n$.
\end{proof}

 \section{Quasicompactness of $S_0$}\label{SectionQuasicompactnessS0}

In this section we discuss how to establish quasicompactness of $S_0$ under strong Diophantine assumption. The reader may recall the Littlewood-Paley decomposition $L^2(K) = \bigoplus_{\ell \geq 0} V_{\ell}$ (see  \eqref{LittlewoodPayleyDecomposition}), where the space of functions $V_{\ell}$ can be pictured as oscillating with frequency $2^{\ell}$. The main result of this section states that under suitable assumptions, the operator $S_0$ has small norm on the space of functions with high enough oscillations.

Recall that we denoted by $\rho_0^{+}$ the Koopman representation induced by the $G$ action on $K$, which contains the zero principal series $\rho_0$ as a subrepresentation and write $S_0^{+} = \rho_0^{+}(\mu)$. Instead of considering $S_0$, we study $S_0^{+}$, which leads to stronger statements.

\begin{theorem}\label{BIGConsequence}
	Let $G$ be a non-compact connected simple Lie group with finite center. For $c_1, c_2 > 0$ there exists $\varepsilon_0 = \varepsilon_0(c_1,c_2) > 0$ such that the following holds.  For any $0 < \varepsilon < \varepsilon_0$ and any symmetric and $(c_1, c_2, \varepsilon)$-Diophantine probability measure $\mu$ there is $L = L(c_1,c_2) \in \Z_{\geq 1}$ such that for $\varphi \in \bigoplus_{\ell \geq L} V_{\ell}$, 
	\begin{equation}\label{S_0WeakHighFrequencyEstimate}
		||S_0^{+}\varphi||_2 \leq \frac{1}{4}||\varphi||_2.
	\end{equation}
\end{theorem}  

Theorem~\ref{BIGConsequence} will be deduced in Section~\ref{BIGConsequence} using results and ideas from \cite{BoutonnetIoanaSalehiGolsefidy2017}, thereby exploiting that the measure $\mu$ has high dimension \eqref{HighDimension} as well as a Littlewood-Paley decomposition and a mixing inequality on $G$. Under the additional assumption that $K$ is semisimple, one may instead follow Bourgain's \cite{Bourgain2012} original ideas and improve \eqref{S_0WeakHighFrequencyEstimate}.

\begin{theorem}\label{BourgainMainResult}
	Let $G$ be a non-compact connected simple Lie group with finite center and maximal compact subgroup $K$. Assume that $K$ is semisimple. For $c_1, c_2 > 0$ there exists $\varepsilon_0 = \varepsilon_0(c_1,c_2) > 0$ such that the following holds.  For any $0 < \varepsilon < \varepsilon_0$ and any symmetric and $(c_1, c_2, \varepsilon)$-Diophantine probability measure $\mu$ there is $L = L(c_1,c_2) \in \Z_{\geq 1}$ such that for $\varphi \in \bigoplus_{\ell \geq L} V_{\ell}$, 
	\begin{equation}\label{S_0HighFrequencyEstimate}
		||S_0^{+}\varphi||_2 \leq \varepsilon^{O_{c_1,c_2}(1)}||\varphi||_2.
	\end{equation}
\end{theorem}

The proof of Theorem~\ref{BourgainMainResult} was exposed in Section \ref{SectionOutline}. As in \cite{BoutonnetIoanaSalehiGolsefidy2017} we exploit that $\mu$ has high dimension, yet we work with the Littlewood-Paley decomposition on $K$ and use that the averages of matrix coefficients of $V_{\ell}$ are small (Proposition~\ref{HighOscillatingAverage}). From these results, one may easily deduce that $S_0$ and $S_0^{+}$ are quasicompact, therefore also implying Theorem~\ref{FurstenbergMeasureLLTMainTheorem}. 

\begin{corollary}\label{CorS_0Quaiscompact}
	Let $G$ be a non-compact connected simple Lie group with finite center and maximal compact group $K$. For $c_1, c_2 > 0$ there exists $\varepsilon_0 = \varepsilon_0(c_1,c_2) > 0$ such that the following holds. For any $0 < \varepsilon < \varepsilon_0$ and any symmetric and $(c_1, c_2, \varepsilon)$-Diophantine probability measure $\mu$, the operators $S_0$ and $S_0^{+}$ are quasicompact.  
\end{corollary}

\begin{proof}
	As $||S_0|| = ||S_0^{+}||$ (by Section $D$ of \cite{Guivarch1980}) and since $\rho_0^{+}$ is a subrepresentation of $\rho_0$, it suffices to show that $S_0^{+}$ is quasicompact. By Lemma~\ref{QuasicompactnessCharacterization}, the estimate \eqref{S_0WeakHighFrequencyEstimate} implies that $\rho_{\mathrm{ess}}(S_0^{+}) \leq \frac{1}{4}$. As for $\varepsilon > 0$ small enough, $||\sqrt{\alpha_g'} - 1||_{\infty} \ll |\delta|\,||g|| \ll \varepsilon^{O(1)}$ for $g \in B_{\varepsilon}$, it holds that $||S_0^{+}|| \geq 1 - \varepsilon^{O(1)}$ and hence the claim follows. 
\end{proof}

We next explain how to deduce from \eqref{S_0WeakHighFrequencyEstimate} that the Furstenberg measure is absolutely continuous. Given a non-degenerate probability measure, we study the operator $$T_0: L^2(\Omega) \to L^2(\Omega), \quad\quad  \varphi \mapsto T_0 \varphi = \int \varphi \circ \alpha_g \, d\mu(g).$$  As we discuss in the proof of Corollary~\ref{BourgainT_0MainResult}, it is shown in \cite{BenoistQuint2018} that if $\rho_{\mathrm{ess}}(T_0) < 1$, then the Furstenberg measure of $\mu$ is absolutely continuous. The following corollary is also necessary to establish Theorem~\ref{FurstenbergMeasureSmoothness}.

\begin{corollary}\label{BourgainT_0MainResult}
	Let $G$ be a non-compact connected simple Lie group with maximal compact subgroup $K$. For $c_1, c_2 > 0$ there exists $\varepsilon_0 = \varepsilon_0(c_1,c_2) > 0$ such that the following holds.  For any $0 < \varepsilon < \varepsilon_0$ and any symmetric and $(c_1, c_2, \varepsilon)$-Diophantine probability measure $\mu$ there is $L = L(c_1,c_2) \in \Z_{\geq 1}$ such that 
	\begin{equation}\label{T_0HighFrequencyEstimate}
		||T_0\varphi||_2 \leq \frac{1}{2}||\varphi||_2 \quad\quad \text{ for }\quad\quad \varphi \in\left( L^2(\Omega) \cap \bigoplus_{\ell \geq L} V_{\ell}\right).
	\end{equation}
	Then $\rho_{\mathrm{ess}}(T_0) < 1$ and the Furstenberg measure of $\mu$ is absolutely continuous. 
\end{corollary}

\begin{proof}
	Using as in the proof of Corollary~\ref{CorS_0Quaiscompact} that $||\sqrt{\alpha_g'} - 1||_{\infty} \ll |\delta|\,||g|| \ll \varepsilon^{O(1)}$ for $g \in B_{\varepsilon}$ and $\varepsilon > 0$ small enough, it follows that $||S_0 - T_0|| \leq \varepsilon^{O(1)}$. Therefore \eqref{T_0HighFrequencyEstimate} is implied by \eqref{S_0HighFrequencyEstimate}. By Lemma~\ref{QuasicompactnessCharacterization} we hence conclude $\rho_{\mathrm{ess}}(T_0) < 1$.
	
	We finally review the argument from \cite{BenoistQuint2018} to show that the Furstenberg measure of $\mu$ is absolutely continuous under the assumption that $\rho_{\mathrm{ess}}(T_0) < 1$. Indeed as $T_0 1 = 1$, it follows that $1$ is in the discrete spectrum of $T_0$. If $\rho_{\mathrm{ess}}(T_0) < 1$, one furthermore concludes (cf. Fact 2.3 of \cite{BenoistQuint2018}) that $1$ is in the discrete spectrum of the adjoint operator $T_0^{*}$ and therefore there is a function $\FurstenbergDensity \in L^2(\Omega)$ satisfying $T_0^{*}\FurstenbergDensity  = \FurstenbergDensity $. One then readily checks that $\FurstenbergDensity  d\Haarof{\Omega}$ is a $\mu$-stationary measure and thus by uniqueness of the Furstenberg measure it holds $d\FurstenbergMeasure = \FurstenbergDensity d\Haarof{K}$. 
\end{proof}

We comment on the organization of this section. Theorem~\ref{BIGConsequence} is proved in Section~\ref{SectionBIGConsequence}. The proof of Theorem~\ref{BourgainMainResult} comprises two steps. In Section~\ref{S_0Vell} we first establish using the flattening results from Theorem~\ref{SuperFlatteningLemma} that $S_0^{+}|_{V_{\ell}}$ has small operator norm. In Section~\ref{ProofBourgainMainResult} we complete the proof of Theorem~\ref{BourgainMainResult} by using that $S_0^{+} V_{\ell}$ and $V_{\ell'}$ are almost orthogonal. Finally in Section~\ref{SectionSmoothnessFurstenberg} we show how to deduce that the Furstenberg measure has a $C^{m}(K)$ density.

\subsection{Proof of Theorem~\ref{BIGConsequence}}\label{SectionBIGConsequence}

Write $T_0^{+} \varphi = \int \varphi \circ \alpha_g \, d\mu(g)$ for $\varphi \in L^2(K)$. Since $||S_0^{+} - T_0^{+}|| \leq \varepsilon^{O(1)}$, as argued in the proof of Corollary~\ref{BourgainT_0MainResult}, in order to prove Theorem~\ref{BIGConsequence} it suffices to show that 
\begin{equation}\label{BIGConsequenceMain}
	||T_0^{+}\varphi||_2 \leq \frac{1}{8}||\varphi||_2
\end{equation}
for $\varphi \in \bigoplus_{\ell \geq L} V_{\ell}$ and $L = L(c_1,c_2)$.

We proceed similarly to the proof of Corollary $C$ of \cite{BoutonnetIoanaSalehiGolsefidy2017}. Indeed, we reduce the problem at hand to studying the regular representation on $L^2(G)$. One then uses the following result of \cite{BoutonnetIoanaSalehiGolsefidy2017}, which may be considered as their core technical contribution, which uses that $\mu$ has high dimension as well as a novel Littlewood-Paley decomposition and a mixing inequality on $G$. We rephrase their result using the notion of $(c_1, c_2, \varepsilon)$-Diophantine measures. 

To introduce notation, for a measurable subset $B \subset G$ we consider the norm $$||f||_{L^2(B)}^2 = \int_B |f(g)|^2 \, d\Haarof{G}.$$

\begin{theorem}(Theorem 6.7 of \cite{BoutonnetIoanaSalehiGolsefidy2017})\label{BIG17Thm6.7}
	Let $G$ be a connected simple Lie group with finite center and $B \subset G$ a measurable set with compact closure. Let $c_1, c_2 > 0$. Then there is $\varepsilon_0 = \varepsilon_0(B,c_1,c_2) > 0$ such that the following holds. For any $0 < \varepsilon < \varepsilon_0$ and any symmetric and $(c_1,c_2,\varepsilon)$-Diophantine probability measure $\mu$ there is a finite dimensional subspace $V_B \subset L^2(B)$ such that $$||\lambda_G(\mu)|_{(V_B)^{\perp}}||_{\mathrm{op}, L^2(B)} \leq \varepsilon^{O_{B,c_1,c_2}(1)}.$$
\end{theorem}

In order to apply Theorem~\ref{BIG17Thm6.7}, we use the following lemma, which is inspired by the proof of Corollary C of \cite{BoutonnetIoanaSalehiGolsefidy2017}. Denote by $\pi_K : G \to K = G/P^{+}$ the natural projection.

\begin{lemma}\label{ComparisonBoundLemma}
	Denote $B = \{ g \in G \,:\, |\kappa(g)| \leq c \}$ for $c > 0$. For small enough $c > 0$ there is a constant $D > 1$ depending on $G$ and $c>0$ such for all $\varphi \in L^2(K)$,
	\begin{equation}\label{ComparisonBound}
		D^{-1}||\varphi||_{L^2(K)} \leq  ||\varphi \circ \pi_{K}||_{L^2(B)}  \leq D||\varphi||_{L^2(K)}.
	\end{equation}
\end{lemma}

\begin{proof}
	Recall that we denote $P^{+} = AN$. By \cite{BekkaDeLaHarpeValetteKazhdanBook} Theorem B.1.4 there is a continuous function $\rho: G \to \R_{>0}$ such that 
	\begin{equation}\label{IntegrationFormula}
		\int_G f(g) \rho(g) \, d\Haarof{G}(g) = \int_{K} \int_{P^{+}} f(kp) \, d\Haarof{P}(p)d\Haarof{K}(k) 
	\end{equation}
	for all $f \in L^1(G)$ with compact support. It moreover holds that $\alpha'_g(xP^{+}) = \frac{\rho(gx)}{\rho(x)}$ for all $x,g \in G$. For $c$ small enough $0 < \rho(g) < 1$ for all $g \in B$ and therefore $\inf_{g \in B} |1 - \rho(g)| > 0$. We choose the constant $D'$ such that $$ \sup_{g \in B} |1 - \rho(g)| \leq D'.$$ We then calculate for $\varphi_1, \varphi_2 \in L^2(K)$ using \eqref{IntegrationFormula},
	\begin{align*}
		&|\Haarof{G}(B) \langle \varphi_1, \varphi_2 \rangle_{L^2(K)} - \langle \varphi_1 \circ\pi_{K}, \varphi_2 \circ \pi_K \rangle_{L^2(B)}| \\
		&\leq \bigg| \int_K \int_{P^{+}} \varphi_1(k)\overline{\varphi_2(k)} 1_B(p) \, d\Haarof{P}(p) d\Haarof{K}(k) - \int 1_B(g) \varphi_1(\pi_K(g))\overline{\varphi_2(\pi_K(g))} \, d\Haarof{G}(g) \bigg| \\
		&\leq \bigg| \int_B  \varphi_1(\pi_K(g))\overline{\varphi_2(\pi_K(g))} (1 - \rho(g)) \, d\Haarof{G}(g) \bigg| \\
		&\leq ||\varphi_1 \circ \pi_K||_{L^2(B)} \sqrt{\int_B |\varphi_2(\pi_K(g))|^2 \, |1 - \rho(g)|^2 \,d\Haarof{G}(g)} \\
		&\leq D' ||\varphi_1 \circ \pi_K||_{L^2(B)} ||\varphi_2 \circ \pi_K||_{L^2(B)}. 
	\end{align*} By a similar argument we may also estimate the latter term by $$\Haarof{G}(B)D'||\varphi_1||_{L^2(K)}||\varphi_2||_{L^2(K)}.$$ Setting $\varphi = \varphi_1 = \varphi_2$ the claim is readily implied by choosing $D$ suitably in terms of $D'$ and $\Haarof{G}(B)$. 
\end{proof} 

Throughout the following denote by $B = \{ g \in G \,:\, |\kappa(g)| \leq c \}$ a set from Lemma~\ref{ComparisonBoundLemma} such that \eqref{ComparisonBound} holds. We are now in a suitable position to apply Theorem~\ref{BIG17Thm6.7}. Indeed for $\varphi \in L^2(K)$  it holds by \eqref{ComparisonBound} that 
\begin{equation}\label{ComparisonBound2}
	||T_0 \varphi||_{L^2(K)} \leq D ||(T_0 \varphi) \circ \pi_K||_{L^2(B)} = D||\lambda_G(\mu) (\varphi \circ \pi_K)||_{L^2(B)}.
\end{equation}
Let $V_B\subset L^2(B)$ the finite dimensional subspace of Theorem~\ref{BIG17Thm6.7}. We then may choose $L$ large enough such that if $\varphi \in \bigoplus_{\ell \geq L} V_{\ell}$ then 
\begin{equation}\label{WeakConvergence}
	||\varphi \circ \pi_K - (\varphi \circ \pi_K)_{(V_B)^{\perp}} ||_{L^2(B)} \leq \frac{1}{16D^2}||\varphi \circ \pi_K||_{L^2(B)},
\end{equation}
where $(\varphi \circ \pi_K)_{(V_B)^{\perp}}$ is the projection of $\varphi \circ \pi_K$ onto $(V_B)^{\perp}$. Indeed this follows using \eqref{IntegrationFormula} and that $V_B$ is finite dimensional. 

We conclude using Theorem~\ref{BIG17Thm6.7}, \eqref{ComparisonBound},\eqref{ComparisonBound2} and \eqref{WeakConvergence},
\begin{align*}
	||T_0 \varphi||_{L^2(K)} &\leq D||\lambda_G(\mu) (\varphi \circ \pi_K)||_{L^2(B)} \\
	&\leq D||\lambda_G(\mu) \left( \varphi \circ \pi_K -  (\varphi \circ \pi_K)_{(V_B)^{\perp}}\right)||_{L^2(B)} + D||\lambda_G(\mu)(\varphi \circ \pi_K)_{(V_B)^{\perp}}||_{L^2(B)} \\
	&\leq \frac{1}{16D} ||\varphi \circ \pi_K||_{L^2(B)} + D\varepsilon^{O_{c_1,c_2}(1)} ||\varphi \circ \pi_K||_{L^2(B)}\\ 
	&\leq \left( \frac{1}{16} + D^2 \varepsilon^{O_{c_1,c_2}(1)}\right) ||\varphi||_{L^2(K)},
\end{align*} showing \eqref{BIGConsequenceMain} by choosing $\varepsilon$ small enough in terms of $c_1$ and $c_2$. The proof of Theorem~\ref{BIGConsequence} is complete.

\subsection{Operator Norm Estimate for $S_0^{+}$ on $V_{\ell}$}\label{S_0Vell}

In this section we prove the following proposition.

\begin{proposition}\label{NormEstimateonVell} For $c_1, c_2 > 0$ there exists $\varepsilon_0 = \varepsilon_0(G,c_1,c_2) > 0$ such that the following holds.  For any $0 < \varepsilon < \varepsilon_0$ and any symmetric and $(c_1, c_2, \varepsilon)$-Diophantine probability measure $\mu$, there is $L = L(G,c_1,c_2) \in \Z_{\geq 1}$ such that $||S_0^{+}|_{V_{\ell}}||_{\mathrm{op}} \leq \varepsilon^{O_{c_1,c_2}(1)}$ for $\ell \geq L$. 
\end{proposition}

Recall that as introduced in Section~\ref{Flattening}, $$P_{\delta} = \frac{1_{B_{\delta}}}{\Haarof{G}(B_{\delta})}.$$ For the proof of Proposition~\ref{NormEstimateonVell}, one estimates by the triangle inequality for $n \geq 1$ and $\varphi \in V_{\ell}$,
\begin{equation}\label{MainEstaimteNormonVell}
	||(S_0^{+})^n \varphi||_2 \leq ||(S_0^+)^n \varphi - \rho_0^+(\mu^{*n}*P_{\delta})\varphi||_2 + ||\rho_0^+(\mu^{*n} * P_{\delta})\varphi||_2.
\end{equation} We aim to show that \eqref{MainEstaimteNormonVell} is very small for a suitably chosen $n$ and $\delta$. For the first term of \eqref{MainEstaimteNormonVell}, we use that the Lipschitz constant of $\varphi$ is $\asymp ||\gamma||^{O(1)}$. Therefore, a $\delta$-perturbation of $(S_0^+)^n\varphi = \rho_0^{+}(\mu^{*n})\varphi$ is small provided we choose $\delta$ miniscule in terms of $\ell$. 

The second term of \eqref{MainEstaimteNormonVell} is dealt with by using that $\mu$ has high dimension. Indeed by Lemma~\ref{SuperFlatteningLemma} it will follow that $\mu^{*n} * P_{\delta}$ has small $||\cdot||_{\infty}$-norm for $n$ chosen in terms of $\delta$. This will allow us to compare $||\rho_0^+(\mu^{*n} * P_{\delta})\varphi||_2$ to the average estimate of matrix coefficients $$\frac{1}{\Haarof{G}(B_R)}\int_{B_R} |\langle \rho_0^{+}(g)\varphi, \varphi \rangle| \, d\Haarof{G}(g) \ll 2^{-\ell/2}||\varphi||_2$$ that was discussed in Section~\ref{AverageMatrixCoeffSemisimple}.

We proceed with some preliminary lemmas used in the proof of Proposition~\ref{NormEstimateonVell}. First, we estimate how much $\rho_0^+(g)\varphi$ differs from $\varphi$, given that $\varphi \in V_{\ell}$ and $g \in B_{\delta}$.

\begin{lemma}\label{DistanceLowFrequency}
	Fix $\ell \geq 0$. Then for $\varphi \in V_{\ell}$ and $0 < \delta \ll 2^{-\ell}$, it holds for $g \in B_{\delta}$, $$||\rho_0^+(g) \varphi - \varphi||_2 \ll e^{O(1)\ell}\delta^{O(1)}||\varphi||_2.$$
\end{lemma}

\begin{proof}
	We first fix $\gamma \in \overline{C} \cap I^{*}$ and denote as usual by $\pi_{\gamma}$ the associated irreducible representation and let $v_1, \ldots , v_n \in \pi_{\gamma}$ be an orthonormal basis of the representation space of $\pi_{\gamma}$. For $k \in B_{\delta}$ in $K$ for $\delta$ small enough, it holds by Lemma~3.1 of \cite{deSaxce2013} that  $||\pi_{\gamma}(k) - \mathrm{Id}_{\pi_{\gamma}}||_\mathrm{op} \ll d_K(k,e)||\gamma||$. Indeed, upon conjugation, we can assume that $k$ is inside the maximal torus $T$ of $K$ and hence we can write $k = e^{X}$ for $X \in \mathfrak{t} = \mathrm{Lie}(T)$ with $||X|| \ll d_K(k,e)$. With these assumptions, the eigenvalues of $\pi_{\gamma}(k) - \mathrm{Id}_{\pi_{\gamma}}$ can be calculated as $e^{\gamma'(X)} - 1$ for $\gamma'$ the weights of $\pi_{\gamma}$. Choosing $\delta \ll 2^{-\ell}$, and therefore having $|\gamma'(X)| \ll 1$, we can bound $\max_{\gamma'} |e^{\gamma'(X)} - 1| \ll \max_{\gamma'} |\gamma'(X)| \ll d_K(g,e) ||\gamma||,$ showing the claim. 
	
	Denote by $\psi$ the matrix coefficient $k \mapsto \sqrt{d_{\pi}} \langle \pi_{\gamma}(k) v_i, v_j \rangle$, satifying $||\psi||_2 = 1$. We first show that $||\rho_0^+(g)\psi - \psi||_2 \ll \delta^{O(1)} ||\gamma||^{O(1)}$ for $g \in B_{\delta}$. Indeed, using as in the proof of Corollary~\ref{CorS_0Quaiscompact} that $||\sqrt{\alpha'_g}(k) - 1||_{\infty} \ll \delta^{O(1)}$ and Lemma~\ref{DimensionLaplacianBound},
	\begin{align*}
		|(\rho_0^+(g) \psi)(k) - \psi(k)| &= \Big| \left(\sqrt{\alpha'_g}(k) - 1 \right) \psi(g^{-1}.k)\Big| + \big|\psi(g^{-1}.k) -  \psi(k) \big| \\
		&\ll \delta^{O(1)}|\psi(g^{-1}.k)| + \sqrt{d_{\pi}}||\pi_{\gamma}(g^{-1}.k) - \pi_{\gamma}(k)||_{\mathrm{op}}  \\
		&\ll \delta^{O(1)}|\psi(g^{-1}.k)| + \delta^{O(1)}||\gamma||^{O(1)},
	\end{align*} which implies the claim using $|\psi(g^{-1}.k)| \leq |\psi(g^{-1}.k) - \psi(k)| + |\psi(k)|$.
	
	To prove the lemma, denote by $(\psi_i)_{i \in I}$ an orthonormal basis of $V_{\ell}$ with functions as in the previous paragraph. Then $|I| \ll e^{O(1)\ell}$ and for $\varphi \in V_{\ell}$ we decompose $\varphi = \sum_{i \in I} a_i \psi_i$, implying using Cauchy-Schwarz,
	\begin{align*}
		||\rho_0^+(g) \varphi - \varphi||_2 \leq \sum_{i\in I} |a_i| \, ||\rho_0^+(g)  \psi - \psi||_2 
		\ll e^{O(1)\ell}\delta^{O(1)}||\varphi||_2.
	\end{align*}
\end{proof}

We next show how to compare $\pi(\nu)\varphi$ with $\pi(\nu*P_{\delta})\varphi$ for a suitable vector $\varphi$ and a unitary representation $\pi$ and probability measure $\nu$.

\begin{lemma}\label{RepresentationDeltaApproxLemma}
	Let $(\pi, \mathscr{H})$ be a unitary representation of $G$ and let $\delta > 0$. Fix $\varphi \in \mathscr{H}$. Assume that $||\pi(g) \varphi - \varphi|| \leq C_{\delta}||\varphi||$ for all $g \in B_{\delta}$ and $C_{\delta} > 0$ a constant. Then for any probability measure $\nu$, $$||\pi(\nu) \varphi - \pi(\nu * P_{\delta})\varphi|| \leq C_{\delta}||\varphi||.$$
\end{lemma}

\begin{proof}
	Using Fubini's theorem and that $1_{B_{\delta}(g)}(h) = 1_{B_{\delta}(h)}(g)$,
	\begin{align*}
		\pi(\nu) \varphi &= \int \frac{1}{\Haarof{G}(B_{\delta}(e))} \left( \int 1_{B_{\delta}(g)}(h)\pi(g) \varphi \, d\Haarof{G}(h) \right) \, d\nu(g) \\ &= \int \frac{1}{\Haarof{G}(B_{\delta}(e))}\left(\int_{B_{\delta}(h)} \pi(g)\varphi \,d\nu(g) \right) d\Haarof{G}(h).
	\end{align*} Furthermore, by the assumption and using that $B_{\delta}(h) = hB_{\delta}(e)$ (the metric on $G$ is left invariant),
	\begin{align*}
		\bigg|\bigg|\int_{B_{\delta}(h)} \pi(g)\varphi \,d\nu(g) -  \nu(B_{\delta}(h)) \cdot  \pi(h)\varphi \bigg|\bigg| & \leq \int_{B_{\delta}(h)} ||(\pi(g) - \pi(h))\varphi|| \, d\nu(g) \\
		& \leq \int_{B_{\delta}(h)} ||\pi(h)(\pi(h^{-1}g) - \mathrm{Id})\varphi|| \, d\nu(g) \\ &\leq \nu(B_{\delta}(h)) C_{\delta} ||\varphi||.
	\end{align*} Finally, as $(\nu * P_{\delta})(h) = \frac{\nu(B_{\delta}(h))}{\Haarof{G}(B_{\delta}(e))}$, 
	\begin{align*}
		&||\pi(\nu) \varphi - \pi(\nu * P_{\delta})\varphi|| \\ &= \bigg|\bigg|  \int \frac{1}{\Haarof{G}(B_{\delta}(e))}\left(\int_{B_{\delta}(h)} \pi(g)\varphi \,d\nu(g) \right) d\Haarof{G}(h) - \int \pi(h)\varphi \, (\nu * P_{\delta})(h) \, d\Haarof{G}(h)  \bigg|\bigg| \\
		&\leq  \int \frac{1}{\Haarof{G}(B_{\delta}(e))} \bigg|\bigg|\int_{B_{\delta}(h)} \pi(g)\varphi \,d\nu(g) - \nu(B_{\delta}(h)) \cdot \pi(h) \varphi\bigg|\bigg| d\Haarof{G}(h)  \\
		&\leq C_{\delta}||\varphi|| \cdot \int \frac{\nu(B_{\delta}(h))}{\Haarof{G}(B_{\delta}(e))} \, d\Haarof{G}(h) = C_{\delta} ||\varphi||,
	\end{align*} using in the last line that by Fubini's theorem $\int \frac{\nu(B_{\delta}(h))}{\Haarof{G}(B_{\delta}(e))} \, d\Haarof{G}(h) = 1$ as $\nu$ is a probability measure.
\end{proof}

\begin{proof}(of Proposition~\ref{NormEstimateonVell})
	Let $\gamma > 0$ be a fixed constant to be determined later. Then by Proposition~\ref{SuperFlatteningLemma} there is $\varepsilon_0 = \varepsilon_0(c_1,c_2) > 0$ and $C_0 = C_0(c_1,c_2) > 0$ such that for $\delta > 0$ small enough it holds that $||(\mu^{*n})_{\delta}||_2 \leq \delta^{-\gamma}$ for any $n \geq C_0\frac{\log \frac{1}{\delta}}{\log \frac{1}{\varepsilon}}$ and $$(\mu^{*n})_{\delta} = \mu^{*n} * P_{\delta}.$$
	
	Let $\varphi \in V_{\ell}$ with $||\varphi||_2 = 1$. Then by the triangle inequality, 
	\begin{align*}
		||(S_0^{+})^n \varphi||_2 \leq ||(S_0^+)^n \varphi - \rho_0^+(\mu^{*n}*P_{\delta})\varphi||_2 + ||\rho_0^+(\mu^{*n} * P_{\delta})\varphi||_2.
	\end{align*} The first term can be estimated using Lemma~\ref{DistanceLowFrequency} and Lemma~\ref{RepresentationDeltaApproxLemma} as $||(S_0^+)^n \varphi - \rho_0^+(\mu^{*n}*P_{\delta})\varphi||_2 \ll e^{O(1)\ell} \delta^{O(1)}$ assuming that $\delta \ll 2^{-\ell}$. For the second term, first notice that by applying Cauchy-Schwarz it follows that $||(\mu^{*n})_{\delta} * (\mu^{*n})_{\delta}||_{\infty} \leq ||(\mu^{*n})_{\delta}||_2^2$. Then with Theorem~\ref{SuperFlatteningLemma} and Proposition~\ref{HighOscillatingAverage},
	\begin{align*}
		||\rho_0^+(\mu^{*n} * P_{\delta}) \varphi||_2^2 &= \langle \rho_0^+(\mu^{*n} * P_{\delta} * \mu^{*n} * P_{\delta}) \varphi, \varphi \rangle \\
		&\leq \int |\langle \rho_0^+(g) \varphi , \varphi \rangle| \, ((\mu^{*n})_{\delta} * (\mu^{*n})_{\delta})(g) \, d\Haarof{G}(g) \\
		&\leq \delta^{-2\gamma} \int_{B_{4n\varepsilon}}  |\langle \rho_0^+(g) \varphi , \varphi \rangle| \, d\Haarof{G}(g) \\
		&\ll \delta^{-2\gamma}\Haarof{G}(B_{4n\varepsilon}) e^{-O(1)\ell} \leq \delta^{-2\gamma}e^{O(1)n\varepsilon} e^{-O(1)\ell}.
	\end{align*} Let $n$ be a power of 2 satisfying $n \asymp C_0  \frac{\log \frac{1}{\delta}}{\log \frac{1}{\eps}}$. Then by using that $S_0^+$ is self-adjoint and $n$ a power of $2$, it follows by induction on $k$ with $2^k = n$ that $||(S_0^+)\varphi ||_2^n \leq ||(S_0^+)^n \varphi||_2$. Therefore it follows for $\delta \ll 2^{-\ell}$ that $$||S_0^+|_{V_{\ell}}||_{\mathrm{op}} \leq D^{\frac{1}{n}} \max\{ e^{\frac{\sigma_1\ell}{n}} \delta^{\frac{\sigma_2}{n}}, \delta^{-\frac{\gamma}{n}} e^{- \frac{\sigma_3 \ell}{n}}\},$$ for $D, \sigma_1, \sigma_2,\sigma_3 > 0$ absolute constants. We choose $\delta = e^{-\max\{ 1, \frac{2 \sigma_1}{\sigma_2}  \}\ell}$ so that $\delta \ll 2^{- \ell}$ and $e^{\frac{\sigma_1\ell}{n}} \delta^{\frac{\sigma_2}{n}} \leq e^{-\frac{\sigma_1 \ell}{n}}$.  We furthermore set $\gamma = \frac{\sigma_3}{2\max\{ 1, 2 \frac{\sigma_1}{\sigma_2} \}}$ and therefore $\delta^{-\frac{\gamma}{n}} e^{- \frac{\sigma_3 \ell}{n}} = e^{- \frac{\sigma_3 \ell}{2n}}$. With these choices, $||S_0^+|_{V_{\ell}}||_{\mathrm{op}} \leq D^{\frac{1}{n}}e^{-O(1)\frac{\ell}{n}}$. In addition we make $\ell$ large enough in terms of $c_1$ and $c_2$ such that $\delta$ becomes small enough for Proposition~\ref{SuperFlatteningLemma} to hold. To conclude, it holds by construction that $\frac{\ell}{n} \asymp_{c_1,c_2} \log \frac{1}{\varepsilon}$ and therefore $e^{-O(1)\frac{\ell}{n}} = \varepsilon^{O_{c_1,c_2}(1)}$ and similarly $D^{\frac{1}{n}} = \varepsilon^{-\frac{O_{c_1,c_2}(1)}{\ell}}$, so choosing $\ell$ additionally larger than a further constant depending on $c_1$ and $c_2$, the claim follows. 
\end{proof}

\subsection{Proof of Theorem~\ref{BourgainMainResult}}\label{ProofBourgainMainResult}

Having established that $||S_0^{+}|_{V_{\ell}}||_{\mathrm{op}}$ is small for $L \geq L(c_1,c_2)$, we aim to convert this to an estimate that $||S_0^{+}|_{\bigoplus_{\ell \geq L}V_{\ell}}||_{\mathrm{op}}$ is also small. We use that the spaces $S_0^{+}V_{\ell}$ and $V_{\ell'}$ are almost orthogonal for $\ell \neq \ell'$ as shown in Lemma~\ref{AlmostOrthogonalRho0}.

The Lie algebra of $K$ is denoted $\mathfrak{k}$ and we also write $\lambda_K$ for the Lie algebra representation induced by the regular representation $\lambda_K$ on $K$. Indeed, for a smooth function $\varphi$ on $K$ the function $(\lambda_K(X)\varphi)(k) = \lim_{t \to 0} \frac{1}{t}(\varphi(e^{-tX}k) - \varphi(k))$ with $X \in \mathfrak{k}$ and $k \in K$ is the directional derivative of $\varphi$ in the direction $-X$. 

As in \cite{Bourgain2012}, we use an argument based on partial integration to show that $S_0^{+}V_{\ell}$ and $V_{\ell'}$ are almost orthogonal. For a general manifold there is no suitable partial integration formula. However, for compact Lie groups we overcome this issue by exploiting that the Laplacian acts as a scalar on functions on $L^2(K)$ induced by the representation $\pi_{\gamma}$. Indeed, for a fixed orthonormal basis $X_1,\ldots , X_{\dim K}$ of $\mathfrak{k}$ recall that the Casimir element is defined as $\triangle = - \sum_{i} X_i \circ X_i$. We then use as replacement to partial integration that 
\begin{equation}\label{PartialIntegration}
	\langle \varphi_1, \lambda_K(\triangle)\varphi_2 \rangle = \sum_i \langle \lambda_K(-X_i) \varphi_1, \lambda_K(X_i) \varphi_2 \rangle. 
\end{equation}
In order to give a suitable estimate for \eqref{PartialIntegration}, we first analyse $||\lambda_K(X) \varphi||_2$ for $X \in \mathfrak{k}$.

\begin{lemma}\label{DirectionalDerivativeEstimate}
	Let $\ell \geq 0$ and $\varepsilon > 0$. Then for $\varphi \in V_{\ell},g \in B_{\varepsilon}$ and $X \in \mathfrak{k}$ of unit norm, $$||\lambda_K(X) \varphi||_2 \ll 2^{\ell} ||\varphi|| \quad\quad \text{and} \quad\quad ||\lambda_K(X) (\rho_0^+(g) \varphi)||_2 \ll (1 + O(\varepsilon^{O(1)}))2^{\ell}||\varphi||_2.$$
\end{lemma}

\begin{proof}
	Without loss of generality we assume that $X \in \mathfrak{t}$. Fix $\gamma \in \overline{C} \cap I^{*}$. The eigenvalues of the operator $\pi_{\gamma}(e^{tX}) - \mathrm{Id}$ can be calculated as $e^{t\gamma'(X)} - 1$ for $\gamma'$ the various weights of the representation $\pi_{\gamma}$. Therefore the operator $\pi_{\gamma}(X) = \lim_{t \to 0} \frac{1}{t}(\pi_{\gamma}(e^{tX}) - \mathrm{Id})$ has eigenvalues $\gamma'(X)$. Let $v_1, \ldots , v_n$ be an orthonormal basis of eigenvectors of $\pi_{\gamma}(X)$. Then the functions $\psi(k) = \sqrt{d_{\gamma}} \langle \pi_{\gamma}(k)v_i, v_j \rangle$ for $k \in K$ satisfy $(\lambda_K(X)\psi)(k) = \sqrt{d_{\gamma}}\langle \pi_{\gamma}(k) v_i , \pi_{\gamma}(X)v_j \rangle = (\gamma'(X) \psi)(k)$. The first claim follows as $||\gamma'(X)|| \ll ||\gamma|| \leq 2^{\ell}$ and by decomposing the function $\varphi$ as a sum of functions of the form $\psi$.
	
	For the second claim recall that $\rho_0^+(g)\varphi = \sqrt{\alpha_g'} \cdot (\varphi \circ \alpha_g)$ and therefore 
	\begin{equation}\label{DirectionalDerivativeMain}
		\lambda_K(X)(\rho_0^+(g)\varphi) =  \left(\lambda_K(X) \sqrt{\alpha_g'} \right)\cdot (\varphi \circ \alpha_g) + \sqrt{\alpha_g'}\cdot \lambda_K(X)(\varphi \circ \alpha_g).
	\end{equation}
	To deal with the first term of \eqref{DirectionalDerivativeMain}, since $\alpha_g'$ is a smooth polynomial perturbation of the identity, it follows that $||\lambda_K(X) \sqrt{\alpha_g'}||_{\infty} \leq (1 + O(\varepsilon^{O(1)}))$ and furthermore using integration by substitution, $||\varphi \circ \alpha_g||_2 \ll (1 + O(\varepsilon^{O(1)}))||\varphi||_2$. For the second term of \eqref{DirectionalDerivativeMain}, we use the chain rule and the first step to conclude that $||\lambda_K(X)(\varphi \circ \alpha_g)||_2 \ll (1 + O(\varepsilon^{O(1)}))2^{\ell}||\varphi||_2$, concluding the lemma.
\end{proof}

We now apply \eqref{PartialIntegration} to prove the following lemma. 

\begin{lemma}\label{AlmostOrthogonalRho0}
	For $\varphi_{\ell_1} \in V_{\ell_1}$ and $\varphi_{\ell_2} \in V_{\ell_2}$ with $\ell_1 \neq \ell_2$ and $g \in B_{\varepsilon}$, $$|\langle \rho_0^+(g)\varphi_{\ell_1}, \varphi_{\ell_2} \rangle| \ll (1 + O(\varepsilon^{O(1)})) 2^{-|\ell_1 - \ell_2|}||\varphi_{\ell_1}||_2 ||\varphi_{\ell_2}||_2. $$
\end{lemma}

\begin{proof}
	Without loss of generality we assume that $\ell_2 > \ell_1$. Denote by $\psi\in V_{\ell_2}$ the function such that $\lambda_K(\triangle) \psi = \varphi_{\ell_2}$. Then by Lemma~\ref{DimensionLaplacianBound}, $||\psi||_2 \ll 2^{-2\ell_2}||\varphi_{\ell_2}||_2$. Using then \eqref{PartialIntegration} and Lemma~\ref{DirectionalDerivativeEstimate},
	\begin{align*}
		|\langle \rho_0^+(g)\varphi_{\ell_1}, \varphi_{\ell_2} \rangle| &= |\langle \rho_0^+(g) \varphi_{\ell_1}, \lambda_K(\triangle) \psi \rangle | \\ &= \bigg|  \sum_{i} \langle \lambda_K(-X_i) \rho_0^+(g)\varphi_{\ell_1}, \lambda_K(X_i)\psi \rangle \bigg| \\ &\leq \sum_{i} ||\lambda_K(-X_i)(\rho_0^+(g)\varphi_{\ell_1})||\, ||\lambda_K(X_i)\psi|| \\
		&\ll (1 + O(\varepsilon^{O(1)}))2^{\ell_1 + \ell_2} ||\varphi_{\ell_1}||_2 ||\psi||_2 \\
		&\ll (1 + O(\varepsilon^{O(1)}))2^{\ell_1 - \ell_2} ||\varphi_{\ell_1}||_2 ||\varphi_{\ell_2}||_2.
	\end{align*} 
\end{proof}

We conclude this section by proving Theorem~\ref{BourgainMainResult} by combining Proposition~\ref{NormEstimateonVell} and Lemma~\ref{AlmostOrthogonalRho0}.

\begin{proof}(of Theorem~\ref{BourgainMainResult})
	By Proposition~\ref{NormEstimateonVell}, there is $\varepsilon_0 = \varepsilon_0(c_1,c_2) > 0$ and $L = L(c_1,c_2) \in \Z_{\geq 1}$ such that $||S_0^+|_{V_{\ell}}||_{\mathrm{op}} \leq \varepsilon^{O_{c_1,c_2}(1)}$ for $\ell \geq L$. Let $\varphi \in \bigoplus_{\ell \geq L} V_{\ell}$ and let $N\geq 1$ to be determined later. Then 
	\begin{align*}
		||S_0^+ \varphi||_2^2 &\leq \sum_{\ell, \ell' \geq L} |\langle S_0^+ \pi_{\ell} \varphi, S_0^+ \pi_{\ell'}\varphi \rangle| \\
		&= \sum_{|\ell - \ell'|\leq N} |\langle S_0^+ \pi_{\ell} \varphi, S_0^+ \pi_{\ell'}\varphi \rangle| + \sum_{|\ell - \ell'| > N} |\langle S_0^+ \pi_{\ell} \varphi, S_0^+ \pi_{\ell'}\varphi \rangle|,
	\end{align*} where both of the sums are with $\ell, \ell' \geq L$. For the first of these two terms one uses the conclusion of Proposition~\ref{NormEstimateonVell},
	$$\sum_{|\ell - \ell'|\leq N} ||S_0^+ \pi_{\ell} \varphi ||\,  ||S_0^+ \pi_{\ell'}\varphi ||\leq N \sum_{\ell \geq L}  ||S_0^+ \pi_{\ell}\varphi||_2^2 \leq N \varepsilon^{O_{c_1,c_2}(1)}||\varphi||_2^2.$$
	Lemma~\ref{AlmostOrthogonalRho0} is used to bound the second term:
	\begin{align*}
		\sum_{|\ell - \ell'| >  N} |\langle S_0^+ \pi_{\ell} \varphi, S_0^+ \pi_{\ell'}\varphi \rangle|  &\ll \sum_{|\ell - \ell'| >  N}  2^{-|\ell - \ell'|} ||\pi_{\ell} \varphi|| \, ||\pi_{\ell'}\varphi ||  \\ &\ll \sum_{|\ell - \ell'| >  N} 2^{-|\ell - \ell'|} ||\pi_{\ell}\varphi||_2^2 \\ &\ll 2^{-N} \sum_{\ell \geq L} ||\pi_{\ell} \varphi||_2^2  = 2^{-N}||\varphi||_2^2.
	\end{align*} Therefore it follows that $||S_0^+ \varphi||_2 \leq \sqrt{N \varepsilon^{O_{c_1,c_2}(1)} + 2^{-N}}||\varphi||_2$. Setting $N =  \log \frac{1}{\varepsilon}$ implies the claim of the theorem.
\end{proof}

\subsection{Smoothness of the Furstenberg Measure}\label{SectionSmoothnessFurstenberg}

In this section we prove Theorem~\ref{FurstenbergMeasureSmoothness}, which we restate here for convenience of the reader.

\begin{theorem}(Theorem~\ref{FurstenbergMeasureSmoothness}) 
	Let $G$ be a non-compact connected simple Lie group with finite center. Let $c_1, c_2 > 0$ and $m \in \Z_{\geq 1}$. Then there is $\eps_m = \eps_m(G,c_1,c_2) > 0$ depending on $G,c_1,c_2$ and $m$ such that every symmetric and $(c_1,c_2, \varepsilon)$-Diophantine probability measure $\mu$ with $\varepsilon \leq \varepsilon_m$ has absolutely continuous Furstenberg measure with density in $C^m(\Omega)$.
\end{theorem}

By Corollary~\ref{BourgainT_0MainResult}, we know that the Furtstenberg measure is absolutely continuous if we choose $\eps_m$ small enough, i.e. there is $\FurstenbergDensity \in L^2(\Omega)$ such that $d\FurstenbergMeasure = \FurstenbergDensity d\Haarof{\Omega}$. In order to prove Theorem~\ref{FurstenbergMeasureSmoothness}, we use the smoothness condition from Lemma~\ref{SmoothnessCharacterization} for $\psi_F$. Indeed, for $P_{\ell}$ the projection from $L^2(K)$ to $V_{\ell}$, it suffices to show $$||P_{\ell}\psi_F||_2 \leq 2^{-(s + 1)\ell}$$  for $s > m + \frac{1}{2} \dim K$ and $\ell$ large enough.

By the characterization of the Furstenberg measure, for any $n \geq 1$ it holds that $\FurstenbergMeasure = \mu^{*n} * \FurstenbergMeasure$ and therefore for $\varphi \in L^2(K)$, 
\begin{align}
	|\langle \FurstenbergDensity, \varphi \rangle| &= \bigg|\int \varphi \, d\FurstenbergMeasure\bigg| \nonumber \\ &= \bigg|\int\int \varphi(g.k) \, d\mu^{*n}(g)d\FurstenbergMeasure(k)\bigg| \nonumber \\ &\leq \bigg|\bigg|  \int \varphi \circ \alpha_g \, d\mu^{*n}(g)  \bigg|\bigg|_{\infty}. \label{ReductionPhi_nEstimate}
\end{align}
We thus study the $L^{\infty}$-norm of the function $$\Phi_n = T_0^n \varphi =\int \varphi \circ \alpha_g \, d\mu^{*n}(g).$$ We will use  Corollary~\ref{BourgainT_0MainResult} to give $L^2$-estimates of $\Phi_n$. In order to convert these estimates to an $L^{\infty}$-bound, we use Agmon's inequality (cf. \cite{AgmonEllipticBook} chapter 13), which we introduce for compact Lie groups. 

\begin{lemma}\label{AgmonsInequality}(Agmon's Inequality for Compact Lie Groups).
	Let $K$ be a compact Lie group. Then there is $t \in \Z_{\geq 2}$ depending on $K$ such that for any $\varphi \in C^{\infty}(K)$,  $$||\varphi||_{\infty} \ll ||\varphi||_2^{1/2} ||\varphi||_{H^t}^{1/2}.$$
\end{lemma}

\begin{proof}
	For $M \in \R_{> 0}$ to be determined, we group together the contribution of the representations with $||\gamma|| \leq M$ and $||\gamma|| > 0$. Indeed, by \eqref{CompactGroupFourierInversion}, for $k \in K$, 
	\begin{align*}
		\varphi(k) &=  \sum_{\gamma \in \overline{C}\cap I^{*}} \sum_{i,j = 1}^{d_{\gamma}} d_{\gamma}^{1/2}a_{ij}^{\gamma} \chi_{ij}^{\gamma}(k) \\
		&= \sum_{||\gamma|| \leq M} \sum_{i,j = 1}^{d_{\gamma}} d_{\gamma}^{1/2}a_{ij}^{\gamma} \chi_{ij}^{\gamma}(k)  + \sum_{||\gamma|| > M} \sum_{i,j = 1}^{d_{\gamma}} d_{\gamma}^{1/2}a_{ij}^{\gamma} \chi_{ij}^{\gamma}(k)  \\
		&= \sum_{||\gamma|| \leq M} \sum_{i,j = 1}^{d_{\gamma}} d_{\gamma}^{1/2}a_{ij}^{\gamma} \chi_{ij}^{\gamma}(k)  + \sum_{||\gamma|| > M} \sum_{i,j = 1}^{d_{\gamma}} \lambda_{\gamma}^t\lambda_{\gamma}^{-t} d_{\gamma}^{1/2} a_{ij}^{\gamma} \chi_{ij}^{\gamma}(k),  \\
	\end{align*} where in the last line we multiplied the second term by $1 = \lambda_{\gamma}^t\lambda_{\gamma}^{-t}$ for some $t\in \Z_{\geq 0}$.  By Cauchy-Schwarz and using Lemma~\ref{DimensionLaplacianBound}, the first term can be bounded by $||\varphi||_2 \sqrt{\sum_{||\gamma||\leq M, i,j} d_{\gamma}} \ll M^{C}||\varphi||_2$, where $C$ is a constant depending on $K$. For the second term, we choose $t$ large enough such that $\sqrt{\sum_{||\gamma|| > M,i,j} \lambda_{\gamma}^{-2t}d_{\gamma} } \ll M^{-C}.$ Again using Cauchy-Schwarz, the second term is bounded by $M^{-C}||\varphi||_{H^t}$. The claim is implied by setting $M = (\frac{||\varphi||_{H^t}}{||\varphi||_2})^{1/2C}$.
\end{proof}

\begin{lemma}\label{Phi_nFourierEstimate}
	For $\varphi \in V_{\ell}$ set $\Phi_n = T_0^n \varphi$. Let $\gamma \in \overline{C}\cap I^{*}$ and $r \in \Z_{\geq 1}$. Then it holds for $\widehat{\Phi_n}(\gamma) = \pi_{\gamma}(\Phi_n)$, $$||\widehat{\Phi_n}(\gamma)||_{\mathrm{op}} \ll_r 2^{O(1)\ell - r\ell}(1 + \varepsilon)^{O(1)nr} ||\gamma||^{O(1)r} ||\varphi||_2.$$
\end{lemma}

\begin{proof}
	Let $v_1, \ldots , v_{d_{\gamma}}$ be an orthonormal basis of $\pi_{\gamma}$. Then
	\begin{align}
		||\widehat{\Phi_n}(\gamma)||_{\mathrm{op}} &\leq d_{\gamma} \sup_{1 \leq i \leq d_{\gamma}} ||\widehat{\Phi}(\gamma)v_i|| \nonumber \\ &\leq  d_{\gamma} \sup_{1 \leq i,j \leq d_{\gamma}} |\langle \widehat{\Phi}(\gamma) v_i, v_j \rangle| \nonumber \\ &= d_{\gamma} \sup_{1 \leq i,j \leq d_{\gamma}} |\langle \Phi_n, \chi^\gamma_{ij} \rangle| \nonumber \\
		&\leq d_{\gamma} \sup_{\substack{g \in \supp{\mu^{*n}} \\ 1 \leq i,j \leq d_{\gamma}}} |\langle \varphi \circ \alpha_g, \chi^{\gamma}_{ij}\rangle|. \label{Phi_nBound}
	\end{align}
	Notice further that for $g \in B_{\varepsilon}$ and a further $\gamma' \in \overline{C} \cap I^{*}$ and $1 \leq i',j' \leq d_{\gamma}$, 
	\begin{align*}
		|\langle \chi^{\gamma'}_{i'j'} \circ \alpha_g, \chi^{\gamma}_{ij} \rangle| &= \frac{\lambda_{\gamma}^r}{\lambda_{\gamma}^r}	|\langle \chi^{\gamma'}_{i'j'} \circ \alpha_g, \chi^{\gamma}_{ij} \rangle| \\
		&= \frac{1}{\lambda_{\gamma}^r} |\langle \chi^{\gamma'}_{i'j'} \circ \alpha_g, \lambda_K(\triangle)^r\chi^{\gamma}_{ij}\rangle| \\
		&\leq \frac{1}{\lambda_{\gamma}^r} \sum_{i_1, \ldots , i_r} |\langle \lambda(-X_{i_1})\cdots \lambda(-X_{i_r}) \chi_{i'j'}^{\gamma'} \circ \alpha_g , \lambda(X_{i_1})\cdots \lambda(X_{i_r})\chi^{\gamma}_{ij} \rangle| \\
		&\ll_r \lambda_{\gamma}^{-r}(1 + \varepsilon)^{O(1)nr} ||\gamma'||^r\, ||\gamma||^r \\ 
		&\ll_r (1 + \varepsilon)^{O(1)nr} ||\gamma'||^r\, ||\gamma||^{-r}
	\end{align*} where for the penultimate line one argues as in Lemma~\ref{DirectionalDerivativeEstimate} and in the last line we use Lemma~\ref{DimensionLaplacianBound}. Similarly, it holds that $|\langle \chi^{\gamma'}_{i'j'} \circ \alpha_g, \chi^{\gamma}_{ij} \rangle| \ll_r (1 + \varepsilon)^{O(1)nr} ||\gamma'||^{-r}\, ||\gamma||^{r}$.
	Then using the decomposition $$\varphi = \sum_{2^{\ell-1}\leq ||\gamma'|| < 2^{\ell}}\sum_{i',j' = 1}^{d_{\gamma}} d_{\gamma'}^{1/2}a^{\gamma'}_{i'j'}\chi_{i'j'}^{\gamma'}$$ we conclude  
	\begin{align*}
		|\langle \varphi \circ \alpha_g, \chi^{\gamma}_{ij}\rangle| &\leq \sum_{\gamma',i',j'} d_{\gamma'}^{1/2} |a_{i'j'}^{\gamma'}|\, |\langle \chi^{\gamma'}_{i'j'} \circ \alpha_g, \chi^{\gamma}_{ij} \rangle| \\
		&\leq 2^{O(1)\ell}||\varphi||_2 \sup_{\gamma',i',j'} |\langle \chi^{\gamma'}_{i'j'} \circ \alpha_g, \chi^{\gamma}_{ij} \rangle| \\
		&\ll_r 2^{O(1)\ell - r\ell}(1 + \varepsilon)^{O(1)nr} ||\gamma||^r ||\varphi||_2.
	\end{align*} This implies the claim by \eqref{Phi_nBound} and using Lemma~\ref{DimensionLaplacianBound}.
\end{proof}

\begin{proof}(of Theorem~\ref{FurstenbergMeasureSmoothness}) 		
	Let $\varphi \in V_{\ell}$ be of unit norm and write $\Phi_n = T_0^n \varphi$. It suffices to prove for $\varepsilon < \varepsilon_m$ and some some $n \geq 1$ that
	\begin{equation}\label{AimDecayBound}
		||\Phi_n||_{\infty} \leq 2^{-(s + 1)\ell},
	\end{equation}
	where $s$ is a constant depending on $G$ and $m$. Indeed, if \eqref{AimDecayBound} holds, then by \eqref{ReductionPhi_nEstimate}, $$||P_{\ell} \psi_F||_2 \ll 2^{O(1)\ell}2^{-(s + 1)\ell},$$ which satisfies the smoothness condition from Lemma~\ref{SmoothnessCharacterization} for $s$ large enough depending on $G$ and $m$.
	
	We will use Agmon's inequality to prove \eqref{AimDecayBound}. Notice first that for the fixed $t \in \Z_{\geq 2}$ from Lemma~\ref{AgmonsInequality},
	\begin{align*}
		||\Phi_n||_{H^t}  &= ||\lambda_K(\triangle)^{t/2} \Phi_n||_2 \\ &\leq \sup_{g \in \mathrm{supp}\mu^{*n}} ||\lambda_K(\triangle)^{t/2}(\varphi \circ \alpha_g)||_{\infty} \\ &\leq ||\lambda_K(\triangle)^{t/2} \varphi||_{\infty}(1 + \varepsilon)^{O(1) n} \leq 2^{A\ell},
	\end{align*} for an absolte constant $A$ and where we choose $n = \frac{1}{10E_2\varepsilon}\ell$ for $E_2$ a fixed constant to be determined later. 
	
	We next bound $||\Phi_n||_2$. In order to do so, we decompose $\Phi_n$ into a low and high frequency part:
	\begin{align*}
		\Phi_n = \Phi_n^{(1)} + \Phi_n^{(2)} \quad\text{where}\quad \Phi_n^{(1)} = \sum_{||\gamma|| \leq L(c_1,c_1)} \sum_{i,j}^{d_{\gamma}} d_{\gamma}^{1/2}\widehat{(\Phi_n)}^{\gamma}_{ij} \chi_{ij}^{\gamma}.
	\end{align*}

	Then for $n \geq 1$, exploiting Corollary~\ref{BourgainT_0MainResult}
	\begin{align}
		||\Phi_{n}||_2 &\leq \bigg|\bigg| \int \Phi_{n-1}^{(1)} \circ \alpha_g \, d\mu(g) \bigg|\bigg|_{\infty} + \bigg|\bigg| \int \Phi_{n-1}^{(2)} \circ \alpha_g \, d\mu(g) \bigg|\bigg|_{2} \nonumber  \\
		&\leq ||\Phi_{n-1}^{(1)}||_{\infty} + \frac{1}{2} ||\Phi_{n-1}^{(2)}||_2 \label{Phi_nIteration}
	\end{align} Using Lemma~\ref{Phi_nFourierEstimate}, it follows for all $m \leq n$ and $r \geq 1$, $$||\Phi_{m}^{(1)}||_{\infty} \ll_r 2^{O(1)\ell - r\ell}(1 + \varepsilon)^{O(1)nr}L(c_1,c_2)^{O(1)r}||\varphi||_2.$$ Iterating \eqref{Phi_nIteration}, there are absolute constants $E_1,E_2,E_3 \geq 1$ such that $$||\Phi_n||_2 \ll_r (n2^{E_1\ell - r\ell}(1 + \varepsilon)^{E_2nr}L(c_1,c_2)^{E_3r} + 2^{-n})||\varphi||_2.$$ By Lemma~\ref{AgmonsInequality}, it therefore follows that $$||\Phi_n||_{\infty} \ll_r (n2^{(E_1 + A)\ell  - r\ell}(1 + \varepsilon)^{E_2nr}L(c_1,c_2)^{E_3r} + 2^{-n})||\varphi||_2.$$
	
	Setting the parameters suitably, the proof is concluded. Indeed, choose for instance $$r = 2(s + 1) + E_1 + A  +100$$ and $n = \frac{1}{10E_2\varepsilon}\ell$. For $s$ large enough and choosing $\varepsilon$ small enough in terms of $r$ and $s$ the claim \eqref{AimDecayBound} holds for large $\ell$ (depending on $s$ and $\varepsilon$). 
\end{proof}

\bibliography{referencesgeneral.bib}

\begin{bibdiv}
\begin{biblist}

\bib{AgmonEllipticBook}{book}{
      author={Agmon, Shmuel},
       title={Lectures on elliptic boundary value problems},
      series={Van Nostrand Mathematical Studies, No. 2},
   publisher={D. Van Nostrand Co., Inc., Princeton, N.J.-Toronto-London},
        date={1965},
        note={Prepared for publication by B. Frank Jones, Jr. with the
  assistance of George W. Batten, Jr.},
}

\bib{Andersen2004}{article}{
      author={Andersen, Nils},
       title={Real {P}aley-{W}iener theorems for the inverse {F}ourier
  transform on a {R}iemannian symmetric space},
        date={2004},
     journal={Pacific Journal of Mathematics},
      volume={213},
       pages={1\ndash 13},
}

\bib{AubinSomeBook}{book}{
      author={Aubin, Thierry},
       title={Some nonlinear problems in {R}iemannian geometry},
      series={Springer Monographs in Mathematics},
   publisher={Springer-Verlag, Berlin},
        date={1998},
        ISBN={3-540-60752-8},
}

\bib{BekkaDeLaHarpeValetteKazhdanBook}{book}{
      author={Bekka, Bachir},
      author={de~la Harpe, Pierre},
      author={Valette, Alain},
       title={Kazhdan's property ({T})},
      series={New Mathematical Monographs},
   publisher={Cambridge University Press, Cambridge},
        date={2008},
      volume={11},
        ISBN={978-0-521-88720-5},
}

\bib{BenoistdeSaxce2016}{article}{
      author={Benoist, Yves},
      author={de~Saxc\'{e}, Nicolas},
       title={A spectral gap theorem in simple {L}ie groups},
        date={2016},
     journal={Invent. Math.},
      volume={202},
      number={2},
       pages={337\ndash 361},
}

\bib{BourgainGamburd2008Invent}{article}{
      author={Bourgain, Jean},
      author={Gamburd, Alex},
       title={On the spectral gap for finitely-generated subgroups of {$\rm
  SU(2)$}},
        date={2008},
     journal={Invent. Math.},
      volume={171},
      number={1},
       pages={83\ndash 121},
}

\bib{BourgainGamburd2010}{article}{
      author={Bourgain, J.},
      author={Gamburd, A.},
       title={A spectral gap theorem in {${\rm SU}(d)$}},
        date={2012},
     journal={J. Eur. Math. Soc. (JEMS)},
      volume={14},
      number={5},
       pages={1455\ndash 1511},
}

\bib{BoutonnetIoanaSalehiGolsefidy2017}{article}{
      author={Boutonnet, R\'{e}mi},
      author={Ioana, Adrian},
      author={Salehi~Golsefidy, Alireza},
       title={Local spectral gap in simple {L}ie groups and applications},
        date={2017},
     journal={Invent. Math.},
      volume={208},
      number={3},
       pages={715\ndash 802},
}

\bib{Bourgain2012}{incollection}{
      author={Bourgain, Jean},
       title={Finitely supported measures on {$SL_2(\Bbb R)$} which are
  absolutely continuous at infinity},
        date={2012},
   booktitle={Geometric aspects of functional analysis},
      series={Lecture Notes in Math.},
      volume={2050},
   publisher={Springer, Heidelberg},
       pages={133\ndash 141},
}

\bib{Bougerol1981}{article}{
      author={Bougerol, Philippe},
       title={Th\'{e}or\`eme central limite local sur certains groupes de
  {L}ie},
        date={1981},
     journal={Ann. Sci. \'{E}cole Norm. Sup. (4)},
      volume={14},
      number={4},
       pages={403\ndash 432 (1982)},
}

\bib{BaranyPollicottSimon2012}{article}{
      author={B\'{a}r\'{a}ny, B.},
      author={Pollicott, M.},
      author={Simon, K.},
       title={Stationary measures for projective transformations: the
  {B}lackwell and {F}urstenberg measures},
        date={2012},
     journal={J. Stat. Phys.},
      volume={148},
      number={3},
       pages={393\ndash 421},
}

\bib{BenoistQuintRandomBook}{book}{
      author={Benoist, Yves},
      author={Quint, Jean-Fran\c{c}ois},
       title={Random walks on reductive groups},
      series={Ergebnisse der Mathematik und ihrer Grenzgebiete. 3. Folge. A
  Series of Modern Surveys in Mathematics [Results in Mathematics and Related
  Areas. 3rd Series. A Series of Modern Surveys in Mathematics]},
   publisher={Springer, Cham},
        date={2016},
      volume={62},
}

\bib{BenoistQuint2018}{article}{
      author={Benoist, Yves},
      author={Quint, Jean-Fran\c{c}ois},
       title={On the regularity of stationary measures},
        date={2018},
     journal={Israel J. Math.},
      volume={226},
      number={1},
       pages={1\ndash 14},
}

\bib{Breuillard2005ProbRel}{article}{
      author={Breuillard, Emmanuel},
       title={Distributions diophantiennes et th\'{e}or\`eme limite local sur
  {$\Bbb R^d$}},
        date={2005},
     journal={Probab. Theory Related Fields},
      volume={132},
      number={1},
       pages={39\ndash 73},
}

\bib{Breuillard2005GAFA}{article}{
      author={Breuillard, Emmanuel},
       title={Local limit theorems and equidistribution of random walks on the
  heisenberg group},
        date={2005},
     journal={Geom. Funct. Anal.},
      volume={25},
       pages={35\ndash 82},
}

\bib{BreimanProbabilityBook}{book}{
      author={Breiman, Leo},
       title={Probability},
      series={Classics in Applied Mathematics},
   publisher={Society for Industrial and Applied Mathematics (SIAM),
  Philadelphia, PA},
        date={1992},
      volume={7},
        ISBN={0-89871-296-3},
        note={Corrected reprint of the 1968 original},
}

\bib{BrocknertomDieckRepresentationBook}{book}{
      author={Br\"{o}cker, Theodor},
      author={tom Dieck, Tammo},
       title={Representations of compact {L}ie groups},
      series={Graduate Texts in Mathematics},
   publisher={Springer-Verlag, New York},
        date={1985},
      volume={98},
        ISBN={0-387-13678-9},
}

\bib{ConzeGuivarch2013}{article}{
      author={Conze, J.-P.},
      author={Guivarc'h, Y.},
       title={Ergodicity of group actions and spectral gap, applications to
  random walks and {M}arkov shifts},
        date={2013},
     journal={Discrete Contin. Dyn. Syst.},
      volume={33},
      number={9},
       pages={4239\ndash 4269},
}

\bib{DiaconisHough2021}{article}{
      author={Diaconis, Persi},
      author={Hough, Robert},
       title={Random walk on unipotent matrix groups},
        date={2021},
     journal={Ann. Sci. \'{E}c. Norm. Sup\'{e}r. (4)},
      volume={54},
      number={3},
       pages={587\ndash 625},
}

\bib{deSaxce2013}{article}{
      author={de~Saxc\'{e}, Nicolas},
       title={Trou dimensionnel dans les groupes de {L}ie compacts semisimples
  via les s\'{e}ries de {F}ourier},
        date={2013},
     journal={J. Anal. Math.},
      volume={120},
       pages={311\ndash 331},
}

\bib{DunfordSchwartzLinearBook}{book}{
      author={Dunford, Nelson},
      author={Schwartz, Jacob~T.},
       title={Linear {O}perators. {I}. {G}eneral {T}heory},
      series={Pure and Applied Mathematics, Vol. 7},
   publisher={Interscience Publishers, Inc., New York; Interscience Publishers,
  Ltd., London},
        date={1958},
        note={With the assistance of W. G. Bade and R. G. Bartle},
}

\bib{EinsiedlerWardFunctionalBook}{book}{
      author={Einsiedler, Manfred},
      author={Ward, Thomas},
       title={Functional analysis, spectral theory, and applications},
      series={Graduate Texts in Mathematics},
   publisher={Springer, Cham},
        date={2017},
      volume={276},
}

\bib{FinisMatz2021}{article}{
      author={Finis, Tobias},
      author={Matz, Jasmin},
       title={On the asymptotics of {H}ecke operators for reductive groups},
        date={2021},
     journal={Math. Ann.},
      volume={380},
      number={3-4},
       pages={1037\ndash 1104},
}

\bib{GoldsheidMargulis1989}{article}{
      author={Gol\cprime~dshe\u{\i}d, I.~Ya.},
      author={Margulis, G.~A.},
       title={Lyapunov exponents of a product of random matrices},
        date={1989},
     journal={Uspekhi Mat. Nauk},
      volume={44},
      number={5(269)},
       pages={13\ndash 60},
}

\bib{Gouezel2014}{article}{
      author={Gou\"ezel, S\'{e}bastien},
       title={Local limit theorem for symmetric random walks in
  gromov-hyperbolic groups},
        date={2014},
     journal={J. Amer. Math. Soc},
      volume={27},
      number={4},
       pages={893\ndash 928},
}

\bib{Guivarch1980}{incollection}{
      author={Guivarc'h, Y.},
       title={Sur la loi des grands nombres et le rayon spectral d'une marche
  al\'{e}atoire},
        date={1980},
   booktitle={Conference on {R}andom {W}alks ({K}leebach, 1979) ({F}rench)},
      series={Ast\'{e}risque},
      volume={74},
   publisher={Soc. Math. France, Paris},
       pages={47\ndash 98, 3},
}

\bib{HallLieBook}{book}{
      author={Hall, Brian~C.},
       title={Lie groups, {L}ie algebras, and representations},
     edition={Second},
      series={Graduate Texts in Mathematics},
   publisher={Springer-Verlag, Cham},
        date={2015},
      volume={222},
        ISBN={978-3-319-13466-6},
        note={An elementary introduction},
}

\bib{Harish-Chandra1958}{article}{
      author={Harish-Chandra},
       title={Spherical functions on a semisimple {L}ie group. {II}},
        date={1958},
     journal={Amer. J. Math.},
      volume={80},
       pages={553\ndash 613},
}

\bib{HelgasonDifferentialBook}{book}{
      author={Helgason, Sigurdur},
       title={Differential geometry, {L}ie groups, and symmetric spaces},
      series={Pure and Applied Mathematics},
   publisher={Academic Press, Inc. [Harcourt Brace Jovanovich, Publishers], New
  York-London},
        date={1978},
      volume={80},
        ISBN={0-12-338460-5},
}

\bib{HelgasonGroupsBook}{book}{
      author={Helgason, Sigurdur},
       title={Groups and geometric analysis},
      series={Pure and Applied Mathematics},
   publisher={Academic Press, Inc., Orlando, FL},
        date={1984},
      volume={113},
        ISBN={0-12-338301-3},
        note={Integral geometry, invariant differential operators, and
  spherical functions},
}

\bib{ItoKawada1940}{article}{
      author={Ito, Kiyoshi},
      author={Kawada, Yukiyosi},
       title={On the probability distribution on a compact group},
        date={1940},
     journal={Proc. Phys.-Math. Soc. Japan},
      volume={22},
       pages={977\ndash 998},
}

\bib{KatoPertubationBook}{book}{
      author={Kato, Tosio},
       title={Perturbation theory for linear operators},
      series={Classics in Mathematics},
   publisher={Springer-Verlag, Berlin},
        date={1995},
        ISBN={3-540-58661-X},
        note={Reprint of the 1980 edition},
}

\bib{KaimanovichLePrince2011}{article}{
      author={Kaimanovich, Vadim~A.},
      author={Le~Prince, Vincent},
       title={Matrix random products with singular harmonic measure},
        date={2011},
     journal={Geom. Dedicata},
      volume={150},
       pages={257\ndash 279},
}

\bib{KnappLieBook}{book}{
      author={Knapp, Anthony~W.},
       title={Lie groups beyond an introduction},
     edition={Second},
      series={Progress in Mathematics},
   publisher={Birkh\"{a}user Boston, Inc., Boston, MA},
        date={2002},
      volume={140},
        ISBN={0-8176-4259-5},
}

\bib{Lalley1993}{article}{
      author={Lalley, Steven},
       title={Finite range random walk on free groups and homogeneous trees},
        date={1993},
     journal={Ann. Probab.},
      volume={21},
      number={4},
       pages={2087\ndash 2130},
}

\bib{Lequen2022}{article}{
      author={Lequen, F\'{e}lix},
       title={Absolutely continuous {F}urstenberg measures for
  finitely-supported random walks},
        date={2022},
        note={\url{https://arxiv.org/abs/2205.11138.pdf}},
}

\bib{LindenstraussVarju2016}{article}{
      author={Lindenstrauss, Elon},
      author={Varj\'{u}, P\'{e}ter~P.},
       title={Random walks in the group of {E}uclidean isometries and
  self-similar measures},
        date={2016},
     journal={Duke Math. J.},
      volume={165},
      number={6},
       pages={1061\ndash 1127},
}

\bib{Mallahi-KaraiMohammdiSalehiGolsefidy2022}{article}{
      author={Mallahi-Karai, Keivan},
      author={Mohammadi, Amir},
      author={Salehi~Golsefidy, Alireza},
       title={Locally random groups},
        date={2022},
     journal={Michigan Math. J.},
      volume={72},
       pages={479\ndash 527},
}

\bib{Newburgh1951}{article}{
      author={Newburgh, J.~D.},
       title={The variation of spectra},
        date={1951},
     journal={Duke Math. J.},
      volume={18},
       pages={165\ndash 176},
}

\bib{RemlingMathOverflow}{misc}{
      author={Remling, Christian},
       title={Uniform decay of operator norm for smooth family of operators},
         how={MathOverflow},
         url={https://mathoverflow.net/q/434679},
        note={URL:https://mathoverflow.net/q/434679 (version: 2022-11-15)},
}

\bib{SchaeferBanachBook}{book}{
      author={Schaefer, Helmut~H.},
       title={Banach lattices and positive operators},
      series={Die Grundlehren der mathematischen Wissenschaften, Band 215},
   publisher={Springer-Verlag, New York-Heidelberg},
        date={1974},
}

\bib{Stone1965}{article}{
      author={Stone, C.},
       title={A local limit theorem for nonlattice multi-dimensional
  distribution functions},
        date={1965},
     journal={Ann. Math. Stat.},
      volume={36},
       pages={546\ndash 551},
}

\bib{Tolli2000}{article}{
      author={Tolli, Filippo},
       title={A {B}erry-{E}sseen theorem on semisimple {L}ie groups},
        date={2000},
     journal={Ann. Inst. H. Poincar\'{e} Probab. Statist.},
      volume={36},
      number={3},
       pages={275\ndash 290},
}

\bib{Varju2012}{article}{
      author={Varj\'{u}, P\'{e}ter~P\'{a}l},
       title={Random walks in euclidean space},
        date={2015},
     journal={Ann. of Math. (2)},
      volume={181},
       pages={243\ndash 301},
}

\bib{WallachReductive1Book}{book}{
      author={Wallach, Nolan~R.},
       title={Real reductive groups. {I}},
      series={Pure and Applied Mathematics},
   publisher={Academic Press, Inc., Boston, MA},
        date={1988},
      volume={132},
        ISBN={0-12-732960-9},
}

\bib{WarnerHarmonic1Book}{book}{
      author={Warner, Garth},
       title={Harmonic analysis on semi-simple {L}ie groups. {I}},
      series={Die Grundlehren der mathematischen Wissenschaften, Band 188},
   publisher={Springer-Verlag, New York-Heidelberg},
        date={1972},
}

\end{biblist}
\end{bibdiv}
\end{document}